\documentclass[10pt]{amsart}
\usepackage{graphicx}
\usepackage[margin=1.305in,tmargin=1.305in,bmargin=1.1in,headsep=.75cm]{geometry}
\usepackage[normalem]{ulem}

\usepackage[utf8]{inputenc}
\usepackage[dvipsnames]{xcolor}
\usepackage{amsfonts, amstext, amsmath, amsthm, amscd, amssymb}
\usepackage{mathrsfs}
\usepackage{mathtools}
\usepackage{tikz}
\usepackage{url}
\usepackage{multirow,arydshln}
\usepackage[shortlabels]{enumitem}

\usepackage{wrapfig}
\usepackage[font=small]{caption}
\usepackage{subcaption}
\usepackage{hyperref}

\definecolor{newblue}{rgb}{0.27, 0.45, 0.8}

\hypersetup{colorlinks=true,linkcolor=newblue,citecolor=newblue, urlcolor=black}

\usepackage{enumitem}

\renewcommand{\setminus}{{\smallsetminus}}

\newcommand{\be}{\begin{equation}}
\newcommand{\ee}{\end{equation}}
\newcommand{\ol}[1]{\overline{#1}}

\newcommand{\alt}[1]{#1^{\star}}

\numberwithin{equation}{section}

\theoremstyle{plain}
\newtheorem{theorem}[equation]{Theorem}
\newtheorem{lemma}[equation]{Lemma}
\newtheorem{claim}{Claim}
\newtheorem{proposition}[equation]{Proposition}

\newtheorem{corollary}[equation]{Corollary}

\newtheorem*{nkbrunnian}{Theorem~\ref{thm:intro-nk-brunnian}}
\newtheorem*{intro-any-no-components}{Theorem~\ref{thm:intro-any-no-components}}
\newtheorem*{thm:intro-surface-links}{Theorem~\ref{thm:intro-surface-links}}

\newtheorem{lem}[equation]{Lemma}
\newtheorem{prop}[equation]{Proposition}

\theoremstyle{definition}

\newtheorem{definition}[equation]{Definition}
\newtheorem*{definition*}{Definition}

\newtheorem{procedure}[equation]{Procedure}

\theoremstyle{remark}
\newtheorem{remark}[equation]{Remark}

\def\N{\mathbb N}
\def\Z{\mathbb Z}
\def\R{\mathbb R}

\def\F{\mathbb{F}}
\def\wt#1{\widetilde{#1}}

\def\sm{\setminus}
\def\S{\Sigma}

\def\wt{\widetilde}

\def\bp{\begin{pmatrix}}
\def\ep{\end{pmatrix}}
\def\ba{\begin{array}}
\def\ea{\end{array}}
\def\bn{\begin{enumerate}}
\def\en{\end{enumerate}}

\DeclareMathOperator\Wh{Wh}
\DeclareMathOperator\BD{BD}

\DeclareMathOperator\Id{Id}

\DeclareMathOperator\coker{coker}

\DeclareMathOperator\CAT{CAT}
\DeclareMathOperator\DIFF{DIFF}
\DeclareMathOperator\TOP{TOP}
\DeclareMathOperator\HFK{{\emph{HFK}}}
\DeclareMathOperator\HFL{{\emph{HFL}}}

\DeclareMathOperator\DLink{{DLink}}
\DeclareMathOperator\Vect{{Vect}}
\DeclareMathOperator\spin{Spin}

\newcommand{\red}{\textcolor{red}}
\newcommand{\blue}{\textcolor{blue}}

\usepackage{pinlabel}
\definecolor{darkpurple}{rgb}{.5,0,.5}
\definecolor{darkteal}{rgb}{0,.5,.5}
\definecolor{darkgreen}{rgb}{.15,.5,.35}

\begin{document}

\title{Brunnian exotic surface links in the 4--ball}

\author[Hayden]{Kyle Hayden}
\address{Department of Mathematics, Rutgers University-Newark, USA}
\email{kyle.hayden@rutgers.edu}

\author[Kjuchukova]{Alexandra Kjuchukova}
\address{Department of Mathematics, University of Notre Dame, USA}
\email{akjuchuk@nd.edu}

\author[Krishna]{Siddhi Krishna}
\address{Department of Mathematics, Columbia University, USA}
\email{siddhi@math.columbia.edu}

\author[Miller]{\\ Maggie Miller}
\address{Department of Mathematics, University of Texas at Austin, USA}
\email{maggie.miller.math@gmail.com}

\author[Powell]{Mark Powell}
\address{School of Mathematics and Statistics, University of Glasgow, United Kingdom}
\email{mark.powell@glasgow.ac.uk}

\author[Sunukjian]{Nathan Sunukjian}
\address{Department of Mathematics and Statistics, Calvin University, USA}
\email{nss9@calvin.edu}



\def\subjclassname{\textup{2020} Mathematics Subject Classification}
\expandafter\let\csname subjclassname@1991\endcsname=\subjclassname

\thanks{
KH was partially supported by NSF grants DMS-1803584 and DMS-2114837. SK was partially supported by NSF grant DMS-1745583. MM was supported by NSF grant DMS-2001675 and a fellowship from the Clay Mathematics Institute. MP was partially supported by EPSRC New Investigator grant EP/T028335/2 and EPSRC New Horizons grant EP/V04821X/2. }

\begin{abstract}
This paper investigates exotic phenomena exhibited by links of disconnected surfaces with boundary 
  properly embedded in the 4-ball. Our main results provide two different constructions of exotic pairs of surface links which are Brunnian, meaning that all their proper sublinks 
  are trivial. Furthermore, we modify these core constructions to vary the number of components in the exotic links, the genera of the components, and the number of components that must be removed before the surfaces become unlinked. 
\end{abstract}

\maketitle

\section{Introduction} \label{section:Introduction}

We construct exotic pairs of surface links in the 4-ball where every sublink is smoothly trivial. A \emph{surface link} is a smooth,  oriented, properly embedded, 2--dimensional submanifold $\Sigma \subseteq B^4$ whose connected components $\Sigma_i$ each have exactly one boundary component, which lies in $S^3$. When~$\Sigma$ is connected, it is called a \emph{surface knot}. If each component $\Sigma_i$ is a disk, we say that $\Sigma$ is a \emph{disk link} (or \emph{disk knot}, in the connected case).
We will say that two surface links $\Sigma$ and $\Sigma'$ with $\partial \Sigma = \partial \Sigma'$  form an \emph{exotic pair} if they are topologically ambiently isotopic rel.~boundary, but there is no ambient diffeomorphism of $W$ carrying $\Sigma$ to $\Sigma'$.

A fundamental open question in the study of surfaces in 4-manifolds asks whether there exists an orientable, exotic pair of surfaces in $S^4$. 
Exotic pairs of orientable surfaces have been constructed in many other 4--manifolds $X$ with $b_2(X)>0$; see e.g.\ the foundational paper by Fintushel--Stern \cite{fintushelstern}. In a related vein, $S^4$ is known to contain exotic nonorientable surfaces \cite{KreckNonorient}. However, these results rely on the topology of the ambient $4$--manifold or nonorientability of the surfaces in an essential way, such as to support the use of gauge-theoretic invariants.

%
Juh\'asz--Miller--Zemke  \cite{JuhaszMillerZemke} and Hayden \cite{Hayden} produced exotic, orientable, connected surface knots. This paper studies surface links, and exhibits increasingly subtle forms of exotic behavior.  
Specifically, we uncover exotic behavior among a particularly delicate class of surface links which are analogous to Brunnian links in $S^3$~\cite{Brunn}.

 
 
 
 \begin{definition*}\leavevmode
 \begin{enumerate}[label=\alph*)]
\item A surface link $\Sigma \subseteq B^4$ is \emph{unlinked} if $\partial \Sigma$ is an unlink in $S^3$ and $\Sigma$ is smoothly  isotopic to a Seifert surface for the unlink in $S^3$.
\item A surface link $\Sigma \subseteq B^4$ is \emph{Brunnian} if removing any component from $\Sigma$ yields an  unlinked surface link. 
\end{enumerate}
 \end{definition*}

Our notion of Brunnian differs from the usual convention in 3-dimensions, where a Brunnian link is required to be non-split. According to our definition, an unlink is Brunnian.





\begin{theorem}\label{thm:intro-any-no-components}
Let $n \geq 2$ be an integer. 
\begin{enumerate}[$(i)$]
  \item\label{item:thm-intro-any-no-comps-i} There exists an exotic pair of $n$--component Brunnian disk links  in $B^4$.
  \item\label{item:thm-intro-any-no-comps-ii} There exists an infinite family of pairwise exotic $n$--component Brunnian surface links in $B^4$ that each consist of a single genus-one surface and $n-1$ disks.  
\end{enumerate}
\end{theorem}



We prove Theorem \ref{thm:intro-any-no-components} by induction on $n$, with $n=2$ the base case. Our constructions utilize the {\emph{Bing doubling}} operation on a disk knot, which we define precisely in Section \ref{background-bing-doubling}. For now, it suffices to know that the Bing doubling operation replaces one disk component of a surface link with two disjoint disks in a neighborhood of the original, as illustrated in Figure~\ref{fig:bing_doubling_disks}.




\begin{figure}
\centering
\begin{subfigure}[b]{0.4\textwidth}\centering
	\includegraphics[scale=.4]{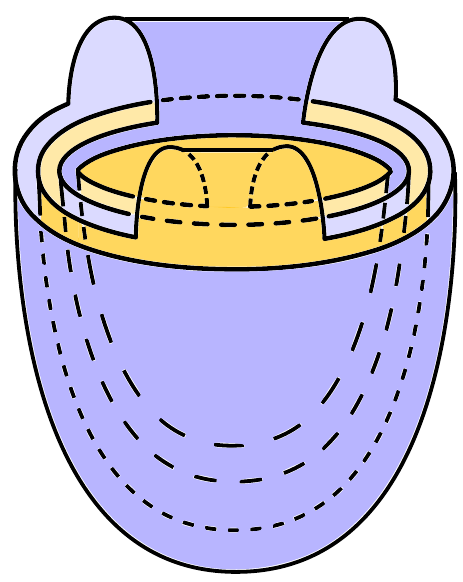}
\caption{}
\label{fig:bing_doubling_disks}
\end{subfigure}
\quad
\begin{subfigure}[b]{0.55\linewidth}
\centering
\def\svgwidth{.75\linewidth}
\begingroup%
  \makeatletter%
  \providecommand\color[2][]{%
    \errmessage{(Inkscape) Color is used for the text in Inkscape, but the package 'color.sty' is not loaded}%
    \renewcommand\color[2][]{}%
  }%
  \providecommand\transparent[1]{%
    \errmessage{(Inkscape) Transparency is used (non-zero) for the text in Inkscape, but the package 'transparent.sty' is not loaded}%
    \renewcommand\transparent[1]{}%
  }%
  \providecommand\rotatebox[2]{#2}%
  \newcommand*\fsize{\dimexpr\f@size pt\relax}%
  \newcommand*\lineheight[1]{\fontsize{\fsize}{#1\fsize}\selectfont}%
  \ifx\svgwidth\undefined%
    \setlength{\unitlength}{167.94595433bp}%
    \ifx\svgscale\undefined%
      \relax%
    \else%
      \setlength{\unitlength}{\unitlength * \real{\svgscale}}%
    \fi%
  \else%
    \setlength{\unitlength}{\svgwidth}%
  \fi%
  \global\let\svgwidth\undefined%
  \global\let\svgscale\undefined%
  \makeatother%
  \begin{picture}(1,0.5972774)%
    \lineheight{1}%
    \setlength\tabcolsep{0pt}%
    \put(0,0){\includegraphics[width=\unitlength,page=1]{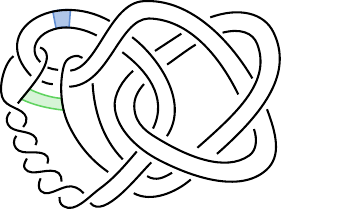}}%
    \put(0.8306165,0.45563877){\color[rgb]{0,0,0}\makebox(0,0)[lt]{\lineheight{1.25}\smash{\begin{tabular}[t]{l}$K$\end{tabular}}}}%
  \end{picture}%
\endgroup%

\caption{}\label{fig:slice-knot}
\end{subfigure}
\caption{(a) The Bing double of a disk consists of two disjoint disks. (b) The knot $K$, together with two bands, leading to the different slice disks $D_1$ and $D_2$, whose  Bing doubles appear in Theorem~\ref{thm:intro-disk-links}.}
\end{figure}


\begin{theorem}\label{thm:intro-disk-links}
The knot $K$ in Figure~\ref{fig:slice-knot} bounds two different slice disks $D_1$ and $D_2$ in $B^4$ $($pictured in Figure~\ref{fig:initial-disks}$)$. Their Bing doubles  $\BD(D_1)$ and $\BD(D_2)$ form an exotic pair of 2-component Brunnian disk links in $B^4$.
\end{theorem}


We briefly outline the strategy which distinguishes the diffeomorphism classes of the disk links in Theorem \ref{thm:intro-disk-links}. The disk links $\BD(D_1)$ and $\BD(D_2)$ have the same boundary in $S^3$, a link $L$. We identify a knot $\gamma$ in the exterior of $L$ such that any diffeomorphism of pairs sending $(B^4, \BD(D_1))$ to $(B^4, \BD(D_2))$ must preserve the isotopy class of $\gamma$. By showing that $\gamma$ is slice in the complement $B^4 \sm \BD(D_2)$ but not in $B^4 \sm \BD(D_1)$, we deduce that the disk link exteriors are not diffeomorphic. 

Next, we briefly sketch our proof of the $n=2$ case of  Theorem~\ref{thm:intro-any-no-components}~\ref{item:thm-intro-any-no-comps-ii}. 
Our construction of an infinite family of pairwise exotic 2-component Brunnian surface links relies on \textit{rim surgery}, a technique  first introduced in \cite{fintushelstern}. 
It takes as input an essential curve $\alpha$ on a surface $\Sigma$ and a knot $J$, and outputs a new surface  $\Sigma(\alpha;J)$ (see Section~\ref{subsection:rim-surgery}).

\begin{theorem}\label{thm:intro-surface-links}
Let $K$ be a strongly quasipositive topologically slice knot and let $\Sigma$ denote the surface link in Figure
~\ref{fig:infinitefamily}. 
For each integer $m\geq 0$, let $J_m \subseteq S^3$ be a knot whose mod 2 Alexander polynomial $\Delta_{J_m}(t) \in \mathbb{F}_2[t^{\pm 1}]$ has $m$ irreducible factors $($counted with multiplicity$)$. 
The 2--component Brunnian surface links 
$\{\Sigma(\alpha;J_m)\}_{m=0}^{\infty}$, each
consisting of a disk and a genus one surface, are pairwise exotic. 
\end{theorem}

\begin{figure}
\labellist
  \pinlabel $S_2$ at 20 20
  \pinlabel {$\textcolor{darkpurple}{S_1}$} at 0 200
  \pinlabel {\textcolor{red}{$\alpha$}} at 290 190
\endlabellist
\includegraphics[width=40mm]{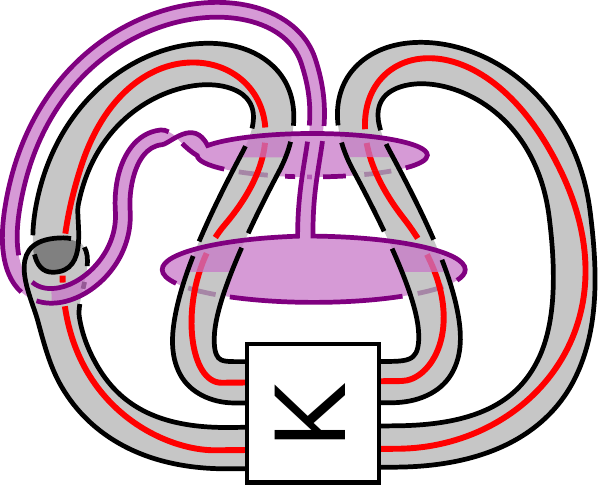}
\caption{The surface link $\S=S_1\sqcup S_2$. At the bottom of the figure, two bands in $S_2$ are knotted through $K$. The diagram of $K$ should be taken to have writhe zero.}\label{fig:infinitefamily}
\end{figure}



The surface links of Theorem~\ref{thm:intro-surface-links} are constructed explicitly in Section~\ref{section:infinitely-many-surface-links}. 
A key ingredient in the construction is a choice of curve $\alpha \subseteq \Sigma$ such that:
\begin{enumerate}[(i)]
    \item $\alpha$ bounds a locally flat disk with interior in the complement of $\Sigma$,
    \item $\alpha$ bounds a smooth disk with interior in the complement of each component $\Sigma_i$, but
    \item the interior of every smooth disk with boundary $\alpha$ intersects~$\Sigma$.
    \end{enumerate}
This ensures that rim surgery along $\alpha$ preserves both the topological isotopy type of the link and the smooth isotopy type of each component, while potentially changing the smooth isotopy type of the link itself. We note that for the last item to hold, $\alpha$ must be essential, so the construction cannot work for disk links. 

To obstruct smooth equivalence of each pair, we use the link Floer cobordism maps induced by each surface link, following the methods of Juh\'asz--Miller--Zemke \cite{JuhaszMillerZemke},
and use link Floer cobordism maps. First, they showed that the cobordism map induced by a quasipositive surface is non-trivial, and then applied 
 a theorem of Juh\'{a}sz--Zemke
\cite{juhaszzemkeconcordancesurgery} to show rim surgery can potentially change the smooth isotopy type of a quasipositive surface. While the surface $\Sigma$ in Theorem~\ref{thm:intro-surface-links} is not quasipositive, we may compose it with another link cobordism and obtain a quasipositive surface. The composition laws for cobordism maps guarantee that the map associated to $\Sigma$ is non-vanishing.

To prove Theorem \ref{thm:intro-any-no-components}, and transition from the $n=2$ cases in Theorems \ref{thm:intro-disk-links} and \ref{thm:intro-surface-links} to $n \geq 3$, we use a technique inspired by work of Cha--Kim \cite{Cha-Kim} on covering links in $S^3$; this is done in Section~\ref{section:covering-maps}.
%
%
%
%
We show that given an exotic pair $\Sigma$ and $\Sigma'$ of Brunnian 2--component surface links constructed as in Theorem~\ref{thm:intro-disk-links} or Theorem~\ref{thm:intro-surface-links}, iteratively Bing doubling one disk component of each of $\Sigma$ and $\Sigma'$ yields an exotic pair of $n$--component surface links. We prove this inductively by investigating the 2-fold covers of $B^4$ branched along a trivial disk component in each of these iterated Bing doubles. 
%




\subsection{Increasing the genera}
We modify the pairs of surface links from Theorem~\ref{thm:intro-surface-links} to increase the genera of the surface components in the 2-component surface link case.

\begin{theorem}\label{thm:increasing-the-genus}
Fix a pair of  integers $r \ge 0$ and $s\ge 1$.  There exists an infinite family of Brunnian 2--component surface links in $B^4$ consisting of a genus $r$ surface and a genus $s$ surface, any two of which form an exotic pair.

In particular, let  $\S^{r,s}_0$ be the surface link constructed in 
Figure~\ref{fig:genusmn}.
For each $n\geq 0$, let $J_n \subseteq S^3$ be a knot whose mod 2 Alexander polynomial has $n$ irreducible factors $($counted with multiplicity$)$. 
Then the 2--component Brunnian surface links  
$\{\Sigma^{r,s}_0(\alpha;J_n)\}$, each consisting of a genus $r$ surface and a genus $s$ surface, are pairwise exotic. 
%
%
%
%
\end{theorem}

The primary constructive technique requires banding the surface from Theorem \ref{thm:intro-surface-links} with another simple Brunnian surface link of controlled genus. An analogous application of the link Floer argument from Theorem \ref{thm:intro-surface-links} obstructs a diffeomorphism of pairs.

\subsection{\texorpdfstring{$\boldsymbol{(n,k)}$}--Brunnian disk links}\label{subsection:n-k-brunnian-disk-links}
We study a natural generalization of Brunnian links.

\begin{definition}
An $n$--component surface link is \textit{$(n,k)$--Brunnian} if every sublink of fewer than $k$ components is an unlink, but every sublink of at least $k$ components is nontrivial.
\end{definition}

The notion of $(n,k)$--Brunnian was introduced by Debrunner~\cite{Debrunner} for links in $S^3$.
Debrunner's original definition requires 
all sublinks of at least $k$ components to be nonsplit; in our adaptation to dimension four, we ask only for these sublinks to be nontrivial. 
By combining his construction of $(n,k)$--Brunnian links with our construction of exotic disk links, we prove the following:

\begin{theorem}\label{thm:intro-nk-brunnian}
For any integers $n$ and $k$ with $n \geq 2$ and $1\le k\le n$, there exists a pair of $(n,k)$--Brunnian disk links in $B^4$ forming an exotic pair.
\end{theorem}

Heuristically, the index $k$ measures the extent to which the link components are entangled. 
At one extreme are $(n,n)$--Brunnian links: these are the nontrivial $n$--component Brunnian links. 
The $(n,1)$--Brunnian links lie at the other extreme: all of their nonempty sublinks are nontrivial. 
Thus Theorem~\ref{thm:intro-nk-brunnian} encapsulates a wide spectrum of exotic behavior. 

The extreme cases of Theorem~\ref{thm:intro-nk-brunnian} follow immediately from the above:
for $k=n$, take the disk links
from  Theorem~\ref{thm:intro-any-no-components}; 
 for $k=1$, consider a split union of $n$ copies of~$D_1$ and the same of~$D_2$, then apply~\cite{Hayden}. 
The intermediate cases require a new construction.




\subsection{Exotic closed surfaces}\label{subsec:closed}
Finally, in Section~\ref{section:closed-surfaces}, we show that exotic surface links in $B^4$ can be used to construct exotic links of closed surfaces in larger 4-manifolds.  This is illustrated using the disk links $\BD(D_1)$ and $\BD(D_2)$ from Theorem~\ref{thm:intro-disk-links}. We first consider the 4--manifold $X$  obtained from $B^4$ by attaching 0-framed 2-handles along the link in $S^3$ bounding the iterated Bing doubles $\BD^k(D_1)$ and $\BD^k(D_2)$. Capping off these disk links with the cores of the 2-handles yields a pair of exotic sphere links in $X$. When $k=1$, we then show that $X$ further embeds into a closed 4--manifold in which these sphere links remain exotic. In each case, we distinguish the exotic sphere links by considering the effect of surgering the ambient 4--manifold along these spheres.

\subsection{Open problems}
Here are some natural further questions that remain open.

\begin{enumerate}
\item  Construct exotic pairs/families of $n$--component Brunnian surface links where the components have arbitrary genera.
\item Construct examples of infinitely many Brunnian disk links that are pairwise exotic.
\item\label{bingquestion} Does there exist a distinct pair of surface links (or knots) in $B^4$ with equivalent Bing doubles? This can be interpreted in both categories, or asked purely for exotic pairs.  
\end{enumerate}

The results of this article could be viewed as evidence towards a negative answer to \eqref{bingquestion}.


\subsection{Organization}

In Section~\ref{background-bing-doubling}, we 
recall the technique of 
Bing doubling disks. Section~\ref{section:exotic-pair-disk-links} constructs the exotic pair of Brunnian 2--component disk links promised in Theorem~\ref{thm:intro-disk-links}.  In Section~\ref{section:infinitely-many-surface-links}, we recall the necessary Heegaard Floer theory and construct the pairwise exotic Brunnian surface links of Theorem~\ref{thm:intro-surface-links}. Section~\ref{section:surfaces-of-higher-genus} extends this construction 
to produce an infinite family of 2-component Brunnian links with arbitrary genera.
Covering surfaces are introduced in Section~\ref{section:covering-maps}; this section also contains the proof of Theorem~\ref{thm:intro-any-no-components}.
Section~\ref{section:nkBrunnian} 
constructs exotic pairs of $(n,k)$-Brunnian disk links. 
Finally, Section~\ref{section:closed-surfaces} produces our examples of 
exotic closed surfaces. 
The appendices detail the computer-assisted calculations used to prove Theorems~\ref{thm:intro-any-no-components} and \ref{thm:intro-disk-links}.

\subsection{Conventions}
We provide some of our conventions.
\begin{itemize}[leftmargin=.8cm]
\item $J$, $K$, and $L$ denote knots or links in $S^3$, while $\Sigma$, $S$, and $D$ denote surfaces in $B^4$, with $D$ always a disk.
\item We work 
in the category $\DIFF$ of smooth manifolds with smooth embeddings and smooth isotopies, and the category $\TOP$ of topological manifolds with locally flat embeddings and topological ambient isotopies. We  specify whether an isotopy or surface is in $\DIFF$ or $\TOP$.
\item Given a smooth or locally flat proper submanifold $Y \subseteq X$, $\nu Y$ will denote an open tubular neighborhood of $Y$.  In the smooth case, the existence of tubular neighborhoods is standard differential topology. In the locally flat case, existence follows from~\cite[Theorem~9.3]{FreedmanQuinn} when $X$ is a 4-manifold.
A closed tubular neighborhood is denoted by $\ol{\nu}Y$.
\end{itemize}

\subsection{Acknowledgements}

This work is the product of a research group formed under the auspices of the American Institute for Mathematics (AIM) in their virtual Research Community on 4-dimensional topology. We are grateful to AIM, and especially to the program organizers Miriam Kuzbary, MM, Juanita Pinz\'on-Caicedo, and Hannah Schwartz. We thank Slava Krushkal for providing valuable feedback on an earlier draft of this paper. SK thanks Jen Hom and John Etnyre for helpful conversations. AK and MP are grateful to the Max Planck Institute for Mathematics in Bonn, where they were visitors while this paper was written. We used KLO \cite{klo}, developed by Frank Swenton, for initial exploration.

\section{Background on Bing doubling}\label{background-bing-doubling}

In this section, we review Bing doubling, the primary tool used throughout to produce surface links with arbitrarily many components. 

We obtain the \emph{Bing double} of a slice disk $D$ in $B^4$ with $\partial D = K$ as follows.  First, define a model Bing double of the core disk $D^2 \times \{0\} \subseteq D^2 \times D^2$ from four parallel copies of the disk, joined with two bands, as depicted in Figure~\ref{fig:BingDouble_Banding}.
The result is two disjoint slice disks in $D^2 \times D^2 \cong B^4$ for the Bing double of $S^1 \times \{0\} \subseteq S^1 \times D^2 \subseteq \partial(D^2 \times D^2) \cong S^3.$
Now, for a slice disk $D$ in $B^4$, choose an orientation-preserving diffeomorphism of a tubular neighborhood $\nu D$ with $D^2 \times D^2$ such that $D$ maps to $D^2 \times \{0\}$ via an orientation-preserving diffeomorphism.  Embed the model Bing double in $B^4$ using the inverse of this identification to yield the pair of disks  $\BD(D) \subseteq \nu D$. Their boundary is the Bing double $\BD(K) \subseteq \nu K$ of the knot~$K$.  Since any two orientation-preserving diffeomorphisms of $D^2 \times D^2$ are isotopic (not necessarily rel.\ boundary), the isotopy class of $\BD(D)$ is independent of the choice of identification $\nu D \cong D^2 \times D^2$.

\begin{figure}[h]\center
\labellist
\pinlabel {$K$} at 210 128
\pinlabel {$K$} at 620 128
\pinlabel {$K$} at 1045 128
\pinlabel {$K$} at 1470 128
\tiny
\pinlabel {attach} at 338 150
\pinlabel {bands} at 338 110
\pinlabel {resolve} at 760 150
\pinlabel {isotopy} at 1180 150
\endlabellist
\includegraphics[scale=.275]{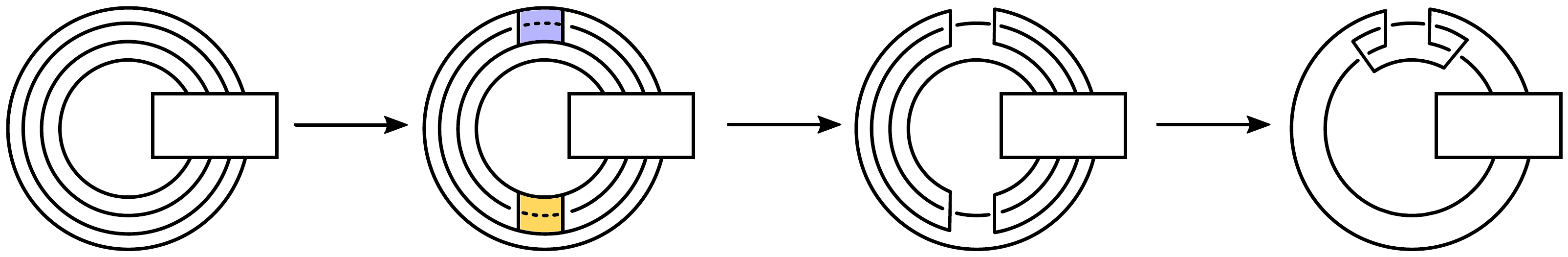}
\caption{Constructing $\BD(D)$, the Bing double of a slice disk $D$, from four copies of $D$ and two bands.}
\label{fig:BingDouble_Banding}
\end{figure}

Going a step further, we define Bing doubling as an operation on ordered surface links.
Let $\Sigma = \Sigma_1 \sqcup \cdots \sqcup \Sigma_k$ be a surface link, with $\S_1 \cong D^2$.  To perform the Bing doubling operation and obtain a new surface link $\BD(\S)$, replace $\Sigma_1$ with $\BD(\S_1)$, and label the two new disks $\BD(\Sigma)_1$ and $\BD(\S)_2$; then label $\BD(\S)_i := \S_{i-1}$ for $3 \leq i \leq k+1$.  This defines $\BD(\S)$ as an ordered surface link obtained from $\S$ via Bing doubling.  For an ordered surface link $\Sigma$, we will also write $\BD(\S)$, where the doubling operation is always performed on the first component. 

Viewing $\BD$ as an operator on ordered surface links, we may perform an iterated Bing doubling. We write $\BD^{k}(\S)$ for the surface link obtained after Bing doubling \emph{the first component} $k$ times. (The ordering of the two components of  $\BD(\S_1)$, assigned arbitrarily at each stage, does not affect the construction.) This operation is integral to the proof of Theorem~\ref{thm:intro-any-no-components}. 


We prove some foundational lemmas about Bing doubling surfaces.

\begin{lemma}\label{lem:doubling-preserves-isotopy}
For $\CAT$ either $\DIFF$ or $\TOP$, let $\Sigma = \Sigma_1 \sqcup \cdots  \sqcup \Sigma_n$ and $\Sigma' = \Sigma'_1 \sqcup \cdots  \sqcup \Sigma'_n$ be surface links in $B^4$ with $\partial \Sigma = \partial \Sigma' \subseteq S^3$. Suppose that $\Sigma $ and $\Sigma'$ are $\CAT$ ambiently isotopic rel.\ boundary as ordered surface links. Assume that $\Sigma_1 \cong D^2 \cong \Sigma_1'$ and that $\partial(\BD(\Sigma_1))= \partial(\BD(\Sigma_1'))$. Then the surface links  $\BD(\Sigma) = \BD(\Sigma_1) \sqcup \Sigma_2 \sqcup \cdots \sqcup \Sigma_n$ and $\BD(\Sigma') = \BD(\Sigma_1') \sqcup \Sigma_2' \sqcup \cdots  \sqcup \Sigma_n'$, obtained by Bing doubling the first component of each surface link, are also $\CAT$ ambiently isotopic rel.\ boundary.
\end{lemma}

\begin{proof}
Let $F_t \colon B^4 \to B^4$, $t \in [0,1]$, be a $\CAT$ ambient isotopy with $F_t|_{S^3} = \Id_{S^3}$ for all $t$, such that $F_0=\Id_{B^4}$ and $F_1(\Sigma) = \Sigma'$.  Let $G \colon D^2 \times D^2 \xrightarrow{\cong} \ol\nu \Sigma_1$ and $G' \colon D^2 \times D^2 \xrightarrow{\cong} \ol\nu \Sigma_1'$ be two identifications, used for the definitions of $\BD(\S_1)$ and $\BD(\S'_1)$ respectively.

Every homeomorphism of $D^2$ is smoothable, and every orientation preserving rel.\ boundary diffeomorphism of $D^2$ is isotopic rel.\ boundary  to the identity by Smale's theorem \cite{Smale-diffeo-D2}. Therefore, by the isotopy extension theorem \cite{EdwardsKirby}, we arrange that the isotopy between $\Sigma$ and $\Sigma'$ sends $\Sigma_1$ to $\Sigma_1'$ respecting fixed choices of parametrizations $D^2 \to \Sigma_1$ and $D^2 \to \Sigma_1'$.
  In other words, extending $F_t$ with a further ambient isotopy supported in a neighborhood of $\Sigma_1$, we arrange that the composition
  \[D^2 \times \{0\}  \xrightarrow{G} \Sigma_1 \xrightarrow{F_1} \Sigma_1'  \xrightarrow{(G')^{-1}} D^2 \times \{0\}\]
is the identity.
Similarly, by uniqueness of normal bundles for 2-dimensional submanifolds of 4-manifolds (which holds for $\CAT=\TOP$ by \cite[Chapter~9]{FreedmanQuinn}), we may assume that $F_1$ sends a given parametrization of  $\ol\nu \Sigma_1$ as $D^2 \times D^2$ to a given such parametrization of $\ol\nu \Sigma_1'$.  The disks $\BD(\Sigma_1) \subseteq \ol\nu \Sigma_1$ are therefore sent to the disks $\BD(\Sigma_1') \subseteq \ol\nu \Sigma_1'$ by the new $F_1$.
\end{proof}

\begin{figure}[b]\center
 \def\svgwidth{\linewidth}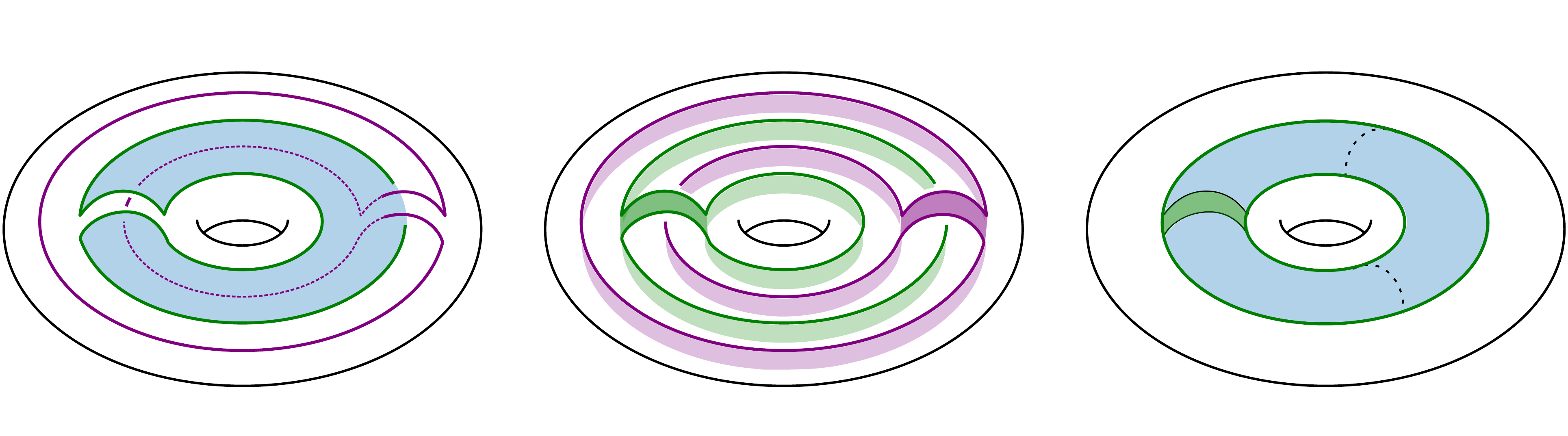
\caption{Throughout, we see $\partial \Sigma_1$ (resp. $\partial \Sigma_2$) as the green (resp. purple) curves inside $U \approx (\partial D^2) \times D^2 \subset S^3$.   {Left:} The disk $D \subset U$ with $\partial D \approx \partial \Sigma_1$.   {Middle:} Both $\Sigma_1$ and $\Sigma_2$ are formed from two disk fibers of $D^2 \times (\partial D^2)$ and a band inside $U$. The band $b$ joins $D^2 \times \{x\}$ and $D^2 \times \{y\}$ to form $\Sigma_1'$.   {Right:} The band $b$ and the disk $D$ together form an annulus $A$ inside the solid torus $U$.}
\label{fig:pieces}
\end{figure}

\begin{lemma}\label{lem:doubling-preserves-brunnian}
Let  $\Sigma = \Sigma_1 \sqcup \cdots \sqcup \Sigma_n \subseteq B^4$ be a Brunnian surface link with $\Sigma_1 \cong D^2$. Then $\BD(\Sigma) := \BD(\Sigma_1) \sqcup \Sigma_2 \sqcup \cdots \sqcup \Sigma_n$ is Brunnian.
\end{lemma}

\begin{proof}
Removing $\Sigma_i$ for $i \geq 2$ from $\BD(\S)$ yields a surface link smoothly isotopic to an unlinked surface with one component Bing doubled, by Lemma~\ref{lem:doubling-preserves-isotopy} with $\CAT=\DIFF$.  This is also unlinked. On the other hand, removing one component of $\BD(\Sigma_1)$ renders the second one an unknotted disk, split from $\Sigma_2 \sqcup \cdots \sqcup \Sigma_n$. Since the latter surface link is also unlinked by the Brunnian hypothesis, $\BD(\Sigma)$ is indeed Brunnian as desired.
\end{proof}

\begin{lem}\label{lem:view}
Let $\Sigma = \Sigma_1 \cup \Sigma_2$ denote the Bing double of the core disk $D^2 \! \times \! \{0\}$ in $D^2 \! \times \! D^2$,  and let $\Sigma_1'$ be a properly embedded disk in $D^2 \! \times \! D^2$ obtained as the push-in of a standard Seifert disk $D \subset S^3$ for $\partial \Sigma_1$, depicted in Figure~\ref{fig:pieces}(a). The surface links $\Sigma = \Sigma_1 \cup \Sigma_2$ and $\Sigma' = \Sigma_1' \cup \Sigma_2$ are smoothly isotopic rel.~boundary in $D^2 \! \times \! D^2$.
\end{lem}

\begin{proof}
We will show that the 2-sphere $\Sigma_1 \cup D$ bounds a 3-ball $\Delta$ in $D^2 \! \times \! D^2$ whose interior is disjoint from $\Sigma_2$. This implies that, in the complement of $\Sigma_2$, the disk $\Sigma_1$ is isotopic (rel.~boundary) to the push-in of the disk $D$, as desired.

 Let $U$ denote the solid torus $(\partial D^2) \! \times \! D^2\subset S^3$ containing $\partial \Sigma$. Note that each circle fiber  $(\partial D^2) \! \times \! \{pt\}$ bounds a disk  $D^2 \! \times \! \{pt\}$ inside $D^2 \! \times \! D^2$, and that distinct circle fibers bound disjoint disks. Each of $\Sigma_1$ and $\Sigma_2$ is formed from a pair of such disks by joining them via a band  inside the solid torus $U$. In particular, suppose that $\Sigma_1$ is formed from two disk fibers  $D^2 \! \times \! \{x\}$ and $D^2  \! \times  \! \{y\}$  via a band $b$, as in Figure~\ref{fig:pieces}(b). Also note that $D$ and $\Sigma_2$ meet along a single ribbon intersection.

 Together, the band $b\subset \Sigma_1$ and the Seifert disk $D$ form an annulus $A$ cobounded by the circle fibers $\partial D^2 \! \times \! \{x\}$ and $ \partial D^2  \! \times  \! \{y\}$. We see the annulus $A$ in Figure~\ref{fig:pieces}(c). Moreover, we observe that all of $A$ decomposes as a union of circle fibers. Taking the union of the disjoint disks in $D^2 \! \times \! D^2$ bounded by these circle fibers  yields a 3-ball $\Delta$ with boundary
$$\partial \Delta =  \left(D^2 \! \times \! \{x\}\right) \, \cup \, A \, \cup \,  \left( D^2  \! \times  \! \{y\}\right),$$
as depicted in  Figure~\ref{fig:decompose}. Finally, by rewriting $A$ as $D \cup b$ and noting that $$\Sigma_1 = \left(D^2 \! \times \! \{x\}\right) \, \cup \, b \, \cup \,  \left( D^2  \! \times  \! \{y\}\right),$$
we conclude that $\partial \Delta = \Sigma_1 \cup D$, as desired; see the right side of Figure~\ref{fig:decompose}. Moreover, we see that $\Delta$ meets $\Sigma_2$ only along the ribbon intersection between $\Sigma_2$ and $D$. 
\end{proof}

\begin{figure}[h] \center
  \def\svgwidth{.85\linewidth}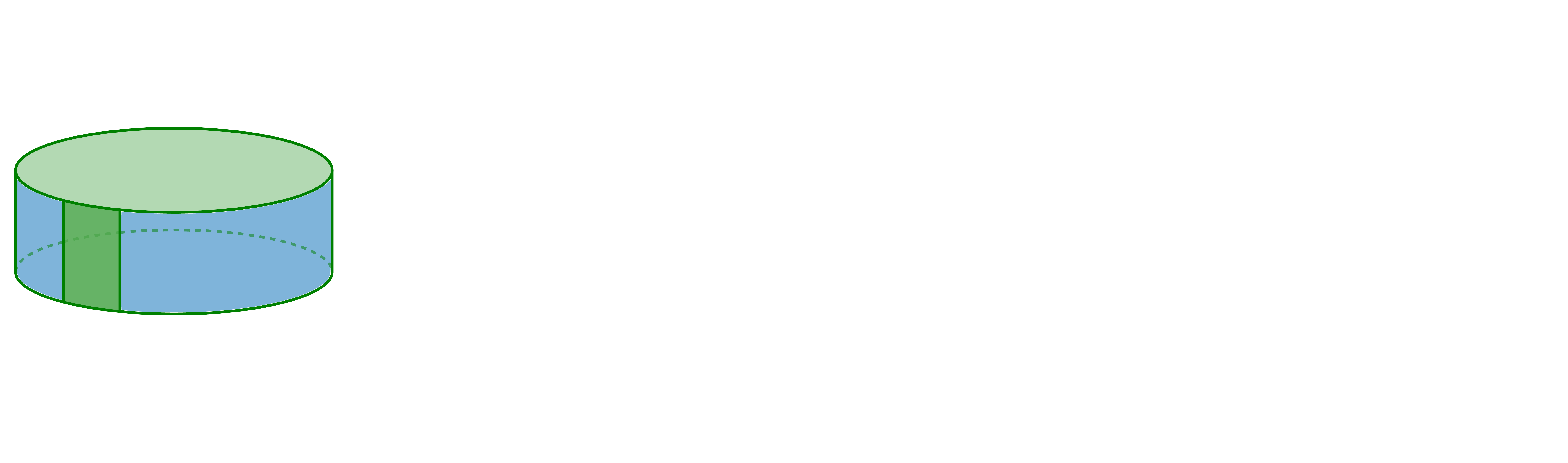
\caption{  {Left:} The 3-ball $\Delta$ is realized as $D^2 \times D^1$.   {Middle:} The boundary of the 3-ball, $\partial \Delta$, is decomposed as $(D^2 \times \{x\}) \cup A \cup (D^2 \times \{y\})$.   {Right:} By realizing $A$ as $b \cup D$, and the band $b$ as part of $\Sigma_1$, we deduce that $\partial \Delta$ is exactly $D \cup \Sigma_1$.}
\label{fig:decompose}
\end{figure}

\section{An exotic pair of disk links }\label{section:exotic-pair-disk-links}

We begin with the pair of disks $D_1$ and $D_2$ in $B^4$ depicted in Figure~\ref{fig:initial-disks}; these are obtained using the construction from \cite[Section 2]{Hayden} (as applied to the link $L14n_{40949}$), but are distinct from the examples used in that paper.  The figure shows two handle diagrams for $B^4$. To see this, note that by erasing the gray disk from either picture, the remaining dotted and $0$-labeled components form a Hopf link, corresponding to a cancelling 1- and 2- handle pair. These disks are disjoint from the 1-handle curves and all intersections between the disks and the 2-handles' attaching regions occur in the disks' interiors. Therefore, after pushing the disks' interiors into the 0-handle, the disks are indeed embedded in $B^4$. 

By construction, these disks are bounded by the same knot $K$ in $S^3$, which we redraw in the standard diagram for $S^3$ in Figure~\ref{fig:slice-knot}.

\begin{figure}[h]\center
\def\svgwidth{\linewidth}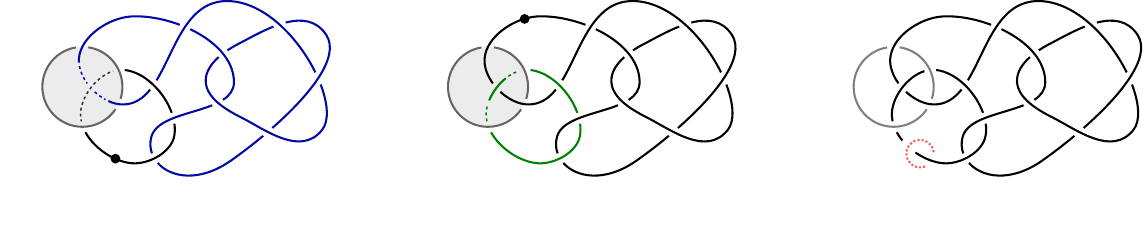
\caption{Parts (a) and (b) depict disks $D_1$ and $D_2$ in non-standard handle diagrams of $B^4$. Part (c)  depicts the slice knot $K=\partial D_1 = \partial D_2$ and a distinguished curve  $\gamma \subseteq S^3 \sm K$, drawn in a non-standard surgery description of $S^3$.}\label{fig:initial-disks}
\end{figure}


\begin{figure}\center
\def\svgwidth{.27\linewidth}
\begingroup%
  \makeatletter%
  \providecommand\color[2][]{%
    \errmessage{(Inkscape) Color is used for the text in Inkscape, but the package 'color.sty' is not loaded}%
    \renewcommand\color[2][]{}%
  }%
  \providecommand\transparent[1]{%
    \errmessage{(Inkscape) Transparency is used (non-zero) for the text in Inkscape, but the package 'transparent.sty' is not loaded}%
    \renewcommand\transparent[1]{}%
  }%
  \providecommand\rotatebox[2]{#2}%
  \newcommand*\fsize{\dimexpr\f@size pt\relax}%
  \newcommand*\lineheight[1]{\fontsize{\fsize}{#1\fsize}\selectfont}%
  \ifx\svgwidth\undefined%
    \setlength{\unitlength}{150.11510053bp}%
    \ifx\svgscale\undefined%
      \relax%
    \else%
      \setlength{\unitlength}{\unitlength * \real{\svgscale}}%
    \fi%
  \else%
    \setlength{\unitlength}{\svgwidth}%
  \fi%
  \global\let\svgwidth\undefined%
  \global\let\svgscale\undefined%
  \makeatother%
  \begin{picture}(1,0.66235546)%
    \lineheight{1}%
    \setlength\tabcolsep{0pt}%
    \put(0,0){\includegraphics[width=\unitlength,page=1]{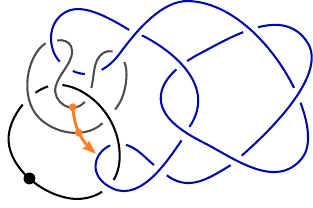}}%
    \put(0.69737982,0.01761345){\color[rgb]{0,0.03529412,0.63921569}\makebox(0,0)[lt]{\lineheight{1.25}\smash{\begin{tabular}[t]{l}$0$\end{tabular}}}}%
    \put(-0.00756256,0.42550195){\color[rgb]{0.30196078,0.30196078,0.30196078}\makebox(0,0)[lt]{\lineheight{1.25}\smash{\begin{tabular}[t]{l}$K$\end{tabular}}}}%
  \end{picture}%
\endgroup%

\caption{Redrawing $K$ in the standard diagram of $S^3$.}\label{fig:slide}
\end{figure}

\begin{prop}\label{prop:initial-disks}
The disks $D_1$ and $D_2$ are topologically ambiently isotopic rel.\ boundary.
\end{prop}

\begin{proof}
The disks are smoothly embedded and are bounded by the same knot in $S^3$, so it suffices to show that these disks' exteriors have $\pi_1 \cong \Z$ \cite[Theorem~1.2]{ConwayPowell}. Handle diagrams for these disk exteriors are drawn in parts (a-1) and (b-1) of Figure~\ref{fig:pi1}; the remaining parts of Figure~\ref{fig:pi1} manipulate and decorate these diagrams to simplify the $\pi_1$ calculation. By tracing the 2-handle curves from parts (a-2) and (b-5) starting at the labeled arrows, we obtain the following presentations:
\begin{align*}
\pi_1(B^4 \setminus D_1) &= \langle x,y \mid x^{-1} x y^{-1} x^{-1} y =1 \rangle = \langle x,y \mid y^{-1} x^{-1} y =1 \rangle  \cong \Z \\
\pi_1(B^4 \setminus D_2) &= \langle x,y \mid x x^{-1} x^{-1} x y x^{-1} y^{-1}=1 \rangle = \langle x,y \mid y x^{-1} y^{-1}=1 \rangle  \cong \Z
\end{align*}
It follows that the disks are topologically ambiently isotopic rel.\ boundary.
\end{proof}

\begin{figure}\center
\def\svgwidth{\linewidth}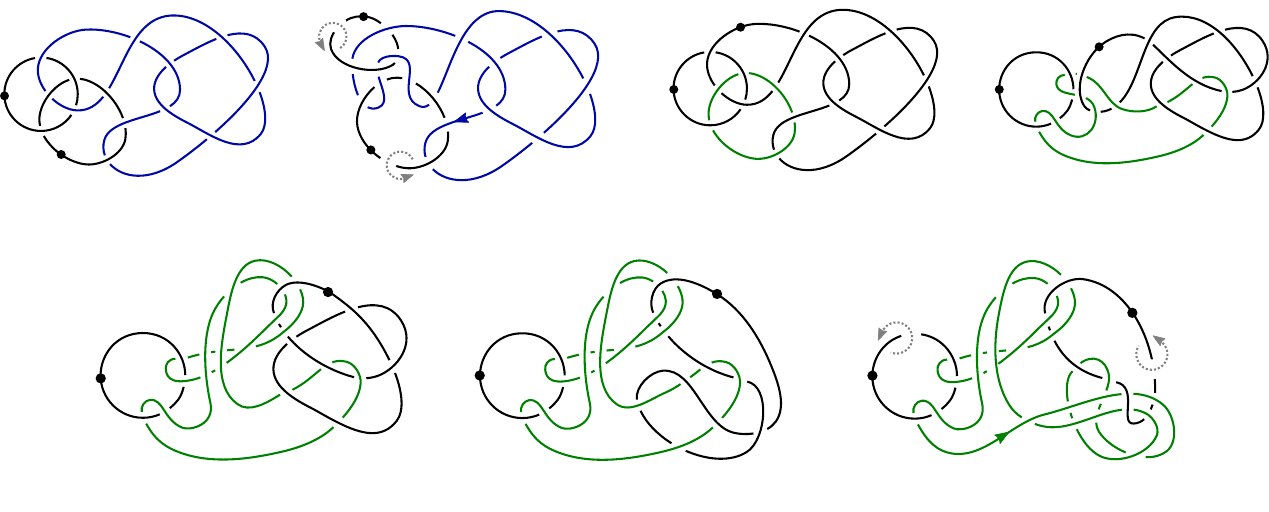
\caption{Redrawing the disk exteriors to compute their fundamental groups.}\label{fig:pi1}
\end{figure}

\begin{figure}\center
\def\svgwidth{.75\linewidth}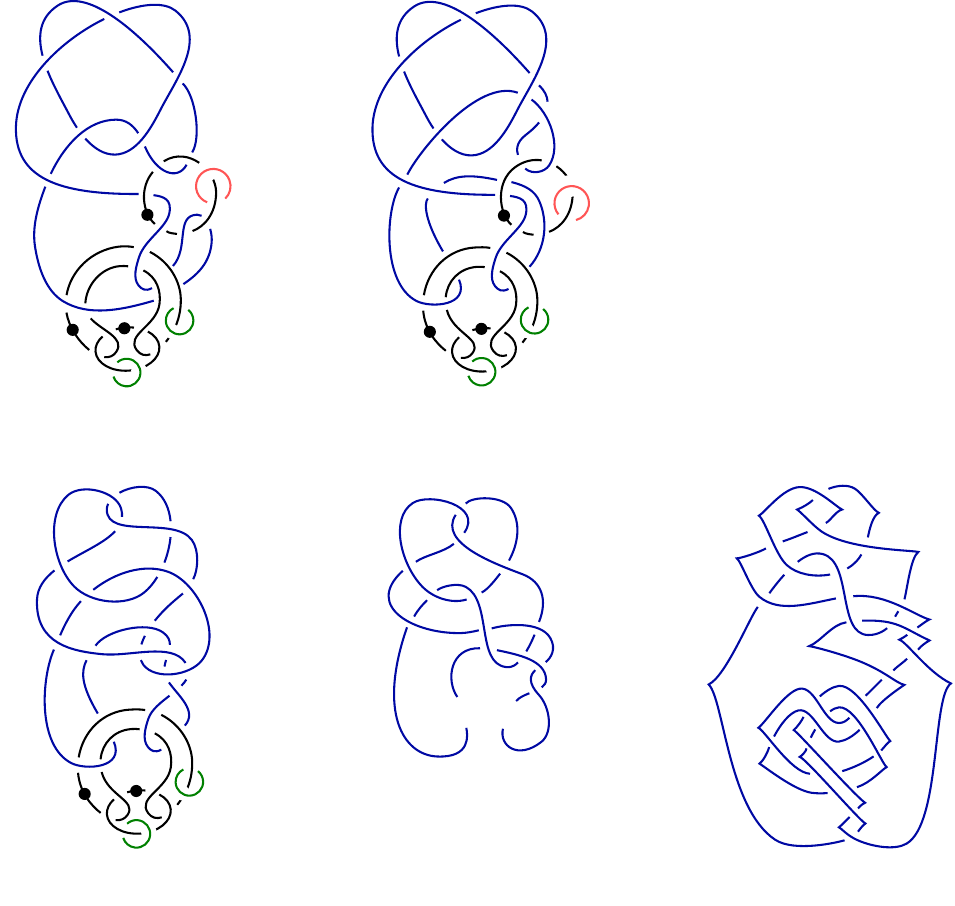
\caption{Attaching three 2-handles to the exterior of $\BD(D_1)$ and simplifying. All steps except (b) to (c) and (e) to (f) are isotopies. From (b) to (c), we cancel the red 2-handle with a 1-handle, introducing a full twist into all blue strands passing through the 1-handle. Diagrams (e) and (f) are related as in Figure~\ref{fig:BDhandles}.}\label{fig:add-handles}
\end{figure}

\begin{figure}\center
\def\svgwidth{\linewidth}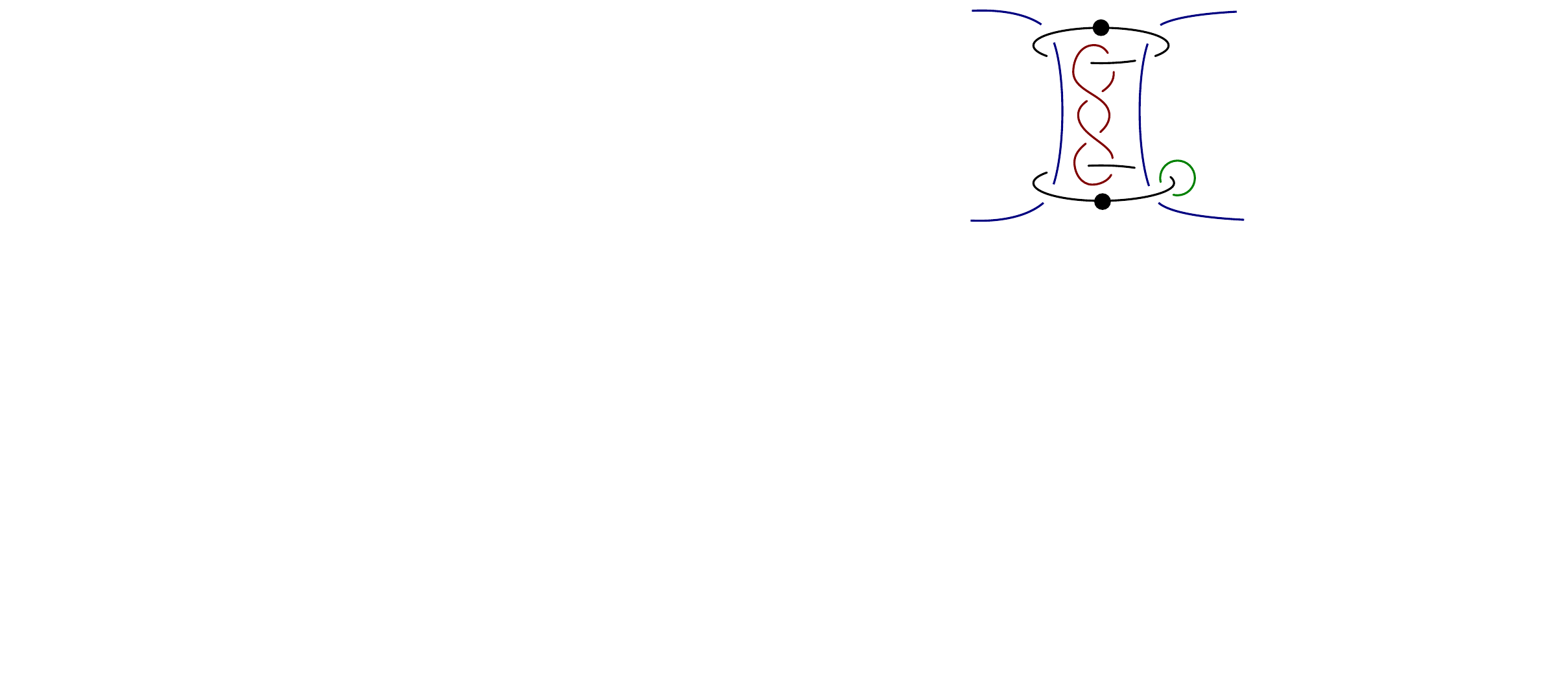
\caption{The process of simplifying (e) to (f) in Figure~\ref{fig:add-handles}. From (a) to (b), we add a cancelling handle pair. From (b) to (c), we remove a cancelling pair. From (c) to (d) to (e) is simply isotopy of the diagram. From (e) to (f), we remove a cancelling pair, and then isotope to obtain (g). From (g) to (h), we do one handle slice, and then from (h) to (i) we remove a cancelling pair and further isotope to obtain (j).}\label{fig:BDhandles}
\end{figure}


Now take the Bing doubles of the disks $D_1$ and $D_2$. Note that each of $\BD(D_1)$ and $\BD(D_2)$ are bounded by the 2--component link $\BD(K)$.  

\begin{corollary}
  The disk links $\BD(D_1)$ and $\BD(D_2)$ are topologically ambiently isotopic rel.\ boundary.
\end{corollary}

\begin{proof}
This follows by combining Proposition~\ref{prop:initial-disks} and Lemma~\ref{lem:doubling-preserves-isotopy} with $\CAT=\TOP$.
\end{proof}

By Lemma~\ref{lem:doubling-preserves-brunnian}, $\BD(D_1)$ and $\BD(D_2)$ are Brunnian.  To prove Theorem~\ref{thm:intro-disk-links}, it therefore remains to establish the following proposition.

\begin{prop}\label{prop:base-case}
  The disk links $\BD(D_1)$ and $\BD(D_2)$ are not smoothly equivalent.
\end{prop}

The proof of the proposition uses the following technical lemma.

\begin{lemma}\label{lem:hyperbolic}
The knot complement $S^3 \sm K$ has a hyperbolic structure with trivial isometry group.
\end{lemma}

\begin{proof}
This is verified using SnapPy \cite{snappy} and Sage \cite{sagemath}; see Appendix~\ref{appendixa} for additional documentation regarding this calculation.
\end{proof}

\begin{proof}[Proof of Proposition~\ref{prop:base-case}]
  To distinguish $\BD(D_1)$ and $\BD(D_2)$, we will examine the simple closed curve $\gamma \subseteq S^3 \sm K$ shown in Figure~\ref{fig:initial-disks}(c). This curve is redrawn in the exteriors of $\BD(D_1)$ and $\BD(D_2)$ in Figure~\ref{fig:BD-exteriors}. The curve $\gamma$ bounds a smooth disk in the exterior of $\BD(D_2)$, since it bounds an obvious disk in the diagram from Figure~\ref{fig:BD-exteriors}(b) that only intersects the 2-handle attaching circle (and not the 1-handle curve).

\begin{figure}\center
\def\svgwidth{.7\linewidth}
\begingroup%
  \makeatletter%
  \providecommand\color[2][]{%
    \errmessage{(Inkscape) Color is used for the text in Inkscape, but the package 'color.sty' is not loaded}%
    \renewcommand\color[2][]{}%
  }%
  \providecommand\transparent[1]{%
    \errmessage{(Inkscape) Transparency is used (non-zero) for the text in Inkscape, but the package 'transparent.sty' is not loaded}%
    \renewcommand\transparent[1]{}%
  }%
  \providecommand\rotatebox[2]{#2}%
  \newcommand*\fsize{\dimexpr\f@size pt\relax}%
  \newcommand*\lineheight[1]{\fontsize{\fsize}{#1\fsize}\selectfont}%
  \ifx\svgwidth\undefined%
    \setlength{\unitlength}{392.84570409bp}%
    \ifx\svgscale\undefined%
      \relax%
    \else%
      \setlength{\unitlength}{\unitlength * \real{\svgscale}}%
    \fi%
  \else%
    \setlength{\unitlength}{\svgwidth}%
  \fi%
  \global\let\svgwidth\undefined%
  \global\let\svgscale\undefined%
  \makeatother%
  \begin{picture}(1,0.32840294)%
    \lineheight{1}%
    \setlength\tabcolsep{0pt}%
    \put(0,0){\includegraphics[width=\unitlength,page=1]{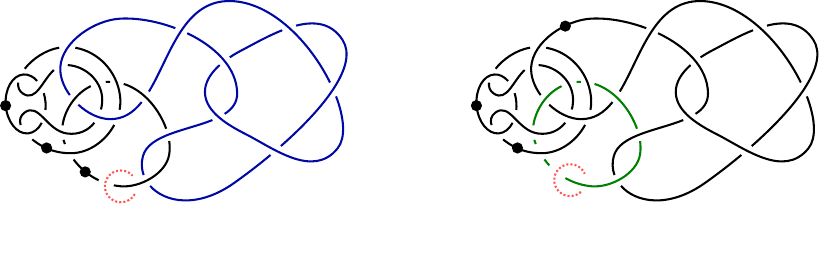}}%
    \put(0.10269412,0.07403866){\color[rgb]{1,0.33333333,0.33333333}\makebox(0,0)[lt]{\lineheight{1.25}\smash{\begin{tabular}[t]{l}$\gamma$\end{tabular}}}}%
    \put(0.64450636,0.08417077){\color[rgb]{1,0.33333333,0.33333333}\makebox(0,0)[lt]{\lineheight{1.25}\smash{\begin{tabular}[t]{l}$\gamma$\end{tabular}}}}%
    \put(0.30229297,0.08301148){\color[rgb]{0,0.03529412,0.63921569}\makebox(0,0)[lt]{\lineheight{1.25}\smash{\begin{tabular}[t]{l}$0$\end{tabular}}}}%
    \put(0.73083962,0.07474268){\color[rgb]{0,0.50196078,0}\makebox(0,0)[lt]{\lineheight{1.25}\smash{\begin{tabular}[t]{l}$0$\end{tabular}}}}%
    \put(0.16178889,0.00890623){\makebox(0,0)[lt]{\lineheight{1.25}\smash{\begin{tabular}[t]{l}(a)\end{tabular}}}}%
    \put(0.75798196,0.00890623){\makebox(0,0)[lt]{\lineheight{1.25}\smash{\begin{tabular}[t]{l}(b)\end{tabular}}}}%
  \end{picture}%
\endgroup%

\caption{Handle diagrams for the exteriors of the disk links $\BD(D_1)$ and $\BD(D_2)$ in parts (a) and (b), respectively.}\label{fig:BD-exteriors}
\end{figure}


On the other hand, we claim that $\gamma$ does not bound a smoothly embedded disk in the exterior of $\BD(D_1)$. To prove this, we begin by eliminating the 1-handles in the exterior of $BD(D_1)$ by attaching three additional $-1$-framed 2-handles along meridians to the dotted 1-handle curves as shown in Figure~\ref{fig:add-handles}, yielding a larger 4-manifold $W$.
Note that, if $\gamma$ were to bound a smoothly embedded disk in the exterior of $\BD(D_1)$, then this disk could be capped off with the core of the $-1$-framed 2-handle attached along $\gamma$ to produce a smoothly embedded 2-sphere of square $-1$ in $W$. We will show that no such 2-sphere can exist. After simplifying the handle diagram (Figures~\ref{fig:add-handles}-\ref{fig:BDhandles}), we see that the resulting 4-manifold $W$ is obtained from $B^4$ by attaching a single $-1$-framed 2-handle along a knot $J \subseteq S^3$. In Figure~\ref{fig:add-handles}(f), the knot $J$ has been drawn as a Legendrian knot with Thurston-Bennequin number $tb=0$. (The Thurston-Bennequin number is computed from a Legendrian knot diagram as the diagram's writhe minus the number of right cusps.) Since the 2-handle is attached with framing $-1= tb-1$,  $W$ admits a Stein structure \cite{yasha:stein}. By \cite{lisca-matic}, it follows that $W$ cannot contain any smoothly embedded 2-sphere of square $-1$. It follows that $\gamma$ cannot bound a smoothly embedded disk in the exterior of $\BD(D_1)$.

This proves that $\BD(D_1)$ and $\BD(D_2)$ are not smoothly equivalent rel.\ boundary. To obtain the stronger conclusion stated in the theorem, we will show that any supposed diffeomorphism $F \colon (B^4,\BD(D_1)) \to (B^4,\BD(D_2))$ must fix $\gamma$ up to isotopy. Observe that, along the boundary, $F$ restricts to a self-diffeomorphism of $(S^3,\BD(K))$. It follows that it further restricts to a self-diffeomorphism of $S^3 \setminus \nu \BD(K)$. By Lemma~\ref{lem:hyperbolic}, $K$ has a hyperbolic complement (with trivial isometry group, to be used below). 
A JSJ decomposition of $S^3 \sm \nu \BD(K)$ is therefore given by \[S^3 \sm \nu \BD(K) = (S^3 \sm \nu K )\cup (S^1 \!\! \times \! D^2 \sm \BD(S^1 \! \times \! \{0\})).\]
The latter piece is diffeomorphic to the complement of the Borromean rings.
By uniqueness of JSJ decompositions of 3-manifolds \cite{jaco1979seifert, johannson1979homotopy} (see also Theorem~\ref{thm:key-JSJ-fact}), up to isotopy any diffeomorphism of pairs must restrict on $S^3 \sm \nu K \subseteq S^3 \sm \nu \BD(K)$ to a self-diffeomorphism of $S^3 \sm \nu K$. 
We need to show that every diffeomorphism of $S^3 \sm \nu K$ is isotopic to the identity to conclude that $\gamma$ is fixed up to isotopy.
By Mostow rigidity \cite{Mostow}, every homotopy equivalence of a hyperbolic 3-manifold is homotopic to an isometry.
Waldhausen's \cite{Waldhausen} work, combined with Hatcher's proof of the Smale conjecture \cite{Hatcher-smale-conj}, shows that for compact, orientable 3-manifolds $N \neq B^3$ with nonempty boundary, homotopic diffeomorphisms are isotopic.
   Combining these two statements, it follows that every diffeomorphism of a hyperbolic knot complement is isotopic to an isometry.
Since the isometry group of $S^3 \sm \nu K$ is trivial, this completes the proof of the claim. \end{proof}

\smallskip

\section{Infinitely many exotic surface links}\label{section:infinitely-many-surface-links}

In this section, we construct the 2--component surface links promised in Theorem~\ref{thm:intro-surface-links}.  In Section~\ref{sec:knotfloer}, we begin with some background on Heegaard Floer cobordism maps, which provides our diffeomorphism obstructions.  In Section~\ref{subsection:rim-surgery}, we review the rim surgery construction. Section~\ref{sec:infinitefamily} constructs infinitely many pairwise exotic Brunnian surface links.

\subsection{Heegaard Floer cobordism maps}\label{sec:knotfloer}

In this section, the obstruction to smooth isotopy comes from maps on link Floer homology. Link Floer homology was introduced by Ozsv\'ath--Szab\'o \cite{osknot,oslink}, while these cobordism maps were later defined by Juh\'asz \cite{juhasz}. 
We refer the reader to e.g.~\cite{JuhaszMillerZemke} for a more detailed description; for our purposes, the following brief description suffices.

\begin{definition}
Let $l$ be a link in a closed, oriented, connected 3-manifold $M$. A \emph{multi-pointed link} $\ell$ is the link $l$ with two basepoints,  labeled $w$ and $z$, chosen per component of $l$.
\end{definition} 

Let $\DLink^{\times}$ be the groupoid of multi-pointed links in closed, oriented, connected 3-manifolds, where the morphisms are smooth multi-pointed isotopies, considered up to isotopy of isotopies.  
Let $\Vect^{\times}$ be the groupoid of $\mathbb{F}_2$-vector spaces with linear isomorphisms.
Link Heegaard Floer homology gives rise to a functor 
\[\widehat{\HFL} \colon \DLink^{\times} \to\Vect^{\times}.\]
  Given the input of a multi-pointed link $\ell$ in a 3-manifold $M$, we call $\widehat{\HFL}(M,\ell)$ the {\emph{link Floer homology}} of $\ell$.



\begin{definition}
Let $M_0$ and $M_1$ be closed, oriented, connected 3-manifolds.  A {\emph{decorated cobordism}} between multi-pointed links $\ell_0\subseteq M_0$ and $\ell_1\subseteq M_1$ consists of a compact, oriented surface $\S$ properly embedded in a compact, connected, oriented 4-manifold $W$ with:
\begin{enumerate}
\item $\partial W=M_0\sqcup-M_1$;
\item $\partial \S=l_0\sqcup -l_1$; and
\item a collection of arcs $\mathcal{A}$ properly embedded in $\S$ such that:
\begin{itemize}
\item $\mathcal{A}$ does not meet any of the $w$ or $z$ basepoints of $\ell_i$;
\item each component of $\ell_i-\{w$ and $z$ basepoints$\}$ meets exactly one endpoint of $\mathcal{A}$;
\item the components of $\S\setminus\mathcal{A}$ can be sorted into two subsurfaces $\S_w$ and $\S_z$ of $\S$ so that all $w$ basepoints are in $\S_w$ and all $z$ basepoints are in $\S_z$.
\end{itemize}
\end{enumerate}

We may view a properly embedded surface $\S$ in a 4-manifold $W$ with connected boundary as a cobordism from $(\emptyset,\emptyset)$ to the boundary $(M,l)$. In this setting, we can decorate $\S$ after choosing $w$ and $z$ basepoints on each  component of $l$ in an essentially canonical way: take $\S_w$ to consist of small bigons including each $w$ point and $\S_z=\S\setminus \S_w$. We call this a {\emph{trivial decoration of $\S$.}} If no other decoration is specified, then we always assume a surface is trivially decorated. 
\end{definition}

Juh\'asz \cite{juhasz} showed that decorated link cobordisms give rise to an extension of the functor $\widehat{\HFL}$. 
Let $\DLink$ denote the category of multi-pointed links $(M,\ell)$, where the morphisms are smooth ambient isotopy classes of decorated cobordisms $[(W,\S,\mathcal{A})]$. Let $\Vect$ denote the category of $\mathbb{F}_2$-vector spaces with linear transformations as morphisms.  
Since we can consider a multi-pointed isotopy as a decorated cobordism, $\DLink^{\times} \subseteq \DLink$ and $\Vect^{\times} \subseteq \Vect$  are subcategories that contain all the objects but fewer morphisms. Juh\'asz constructed a functor 
\[\widehat{\HFL} \colon \DLink \to \Vect\]
with the same value $\widehat{\HFL}(M,\ell)$ on multi-pointed links  as the functor $\widehat{\HFL} \colon \DLink^{\times}\to\Vect^{\times}$ introduced above; this explains why we use the same notation for both functors. 


Next, we make explicit the above formalism, and recall a key computation from \cite{JuhaszMillerZemke}.


\begin{remark}\label{rem:usefulprop}
\nopagebreak \
\begin{enumerate}
\item Given a decorated link cobordism $(W,\S,\mathcal{A})$ from $(M_0,\ell_0)$ to $(M_1,\ell_1)$ there is an induced map 
\[F_{W,\S}\colon \widehat{\HFL}(M_0,\ell_0)\to\widehat{\HFL}(M_1,\ell_1).\]
We will omit the arcs $\mathcal{A}$ from the notation for cobordism maps.  
This map is well-defined up to smooth ambient isotopy of $\S$ rel.\ boundary. 
Given a surface $\S$ whose boundary is a multi-pointed link $\ell$ in a 3-manifold $M$, we may view $\S$ as having trivial decoration; moreover, $\S$ naturally induces a map 
\[F_{W,\S}\colon \widehat{\HFL}(\emptyset)\to\widehat{\HFL}(M,\ell).\] 
The vector space $\widehat{\HFL}(\emptyset)$ is a copy of $\mathbb{F}_2$, so $F_{W,\S}$ is determined by the image of the single generator.
\item\label{prop1} According to \cite[Theorem~11.3]{juhasz}, link cobordism maps behave well under composition. That is, if a decorated cobordism $(W,\S)$ from $(M_0,\ell_0)$ to $(M_1,\ell_1)$ can be split as the composition of decorated cobordisms $(W_0,\S_0)$ from $(M_0,\ell_0)$ to $(M',\ell')$, followed by $(W_1,\S_1)$ from $(M',\ell')$ to $(M_1,\ell_1)$, then \[F_{W,\S}=F_{W_1,\S_1}\circ F_{W_0,\S_0}.\]
\item\label{prop2} By \cite[Corollary 8.4]{JuhaszMillerZemke}, if $\S\subseteq B^4$ is obtained by pushing the interior of a quasipositive Seifert surface for a knot in $S^3$ into $B^4$ then $F_{B^4,\S}$ is nonvanishing. (In fact, $\Omega(\S)=0$; see Section \ref{sec:omega}.)
\end{enumerate}
\end{remark}

The composition law of Remark~\ref{rem:usefulprop}~\eqref{prop1} provides a means of decomposing a complicated cobordism into simpler ones. Remark~\ref{rem:usefulprop}~\eqref{prop2} is useful because in practice, link cobordism maps are difficult to compute; this gives a large family of examples where we at least know the induced map is nontrivial. For example, the genus one Seifert surface for any iterated positive untwisted Whitehead double of a positive knot $K$ (i.e.\ $\Wh_+(\Wh_+(\cdots(\Wh_+(K))\cdots))$) is quasipositive \cite{rudolphsqp}.

\subsection{The \texorpdfstring{$\boldsymbol{\Omega}$}{Omega}-invariant}\label{sec:omega}
The link cobordism map $F_{W,\S}$ is an invariant of $\S$ only up to isotopy rel.\ boundary. To obstruct two surfaces from being smoothly equivalent, we use the invariant $\Omega$ from \cite[Section 6]{JuhaszMillerZemke}, defined for surfaces embedded in $B^4$.

Given an oriented, properly embedded surface $\S$ in $B^4$ with positive genus and connected boundary, 
there exists an invariant $\Omega(\S)\in\Z^{\ge 0}\cup\{-\infty\}$ with the property that if $(B^4,\S)$ is diffeomorphic to $(B^4,\S')$, then $\Omega(\S)=\Omega(\S')$. The invariant $\Omega(\S)$ is defined to be $-\infty$ if and only if $F_{B^4,\S}$ vanishes.

In~\cite{JuhaszMillerZemke}, the surface $\S$ is assumed to have connected boundary; to avoid re-writing that material for the disconnected boundary case, we continue to use $\Omega$ in the connected boundary setting. The details of the construction of $\Omega$ are beyond the scope of this paper, so we provide a heuristic description of $\Omega(\S)$ when $\S$ has genus $g$.
\begin{itemize}
    \item We can view $F_{B^4,\S}(1)$ as an element of $\widehat{\HFL}(S^3,\ell)\otimes\mathbb{F}_2[\mathbb{Z}^{2g}]$.
    \item For each element $a$ of $\mathbb{F}_2[\mathbb{Z}^{2g}]$, $\Omega(a)$ denotes the number of irreducible factors of $a$, counted with multiplicity (here we use the fact that $\mathbb{F}_2[\mathbb{Z}^{2g}]$ is a UFD).
    \item  
    Now let
    $\Omega(\S)=\max\{\Omega(a)\mid F_{B^4,\S}(1)=a\cdot y$ for some $y\}$.
\end{itemize}

Note that this is a description and not the definition, which requires deformed knot Floer homology; see Remark \ref{rem:twistedfloer} for a longer discussion and  \cite{JuhaszMillerZemke} for even more details.


\subsection{Rim surgery}\label{subsection:rim-surgery}

We briefly recall the technique \textit{rim surgery}. This well known method for producing potentially exotic pairs of surfaces in 4-manifolds was invented by Fintushel-Stern \cite{fintushelstern}, and generalizes the twist-spin construction of Zeeman \cite{zeeman}.

\begin{definition}
Let $X$ be a 4-manifold and let $\S \subseteq X$ be a smoothly embedded surface. Let $\alpha \subseteq \Sigma$ be a simple closed curve with $w_1^{\Sigma}(\alpha) =0$, 
and $J \subseteq S^3$ be a knot. 
(Here, $w_1^{\Sigma}$ is the first Steifel-Whitney class of $\Sigma$, 
so the condition $w_1^{\Sigma}(\alpha) =0$ requires that $\alpha$ admits a neighbourhood homeomorphic to an annulus.) 

Choose a framing $\eta$ of the 2-subbundle of the normal bundle of $\alpha$ that is normal to $\Sigma$. 
Now we can identify $\nu\alpha$ with $B^3\times S^1$; 
choose this identification so that $\alpha=\{0\} \times S^1$, $\S\cap\nu\alpha=I\times S^1$ for a fixed vertical arc $I\subseteq B^3$, and $\eta$ restricts to the same pair of vectors in $T_0 B^3$ for each $\{0\} \times \theta$. We construct the \emph{rim surgered surface} $\S(\alpha;J) \subseteq X$ as follows: the surface agrees with $\S$ outside of $\ol\nu\alpha$. Note that $(\ol\nu\alpha,\S\cap \ol\nu\alpha)=(B^3\times S^1,I\times S^1)$, where $I$ is an unknotted arc in $B^3$. To obtain $\S(\alpha;J)$, replace each $(B^3, I)$ with $(B^3, \mathring{J})$, i.e.\ the tangle obtained from $(S^3,J)$ by deleting a small ball from $S^3$ centered at a point on $J$.
\end{definition}

As written, $\S(\alpha;J)$ depends on our choice of framing $\eta$. There are an integers' worth of such framings that induce a fixed orientation on the subbundle.  Choosing a different $\eta$ twists $B^3\times\theta$ an integer number of times about a fixed axis with boundary $\partial\mathring{J}$, as $\theta$ goes from $0$ to $2\pi$ (see Figure~\ref{fig:twisting}, second row vs.\ third row).

One setting in which it is easy to specify a framing $\eta$ is in the case that $\alpha$ bounds a framed locally flat disk $\Delta$ in the complement of $\Sigma$. Here, the disk $\Delta$ is \emph{framed} if a section of the normal bundle of $\alpha$ that lies in $T\Sigma$ extends to a nonvanishing section over all of $\Delta$. We can then specify that for each $x\in\alpha$, the first coordinate of $\eta(x)$ points into $\Delta$. This ensures that $\Delta$ intersects $(\nu\alpha)$ in an annulus of the form (arc)$\times S^1$. 
Here, we say that replacing $\S\cap \ol\nu\alpha$ with $\mathring{J}\times S^1$ yields $\S(\alpha;J,0)$. We call this {\emph{0-twisted rim surgery}} (relative to $\Delta$). If we twist the copy of $\mathring{J}$ a total of $n$ times, as above, we instead obtain a surface that we call $\S(\alpha;J,n)$. We call this {\emph{$n$-twisted rim surgery}}.
When $\Delta$ has been implicitly specified, we write $\S(\alpha;J,n)$ as a well-defined surface without reference to $\Delta$. In Figure~\ref{fig:twisting}, we illustrate a neighborhood of $\alpha$ intersecting $\S$, $\S(\alpha;T,0)$, and $\S(\alpha;T,1)$ for $T$ the right-handed trefoil.

\begin{figure}
\labellist
  \pinlabel \footnotesize{$B^3\times0$} at 27 200
  \pinlabel \footnotesize{$B^3\times\pi/2$} at 99 200
  \pinlabel \footnotesize{$B^3\times\pi$} at 170 200
  \pinlabel \footnotesize{$B^3\times3\pi/2$} at 240 200
  \pinlabel \footnotesize{$B^3\times2\pi$} at 310 200
  \pinlabel \footnotesize{$B^3\times0$} at 27 95
  \pinlabel \footnotesize{$B^3\times\pi/2$} at 99 95
  \pinlabel \footnotesize{$B^3\times\pi$} at 170 95
  \pinlabel \footnotesize{$B^3\times3\pi/2$} at 240 95
  \pinlabel \footnotesize{$B^3\times2\pi$} at 310 95
  \pinlabel \footnotesize{$B^3\times0$} at 27 -12
  \pinlabel \footnotesize{$B^3\times\pi/2$} at 99 -12
  \pinlabel \footnotesize{$B^3\times\pi$} at 170 -12
  \pinlabel \footnotesize{$B^3\times3\pi/2$} at 240 -12
  \pinlabel \footnotesize{$B^3\times2\pi$} at 310 -12
\pinlabel \textcolor{blue}{$\Delta$} at 115 250
\pinlabel \textcolor{red}{$\alpha$} at 20 242
\pinlabel $\S\cap \ol\nu\alpha$ at -25 240
\pinlabel $\S(\alpha;T,0)\cap \ol\nu\alpha$ at -45 133
\pinlabel $\S(\alpha;T,1)\cap \ol\nu\alpha$ at -45 28
\endlabellist
\hspace{.5in}\includegraphics[width=100mm]{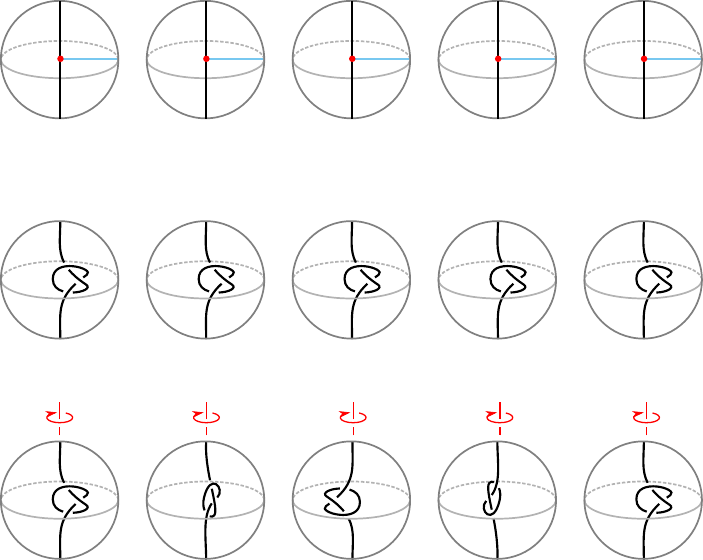}
\vspace{.2in}
\caption{  {Top row:} we draw a neighborhood $\ol\nu\alpha$ of $\alpha$, a simple closed curve in $\S$ with annular neighborhood.  In blue, we draw the portion of a framed disk $\Delta$ bounded by $\alpha$ that intersects $\ol\nu\alpha$. We draw $\ol\nu\alpha$ as $B^3\times S^1$, with $\alpha$ as $0\times S^1$. The surface $\S$ intersects $\ol\nu\alpha$ in (vertical arc)$\times S^1$ as illustrated. The parameterization of $\ol\nu\alpha$ is chosen so that $\Delta$ intersects it in (arc from $0$ to $\partial B^3$)$\times S^1$, as illustrated.   {Middle row:} replacing $\S\cap \ol\nu\alpha=I\times S^1$ with $\mathring{T}\times S^1$ yields $\S(\alpha;T,0)$.   {Bottom row:} replacing the $I\times \theta$s of $\S\cap \ol\nu\alpha$ with copies of $\mathring{T}$ that rotate once about a vertical axis as $\theta$ runs from $0$ to $2\pi$ yields $\Sigma(\alpha;T,1)$.}\label{fig:twisting}
\end{figure}

\begin{remark}\label{rem:zeeman}
Let $\Sigma$ be an unknotted sphere in $S^4$ obtained by doubling $(B^4,D)$, where $D$ is a boundary-parallel disk in $B^4$. Let $\alpha=\partial D$, so $\alpha$ is a curve in $\Sigma$ that bounds a framed disk (a parallel copy of $D$) into the complement of $\S$. Then $\S(\alpha;J,n)$ is the $n$-twist spin of the knot $J$ as constructed by Zeeman \cite{zeeman}.  Zeeman also showed that the 1-twist spin of any knot $J$ is an unknotted $S^2 \subseteq S^4$. 
\end{remark}

Next, we 
give a criterion for rim surgery to preserve the isotopy class.

\begin{lemma}$($\cite[Corollary 2.7]{JuhaszMillerZemke}$)$\label{lemma:when-rim-surgery-doesnt-change-anything}
  If $\alpha$ bounds a $\CAT$-embedded framed disk $D$ whose interior lies in $B^4 \sm \S$, then $\Sigma(\alpha;J,1)$ and $\S$ are $\CAT$ ambiently isotopic rel.\ boundary. 
\end{lemma}


\begin{proof}
This proof is based on Zeeman's \cite{zeeman} work on twist-spun knots. Since $D$ is framed, there is a thickening of $D$ to $D\times I$ meeting $\S$ in $(\partial D)\times I$ (i.e.\ we can take parallel copies of $D$ that are disjoint from $D$ and have boundary on $\S$). 
Let $D'$ and $D''$ be two parallel copies of $D$, pushed to have boundary curves on opposite sides of $\alpha$ in $\S$, and compress $\S$ along both $D'$ and $D''$ to obtain a surface $\S'$. Thus, $\S'$ has a 2-sphere component $U$ bounding $D\times I$, and $U$ is $\CAT$-unknotted in the complement of the rest of $\S'$.

The surface $\Sigma$ can be obtained from $\Sigma'$ by surgery along two framed arcs $\eta'$ and $\eta''$ (undoing the compressions along $D'$ and $D''$, respectively), as in Figure~\ref{fig:rimnothing}.
Additionally, the curve $\alpha$ exists in both $\Sigma$ and $\Sigma'$, so rim surgery can be performed on both surfaces along $\alpha$. By the same logic, we see that $\Sigma(\alpha;J,1)$ can be obtained from $\Sigma'(\alpha;J,1)$ by surgery along $\eta'$ and $\eta''$. Note that one component of $\Sigma'(\alpha;J,1)$ is the rim-surgered 2-sphere $U(\alpha;J,1)$. As in Remark~\ref{rem:zeeman},  Zeeman~\cite{zeeman} showed that the 1-twist spin $U(\alpha;J,1)$ is  $\CAT$-unknotted for any $J$, so $\Sigma'(\alpha;J,1)$ is $\CAT$-isotopic to $\Sigma'$. Since $\eta'$ and $\eta''$ are arcs each meeting $\nu U$ in only one interval, we may arrange for this isotopy to take $\eta'$ and $\eta''$ to $\eta'$ and $\eta''$, respectively; this uses the fact that $\nu U \cong S^2 \times D^2$ is simply connected, and that every pair of embedded arcs in a 4-manifold which are homotopic to one another rel.\ endpoints are isotopic to one another rel.\ endpoints. Therefore, $\Sigma(\alpha;J,1)$ is $\CAT$-isotopic to the result of surgering $\S'$ along $\eta'$ and $\eta''$, which we already noted is isotopic to $\Sigma$. Thus  $\Sigma(\alpha;J,1)$ is $\CAT$-isotopic to $\S$ as desired.
\end{proof}


\begin{figure}
\labellist
  \pinlabel $\textcolor{blue}{D}$ at 35 84
  \pinlabel $\textcolor{blue}{D'}$ at 37.5 48
  \pinlabel $\textcolor{blue}{D''}$ at 38 121
\pinlabel $\textcolor{red}{\alpha}$ at -10 83
\pinlabel $U$ at 190 70
\pinlabel $\eta'$ at 153 36
\pinlabel $\eta''$ at 154 132
\endlabellist
\includegraphics[width=50mm]{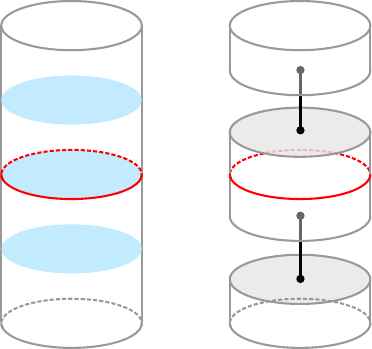}
\caption{  {Left:} a neighborhood of a framed disk $D$ with boundary $\alpha$ on $\S$. Since $D$ is framed, we may find two parallel copies $D'$ and $D''$ of $D$ on opposite sides of $D$ (in $\Sigma$) with boundary on $\Sigma$.   {Right:} compressing $\S$ along $D'$ and $D''$ yields a disconnected surface $\S'$ with an unknotted 2-sphere component $U$. Surgering $\S'$ along the arcs $\eta',\eta''$ (with framing chosen to yield an surface orientable in the pictured neighborhood of $D$) yields $\S$ again.}\label{fig:rimnothing}
\end{figure}

On the other hand, rim surgery does often change the isotopy type of a surface when working outside the setting of Lemma~\ref{lemma:when-rim-surgery-doesnt-change-anything}. One potential obstruction to smooth isotopy comes from the link cobordism maps of Section~\ref{sec:knotfloer}: if surfaces $\S_1$ and $\S_2$ in $X$ are isotopic rel.\ boundary, then $F_{X,\S_1}=F_{X,\S_2}$ (where we choose $w,z$ basepoints of $\partial \S_1=\partial \S_2$ to agree).

\begin{lemma}[{\cite[Theorem 1.1]{juhaszzemkeconcordancesurgery} (see also \cite[Theorem 5.1]{JuhaszMillerZemke})}]\label{lem:effectofrim}
Let $\S$ be a smooth surface properly embedded in $B^4$. 
Let $\alpha$ be an essential, non-separating simple closed curve on $\Sigma$ and $J$ a knot. Then 
\[F_{B^4,\S(\alpha;J)}(1)=\Delta_J(y_\alpha)F_{B^4,\S}(1)\ \in \widehat{\HFL}(S^3,\partial \Sigma).\]
In particular, if $F_{B^4,\S}$ is nonvanishing, and $J$ has nontrivial Alexander polynomial, then $\S$ and $\S(\alpha;J)$ are not smoothly isotopic rel.\ boundary. More generally, if $F_{B^4,\S}$ is nonvanishing and if $J_1$ and $J_2$ are knots with Alexander polynomials that are distinct up to multiplication with a monomial $\pm t^k$, then the surfaces $\S(\alpha; J_1)$ and $\S(\alpha;J_2)$ are not smoothly isotopic rel.\ boundary.

For $\partial \S$ connected and a non-negative integer $k$ 
equal to 
the number of irreducible factors of the mod 2 Alexander polynomial $\Delta_J(t) \in \mathbb{F}_2[t^{\pm 1}]$ of $J$  $($counted with multiplicity$)$, we conclude $\Omega(\S(\alpha;J))=\Omega(\S)+k$. 
\end{lemma}

\begin{remark}\label{rem:twistedfloer}
In order to parse Lemma \ref{lem:effectofrim}, we will discuss the totally twisted link Floer cobordism maps of \cite{juhaszzemkeconcordancesurgery} and \cite{JuhaszMillerZemke}. 

Given an $n$-tuple $\boldsymbol{\omega}=(\omega_1,\ldots,\omega_n)$ of closed 2-forms on a manifold $X$, we obtain an action of $C_2^{\operatorname{sm}}(X;\mathbb{Z})$ (the group 
of smooth 2-chains on $X$) on $\mathbb{R}^n$ defined by \[h\cdot (a_1,\ldots,a_n)=(a_1+\langle \omega_1,h\rangle,\ldots,a_n+\langle \omega_n,h\rangle).\] Note that for $h$ in the image of $\partial_3 \colon C_3^{\operatorname{sm}}(X;\mathbb{Z})\to C_2^{\operatorname{sm}}(X;\mathbb{Z})$, this is the trivial action, since \[\langle \omega_i,h\rangle = \langle \omega_i,\partial_3 g\rangle = \langle d\omega_i,g\rangle = \langle 0,g\rangle = 0\] if $h=\partial_3 g$ is a boundary. 
Thus, there is in fact an action of $\coker(\partial_3)$ 
on $\mathbb{R}^n$ induced by $\boldsymbol{\omega}$.

We want to extend the action of $C_2^{\operatorname{sm}}(X;\mathbb{Z})$ on $\R^n$ to an action of  $\mathbb{F}_2[C_2^{\operatorname{sm}}(X;\mathbb{Z})]$ on $\F_2[\mathbb{R}^n]$, and then consider the induced action of $\mathbb{F}_2[\coker(\partial_3)]$ on $\F_2[\mathbb{R}^n]$.
Let $\{y_1,\dots,y_n\}$ be the standard basis of $\R^n$, and write the group structure of $\R^n$  multiplicatively.  Thus $y_1^{a_1}\cdots y_n^{a_n}$ denotes the element of $\mathbb{F}_2[\mathbb{R}^n]$ corresponding to $(a_1,\ldots, a_n) \in \R^n$. Similarly, write $e^h$ to indicate the element of $\mathbb{F}_2[\coker(\partial_3)]$ corresponding to $h \in \coker (\partial_3)$. Then we can extend the action above to the promised action on $\F_2[\mathbb{R}^n]$ via  
\[e^h\cdot y_1^{a_1}\cdots y_n^{a_n}=y_1^{a_1+\langle \omega_1,h\rangle}\cdots y_n^{a_n+\langle \omega_n,h\rangle}.\] We write $\mathbb{F}_2[\mathbb{R}^n]_{\boldsymbol{\omega}}$ to indicate $\mathbb{F}_2[\mathbb{R}^n]$ considered as an $\mathbb{F}_2[\coker(\partial_3)]$-module equipped with the action determined by $\boldsymbol{\omega}$.


Juh\'asz-Zemke \cite{juhaszzemkeconcordancesurgery} give a way of twisting the sutured Floer homology of a sutured manifold $(M,\gamma)$ (e.g.\ the complement of $\partial \S$ in $S^3$) via an $n$-tuple $\boldsymbol{\omega}$ of closed 2-forms on $M$. This is relevant to our setting because the sutured Floer homology of 0-surgery on a link in $S^3$ is tautologically equal to the hat version of link Floer homology of that link, so this might be viewed as another perspective on link Floer homology.

Let $(M,\gamma)$ be a balanced sutured 3-manifold and let $(S,\alpha,\beta)$ be an admissible sutured Heegaard decomposition for $(M,\gamma)$ in the sense of \cite{juhasz06}. 
The twisted sutured Floer chain complex 
${\underline{CF}}(S,\alpha,\beta,\mathfrak{s};\mathbb{F}_2[\coker(\partial_3)])$ is generated by $\mathbf{x}\otimes e^h$ across all $h\in \coker(\partial_3)$ and $x\in \mathbb{T}_\alpha\cap \mathbb{T}_\beta$ with $s_z(\mathbf{x})=\mathfrak{s}$. 
Here, $\mathfrak{s}$ is a relative spin$^{c}$ structure on $M$, and $\mathbb{T}_\alpha$ and $\mathbb{T}_\beta$ are the totally real tori in the complex manifold $\operatorname{Sym}^{\operatorname{genus}(S)} (S)$ (of real dimension $2 \cdot \operatorname{genus}(S)$), which is the $\operatorname{genus}(S)$-fold symmetric product of $S$ with itself used in the construction of (every flavor of) Heegaard-Floer homology.


The differential on the twisted sutured Floer complex is given by \[\partial(x\otimes e^h)=\sum_{\mathbf{y}\in\mathbb{T}_\alpha\cap\mathbb{T}_\beta}\sum_{\substack{\phi\in\pi_2(x,y)\\\mu(\phi)=1}}\#(\mathcal{M}(\phi)/\mathbb{R})\cdot\mathbf{y}\otimes e^{\widetilde{D}(\phi)}\cdot e^h,\] where $\widetilde{D}(\phi)\in \coker(\partial_3)$ is obtained from the domain of $\phi$ by smoothly capping off with compressing disks for $\alpha,\beta$ \cite{juhaszzemkeconcordancesurgery}. 
Let \[ 
\underline{CF}(S,\alpha,\beta;\mathbb{F}_2[\coker(\partial_3)]):=\bigoplus_{\mathfrak{s}\in\spin^c(M)} \underline{CF}(S,\alpha,\beta,\mathfrak{s};\mathbb{F}_2[\coker(\partial_3)]).\] 

Now we set \[CF(S,\alpha,\beta,\mathfrak{s};\mathbb{F}_2[\mathbb{R}^n]_{\boldsymbol{\omega}}):=\underline{CF}(S,\alpha,\beta,\mathfrak{s};\mathbb{F}_2[\coker(\partial_3)])\otimes_{\mathbb{F}_2[\coker(\partial_3)]} \mathbb{F}_2[\mathbb{R}^n]_{\boldsymbol{\omega}}.\] 
This is the \emph{sutured Floer chain complex perturbed by $\boldsymbol{\omega}$}.
Analogously to above, we also define:
\[CF(S,\alpha,\beta;\mathbb{F}_2[\mathbb{R}^n]_{\boldsymbol{\omega}}):=
\bigoplus_{\mathfrak{s}\in\spin^c(M)} CF(S,\alpha,\beta,\mathfrak{s};\mathbb{F}_2[\mathbb{R}^n]_{\boldsymbol{\omega}}).\] 

Choosing a different admissible Heegaard diagram for $(M,\gamma)$ preserves the chain complex  $\underline{CF}(S,\alpha,\beta;\mathbb{F}_2[\coker(\partial_3)])$ up to chain homotopy equivalence and the action of $\mathbb{F}_2[\coker(\partial_3)]$. (Baldwin and Sivek \cite{baldwinsivekcontact} call this a {\emph{projective transitive system.}}) See Section 6 of \cite{juhaszzemkeconcordancesurgery} to understand transition maps. We denote the homology of this chain complex by  \[\underline{SHF}(M,\gamma;\mathbb{F}_2[\coker(\partial_3)]),\] and  we write \[SHF(M,\gamma;\mathbb{F}_2[\mathbb{R}^n]_{\boldsymbol{\omega}})\] for the homology of $CF(S,\alpha,\beta;\mathbb{F}_2[\mathbb{R}^n]_{\boldsymbol{\omega}})$.

Back in the 4-dimensional world, given a suitably decorated sutured cobordism $W$ from $(M,\gamma)$ to $(M',\gamma')$, Juh\'asz and Zemke \cite{juhaszzemkeconcordancesurgery} construct a twisted map
\[\underline{F}_W \colon \underline{SHF}(M,\gamma;\mathbb{F}_2[\coker(\partial_3))])\to \underline{SHF}(M',\gamma';\mathbb{F}_2[\coker(\partial'_3)]).\]  
This procedure involves decomposing $W$ into elementary pieces and defining a map associated to each piece; we refer the interested reader to Section 7 of \cite{juhaszzemkeconcordancesurgery}. Finally, we define \[F_{W;\boldsymbol{\omega}} \colon SHF(M,\gamma;\mathbb{F}_2[\mathbb{R}^n]_{\boldsymbol{\omega}})\to SHF(M',\gamma';\mathbb{F}_2[\mathbb{R}^n]_{\boldsymbol{\omega}}),\]
to be given by \[F_{W;\boldsymbol{\omega}}=\underline{F}_W\otimes 1_{\mathbb{F}_2[\mathbb{R}^n]_{\boldsymbol{\omega}}}.\] This means that if $\underline{F}_W(1)$ has a term of the form $[x]\cdot e^h$, then $F_{W;\omega}(1)$ has a corresponding term $[x]\cdot y_1^{\langle \omega_1,h\rangle}\cdots y_n^{\langle \omega_n,h\rangle}$. Thus, $F_{W;\boldsymbol{\omega}}(1)$ is a polynomial in $y_1,\ldots, y_n$, albeit with real exponents.

In the setting of our paper, let $\boldsymbol{\omega}$ be an $n$-tuple 
(with $n$ the sum of the genera of the connected components of $\S$) of closed 2-forms on $X:=B^4\setminus\nu\S$, so $F_{B^4,\Sigma;\boldsymbol{\omega}}=F_{X;\boldsymbol{\omega}}$. Then $F_{B^4,\Sigma;\boldsymbol{\omega}}(1)$ is a polynomial in $y_1,\ldots, y_n$ (although again note that we mean with real-valued exponents and coefficients in $\widehat{\HFL}(S^3,\partial \S)$). Lemma \ref{lem:effectofrim} says that \[F_{B^4,\S(\alpha;J);\boldsymbol{\omega}}(1)=\Delta_J(y_1^{\langle\omega_1,[T_\alpha]\rangle}\cdots y_n^{\langle\omega_n,[T_\alpha]\rangle})\cdot F_{B^4,\S;\boldsymbol{\omega}}(1),\] where  
$\Delta_J$ is the mod 2 Alexander polynomial of $J$ and $T_\alpha$ is the rim 2-torus $\partial\nu\S|_{\alpha}$ centered about~$\alpha$.

Generally, it does not make sense to count factors of polynomials with real-valued exponents. However, \cite[Proposition 4.3]{JuhaszMillerZemke} shows that there exists a monomial $m \in \mathbb{F}_2[\R^n]$ such that $m\cdot F_{B^4,\S;\boldsymbol{\omega}}(1)\in \widehat{\HFL}(S^3,\partial\S)\otimes\mathbb{F}_2[\mathbb{Z}^n]$. Then $\Omega_{\boldsymbol{\omega}}(\S)$ is defined to be the largest number of irreducible factors of any element of $\mathbb{F}_2[\mathbb{Z}^n]$ dividing $m\cdot F_{B^4,\S;\boldsymbol{\omega}}(1)$, so in this language it is clear that $\Omega_{\boldsymbol{\omega}}(\S(\alpha;J))=\Omega_{\boldsymbol{\omega}}(\S)+k$, where $k$ is the number of irreducible factors of $\Delta_J(t)$ as a polynomial in $\F_2[t^{\pm}]$. Finally \cite[Section 6.2]{JuhaszMillerZemke} shows that $\Omega_{\boldsymbol{\omega}}$ is independent of the choice of $\boldsymbol{\omega}$ and set $\Omega:=\Omega_{\boldsymbol{\omega}}$. See \cite{juhaszzemkeconcordancesurgery} and \cite{JuhaszMillerZemke} for more details.
\end{remark}


\subsection{Construction of exotic surface links}\label{sec:infinitefamily}

We now construct an infinite family of exotic 2--component Brunnian links. Each link will have a disk component and a genus one component. Fix~$K$ to be some strongly quasipositive topologically slice knot, e.g.\ the positive untwisted Whitehead double of the right-handed trefoil. This is topologically slice by \cite[Theorem~11.7B]{FreedmanQuinn} (as the Alexander polynomial of any untwisted Whitehead double is trivial), and strongly quasipositive by \cite{rudolphsqp}.

Let $\S=S_1\sqcup S_2$ be the surface link depicted in Figure~\ref{fig:infinitefamily}. We draw a disk and a genus one surface immersed in $S^3$ with ribbon intersections. Pushing the interiors of these surface slightly into $B^4$ yields disjoint surfaces. Moreover, each knot $\partial S_i$ is an unknot, $S_1$ is an unknotted disk, and $S_2$ is an unknotted genus one surface.

\begin{proposition}\label{prop:getdisks}
Let $\alpha$ be the curve on $S_2$ pictured in Figure~\ref{fig:infinitefamilytopslicedisk}. Then $\alpha$ bounds framed disks $\Delta$ and $\Delta'$ into the complement of $S_2$, with $\Delta$ smooth and $\Delta'$ locally flat, and with $\Delta'$ disjoint from $S_1$. Moreover, $\Delta$ and $\Delta'$ can be taken to agree near their common boundary~$\alpha$.
\end{proposition}

\begin{proof}
In the left two frames of Figure~\ref{fig:infinitefamilytopslicedisk}, we illustrate that $\alpha$ bounds a framed locally flat disk $\Delta'$ whose interior is disjoint from both $S_2$ and $S_1$. The construction of the disk is explained in detail in the caption of Figure~\ref{fig:infinitefamilytopslicedisk}. 
This disk is obtained from the two schematically pictured locally flat slice disks for the topologically slice knot $K$ by gluing them along the arc where their boundaries coincide. 

Viewed as a knot in $S^3$ on the ribbon surface $S_2$, $\alpha$ is an unknot and the framing induced by $S_2$ on $\alpha$ is the 0-framing. We conclude that $\alpha$ bounds a framed smooth disk $\Delta$ into $B^4$ that is disjoint from $S_2$ in its interior. Moreover, as drawn, $\Delta'$ induces the 0-framing on $\alpha$ as a knot in $S^3$, so we can arrange for $\Delta$ and $\Delta'$ to agree near their boundaries. 
\end{proof}

Now that we have specified framed disks (that agree near their boundaries) bounded by $\alpha$, we can consider $n$-twisted rim surgery on $\alpha$.

\begin{corollary}\label{cor:gettopiso}
For any knot $J$, $S_2(\alpha;J,1)$ is smoothly isotopic rel.\ boundary to $S_2$ and $\S(\alpha;J,1)$ is topologically isotopic rel.\ boundary to $\S$.
\end{corollary}

\begin{proof}
This follows immediately from Proposition~\ref{prop:getdisks} and Lemma~\ref{lemma:when-rim-surgery-doesnt-change-anything}.
\end{proof}

\begin{figure}
\labellist
\pinlabel \red{parallel locally flat disks for $K$, whose union} at 315 10
\pinlabel \red{along the upper horizontal arc forms $\Delta'$} at 315 -15
\endlabellist
\includegraphics[width=120mm]{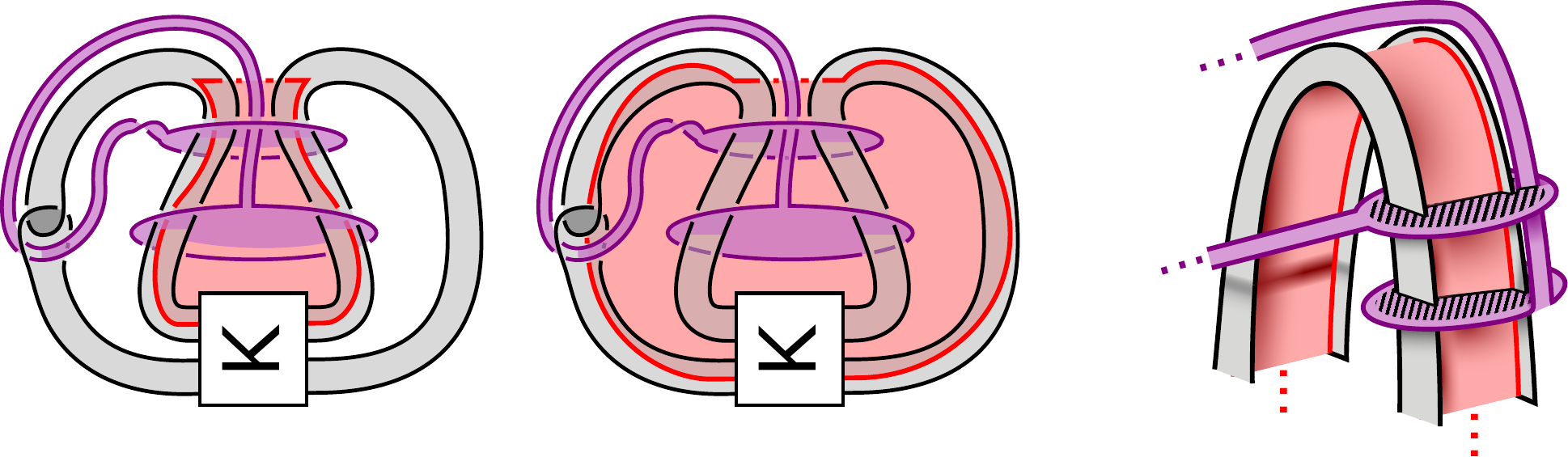}
\vspace{.1in}
\caption{
{Left two panels:} because $K$ is topologically slice, $\alpha$ bounds a topological locally flat disk $\Delta'$ whose interior is disjoint from both $S_2$ and $S_1$ (and which is normal to $S_2$ near its boundary). The disk is obtained by gluing the two indicated locally flat disks together along the two planar surfaces indicated, which are themselves glued along an arc. The visible portion of $\Delta'$ (projected to $S^3$) intersects $S_1$ in two ribbon arcs. {Right panel:} A close-up of a portion of the first two figures. In the projection to $S^3$, the indicated portion of $\Delta'$ intersects $S_1$ in two ribbon intersections. The projections of $S_1,S_2$ also have ribbon intersections, and the projection of $S_1$ has a ribbon self-intersection. In $B^4$, the shaded regions of $S_1$ lie further toward the center of $B^4$, so that the interiors of all surface are disjoint.}
\label{fig:infinitefamilytopslicedisk}
\end{figure}

In order to apply Lemma~\ref{lem:effectofrim} and obstruct smooth isotopy rel.\ boundary of the surface links, we must show that the map $F_{B^4,L}:\widehat{\HFK}(\emptyset)\to\widehat{\HFK}(S^3,\partial \S)$ is nonvanishing. We achieve this by obtaining a strongly quasipositive surface from $\S$, even though $\S$ itself is {\emph{not}} strongly quasipositive (since $S_2$ is not).

\begin{figure}
\labellist
  \pinlabel $\blue{b}$ at 60 70
\endlabellist
\includegraphics[width=90mm]{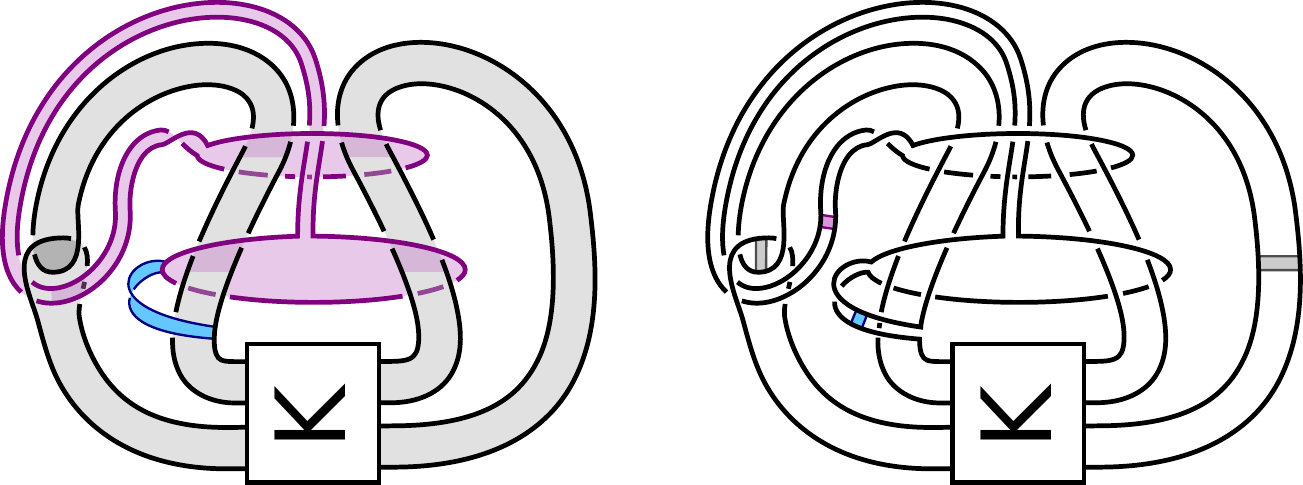}
\caption{  {Left:} we form a connected surface $S$ from $S_2$ and $S_1$ via attaching the band $b$ pictured.   {Right:} We draw the knot $\partial S$ with four bands contained in $S$ representing index-1 critical points of the radial height function on $B^4$ restricted to $S$. Surgering $\partial S$ along these bands yields an unlink; disks bounded by this unlink give the minima of $S$.} \label{fig:infinitefamilyband}
\end{figure}

\begin{proposition}\label{prop:bandgetstd}
Let $S$ be the connected surface obtained from $\Sigma$ by attaching $S_2$ and $S_1$ via the band depicted in the left of Figure~\ref{fig:infinitefamilyband} $($with interior pushed slightly into $B^4)$. $S$ is smoothly isotopic to the standard genus one Seifert surface for $\Wh_+(\Wh_+(K))$ $($with interior pushed slightly into $B^4)$.
\end{proposition}

\begin{proof}
In Figure~\ref{fig:infinitefamilyisotopy}, we isotope $S$ to see $\partial S=\Wh_+(\Wh_+(K))$ (see also Figure \ref{fig:whwhk}). The surface $S$ is embedded in $B^4$ with four index-1 and three index-0 points with respect to the standard radial height function on $B^4$. The index-1 points are flattened and drawn as bands attached to $\partial S$; surgering $\partial S$ along these bands yields a 3--component unlink. Moreover, two of these bands, labeled $b_1$ and $b_2$, lie on the standard genus one Seifert surface for $\Wh_+(\Wh_+(K))$. After surgering $\partial S$ along these bands, we obtain an unknot with two trivial bands attached. We conclude that the remaining two bands (i.e.\ index-1 points of $S$) can be removed via isotopy of $S$. Thus, $S$ is isotopic to the standard genus one Seifert surface for $\Wh_+(\Wh_+(K))$.
\end{proof}

\begin{figure}
\labellist
  \pinlabel {slide bands} at 640 500
  \pinlabel {$b_1$} at 365 135
  \pinlabel {$b_2$} at 640 105
\endlabellist
\includegraphics[width=120mm]{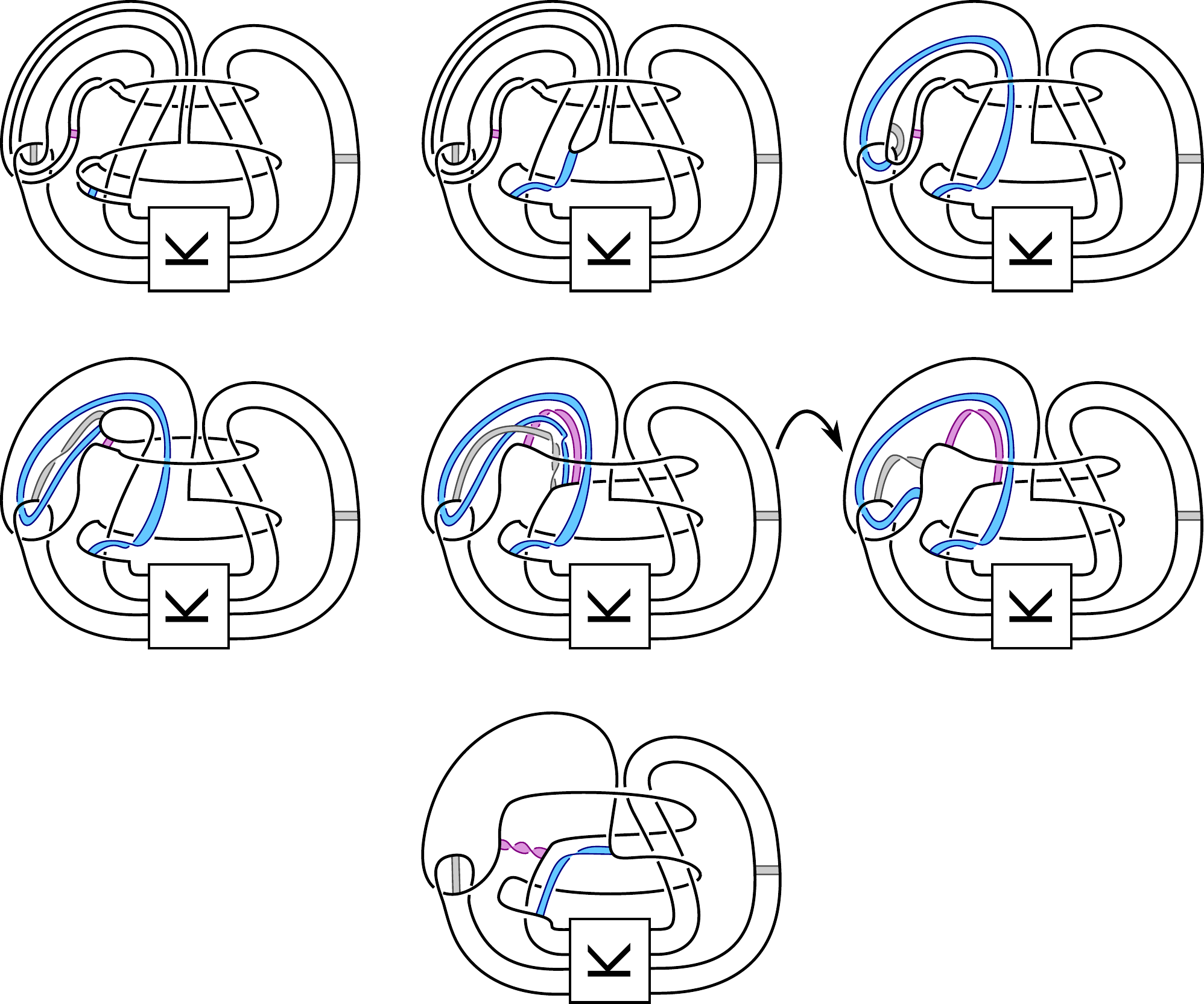}
\caption{Left to right, top to bottom: we isotope the diagram for $S$ obtained in Figure~\ref{fig:infinitefamilyband}. Between the fifth and sixth frame, we slide two bands over another band. Isotopy and slides of bands describe isotopy of $S$. In the final diagram, we see that $\partial S=\Wh_+(\Wh_+(K))$ and note that the bands $b_1$ and $b_2$ lie on the standard genus one surface for $\Wh_+(\Wh_+(K))$.} \label{fig:infinitefamilyisotopy}
\end{figure}

\begin{figure}
\labellist
  \pinlabel {$\Wh_+(K)$} at 60 -15
  \pinlabel {$\Wh_+(\Wh_+(K))$} at 210 -15
  \pinlabel {$\Wh_+(\Wh_+(K))$} at 362 -15
\endlabellist
    \includegraphics[width=120mm]{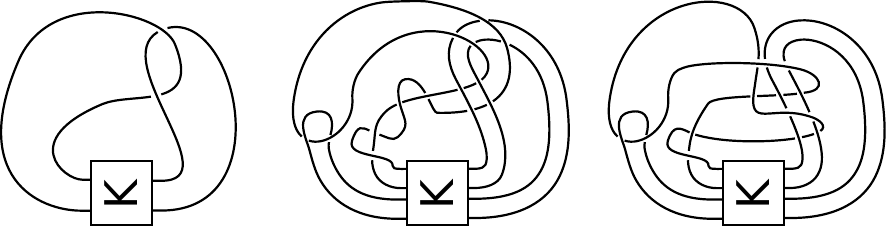}
    \vspace{.2in}
    \caption{  {Left:} The knot $\Wh_+(K)$. Note that this diagram has writhe 2.   {Middle:} The knot $\Wh_+(\Wh_+(K))$.   {Right:} we isotope the middle drawing of $\Wh_+(\Wh_+(K))$ to agree with $\partial S$ as in the last panel of Figure \ref{fig:infinitefamilyisotopy}.}
    \label{fig:whwhk}
\end{figure}

\begin{figure}
\labellist
 \pinlabel {resolve} at 340 195
 \pinlabel {two} at 340 165
 \pinlabel {bands} at 340 135
\endlabellist
\includegraphics[width=120mm]{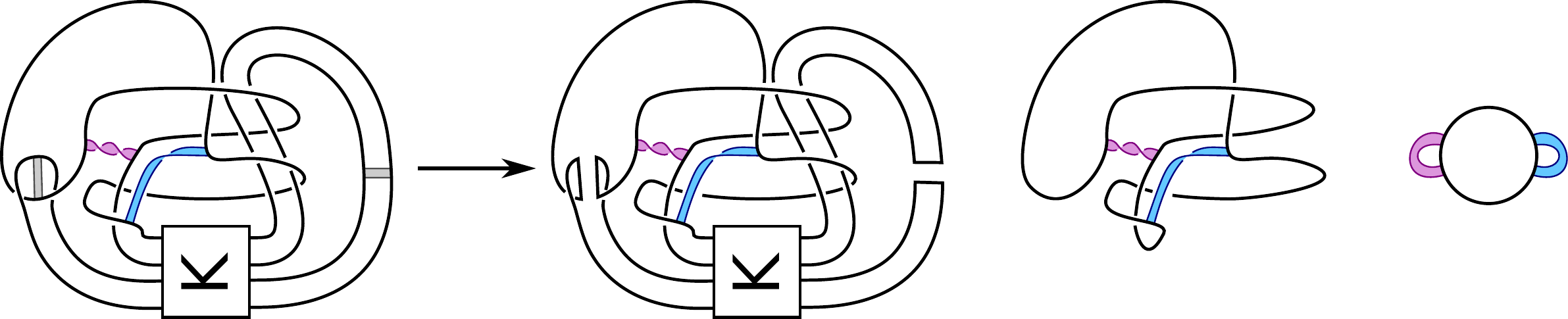}
\caption{Starting with the last diagram of Figure~\ref{fig:infinitefamilyisotopy}, we surger $\partial S$ along the two bands lying on the standard surface for $\Wh_+(\Wh_+(K))$. We then isotope the result to obtain the diagram on the right: two trivial bands attached to an unknot. We conclude that the remaining two bands are cancellable, and therefore $S$ is isotopic to the standard genus one surface for $\Wh_+(\Wh_+(K))$. \label{fig:infinitefamilycancel}}
\end{figure}

\begin{corollary}\label{cor:nonvanish}
The map $F_{B^4,\Sigma}\colon \widehat{\HFK}(\emptyset)\to\widehat{\HFK}(S^3,\partial \Sigma)$ is nonvanishing.
\end{corollary}

\begin{proof}
Let $B$ denote the link cobordism described by the band $b$, as shown on the left of Figure~\ref{fig:infinitefamilyband}. This $B$ is a cobordism from $\partial \S$ to $\Wh_+(\Wh_+(K))$. By Proposition~\ref{prop:bandgetstd}, $B\circ \S$ is isotopic to the standard genus one Seifert surface $S$ for $\Wh_+(\Wh_+(K))$. By Remark~\ref{rem:usefulprop}~\eqref{prop1}, we have \[F_{B^4,S}=F_{S^3\times I, B}\circ F_{B^4,\S}.\]  Since $K$ is a strongly quasipositive knot, $S$ is strongly quasipositive surface \cite{rudolphsqp}. By Remark~\ref{rem:usefulprop}~\eqref{prop2}, $F_{B^4,S}$ is nonvanishing. Thus, $F_{B^4,\S}$ is also nonvanishing.
\end{proof}

\begin{thm:intro-surface-links}
Let $K$ be a strongly quasipositive topologically slice knot and let $\S_0$ denote the surface link of Figure~\ref{fig:infinitefamily}. 
For each integer $n \geq 1$, let $J_n \subseteq S^3$ be a knot whose Alexander polynomial has $n$ irreducible factors $($counted with multiplicity$)$. 
Let $\S_n:=\S_0(\alpha;J_n,1)$, where $\alpha$ is the curve on $S_2$ illustrated in Figure~\ref{fig:infinitefamilytopslicedisk}. Then each $\S_n$ is Brunnian. Additionally, for $n\neq m$, we find that $\S_n$ and $\S_m$ form an exotic pair.
\end{thm:intro-surface-links}

The positive untwisted Whitehead double of the right-handed trefoil is one possible choice for $K$. The knot $J_n$ may, for example, be a connect sum of $n$ trefoils.

\begin{remark}\label{remark:distinct-alex-polys}
In Theorem~\ref{thm:intro-surface-links}, if we ask only that for $n\neq m$, $J_n$ and $J_m$ have inequivalent mod 2 Alexander polynomials, up to multiplication by a unit in $\mathbb{F}_2[t^{\pm 1}]$, rather than that their mod 2 Alexander polynomials have different numbers of irreducible factors, then we may still conclude that $\S_n$ and $\S_m$ are not smoothly isotopic rel.\ boundary by applying Lemma~\ref{lem:effectofrim} and Corollary~\ref{cor:nonvanish}. But, for us at least (and also for \cite{JuhaszMillerZemke}), the invariant $\Omega$ is needed to obstruct diffeomorphism of pairs.
\end{remark}

%

\begin{proof}[Proof of Theorem \ref{thm:intro-surface-links}]
Note that $\S_0$ is Brunnian by construction.  It follows from Corollary~\ref{cor:gettopiso} that each $\S_n$ is Brunnian, and that for all $n$,  $\S_n$ is topologically isotopic rel.\ boundary to $\S_0$.

Assume $n<m$. Let $S$ denote the genus one Seifert surface for $\Wh_+(\Wh_+(K))$ with interior pushed slightly into $B^4$. By Proposition~\ref{prop:bandgetstd}, there is a band $b$ so that $\S_n\cup b$ is isotopic rel.\ boundary to $S(\alpha;J_n,1)$. Applying Remark~\ref{rem:usefulprop}~\eqref{prop2} and Lemma~\ref{lem:effectofrim}, we deduce that 
$\Omega(S(\alpha;J_n,1))=n$. Suppose there is a diffeomorphism $f\colon (B^4,\S_n)\to (B^4,\S_m)$. We cannot assume that $f$ restricts to the identity on $\partial B^4$, so this does not yield an equivalence from $S(\alpha;J_n,1)$ to $S(\alpha;J_m,1)$. Instead, we obtain an equivalence from $(B^4,S(\alpha;J_n,1))$ to $(B^4,f(b)\cup \S_m)$, where $f(b)$ is some other band attached to $\partial \S_m$. 
Note that $f(b)\cup \S_m$ is obtained from $f(b)\cup \S_0$ by $1$-twist rim surgery along $\alpha$ using $J_m$. Applying Lemma~\ref{lem:effectofrim}, we have either 
\[\begin{array}{rcll}
\Omega(f(b)\cup \S_m)&=&-\infty &\text{if $F_{B^4,f(b)\cup \S_0}$ vanishes, or}\\
\Omega(f(b)\cup \S_m) &\ge & m &\text{if $F_{B^4,f(b)\cup \S_0}$ does not vanish}.
\end{array}\]

In both cases, $\Omega(f(b)\cup \S_m)\neq \Omega(S(\alpha;J_n,1))=n$, yielding a contradiction. We conclude that there is no diffeomorphism from $(B^4,\S_n)$ to $(B^4,\S_m)$.
\end{proof}

\section{Surfaces of higher genus}\label{section:surfaces-of-higher-genus}

Next, we extend the construction of Section~\ref{sec:infinitefamily} to produce higher genus surfaces, proving Theorem~\ref{thm:increasing-the-genus}. From now on, let $\S_0=S_1\cup S_2$ be as in Theorem~\ref{thm:intro-surface-links} (i.e.\ the surface link in Figure~\ref{fig:infinitefamily}, with $K$ a topologically slice, strongly quasipositive knot). Recall that $S_1$ has genus zero and $S_2$ has genus one. In shorthand, we write that $\S_0$ has genus $(0,1)$.

In order to use $\S_0$ to produce exotic surfaces of higher genus, first we  describe a certain high genus Brunnian link. Fix a positive integer $g$ and let $R^g=R^g_1\sqcup R^g_2$ be the surface link in $B^4$ illustrated in Figure~\ref{fig:linkR}. We also draw a band $b^g_R$ between the two components of $R^g$. The link $R^g$ is Brunnian of genus $(0,g)$.

\begin{figure}
\labellist
 \pinlabel \textcolor{darkpurple}{$R^g_1$} at 60 10
 \pinlabel $R^g_2$ at -10 100
 \pinlabel \textcolor{blue}{$b^g_R$} at 5 40
\pinlabel $\cdots$ at 249 110
\pinlabel {genus $g$} at 249 95
\pinlabel $\cdots$ at 249 15
\endlabellist
\includegraphics[width=100mm]{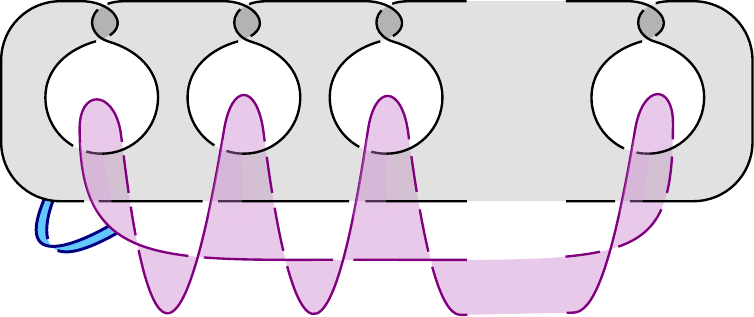}
\caption{A 2--component surface link $R^g=R_1^g\sqcup R_2^g$ and a band $b^g_R$ between the components of $R^g$. Note that $R_1^g$ is an unknotted disk and $R^g_2$ is an unknotted genus $g$ surface. Viewed as a ribbon surface in $S^3$, there are $g$ arcs of intersection between the pictured ribbon surfaces.}\label{fig:linkR}
\end{figure}

\begin{proposition}\label{prop:gettrefoil}
The surface $\widetilde{R}^g$ obtained by gluing $R^g_1$ and $R^g_2$ together along the band $b^g_R$, as in Figure~\ref{fig:linkR}, is isotopic to a quasipositive Seifert surface $($pushed slightly into $B^4)$.
\end{proposition}

\begin{proof}
In Figure~\ref{fig:linkR_isotopy}, we illustrate an isotopy from $\widetilde{R}^g$ to a quasipositive Seifert surface (slightly pushed into $B^4$) for a connected sum of $g$ right-handed trefoil knots.
\end{proof}

\begin{figure}
\labellist
\pinlabel {\LARGE$\cdots$} at 250 690
\pinlabel {\LARGE$\cdots$} at 250 495
\pinlabel {\LARGE$\cdots$} at 250 300
\pinlabel {\LARGE$\cdots$} at 250 105

\pinlabel {\LARGE$\cdots$} at 250 0
\pinlabel {\LARGE$\cdots$} at 250 150
\pinlabel {\LARGE$\cdots$} at 250 195
\pinlabel {\LARGE$\cdots$} at 250 235
\pinlabel {\LARGE$\cdots$} at 250 345
\pinlabel {\LARGE$\cdots$} at 250 390
\pinlabel {\LARGE$\cdots$} at 250 415
\pinlabel {\LARGE$\cdots$} at 250 445
\pinlabel {\LARGE$\cdots$} at 250 539
\pinlabel {\LARGE$\cdots$} at 250 583
\pinlabel {\LARGE$\cdots$} at 250 610
\pinlabel {\LARGE$\cdots$} at 250 639
\pinlabel {\LARGE$\cdots$} at 250 734
\endlabellist
\includegraphics[width=90mm]{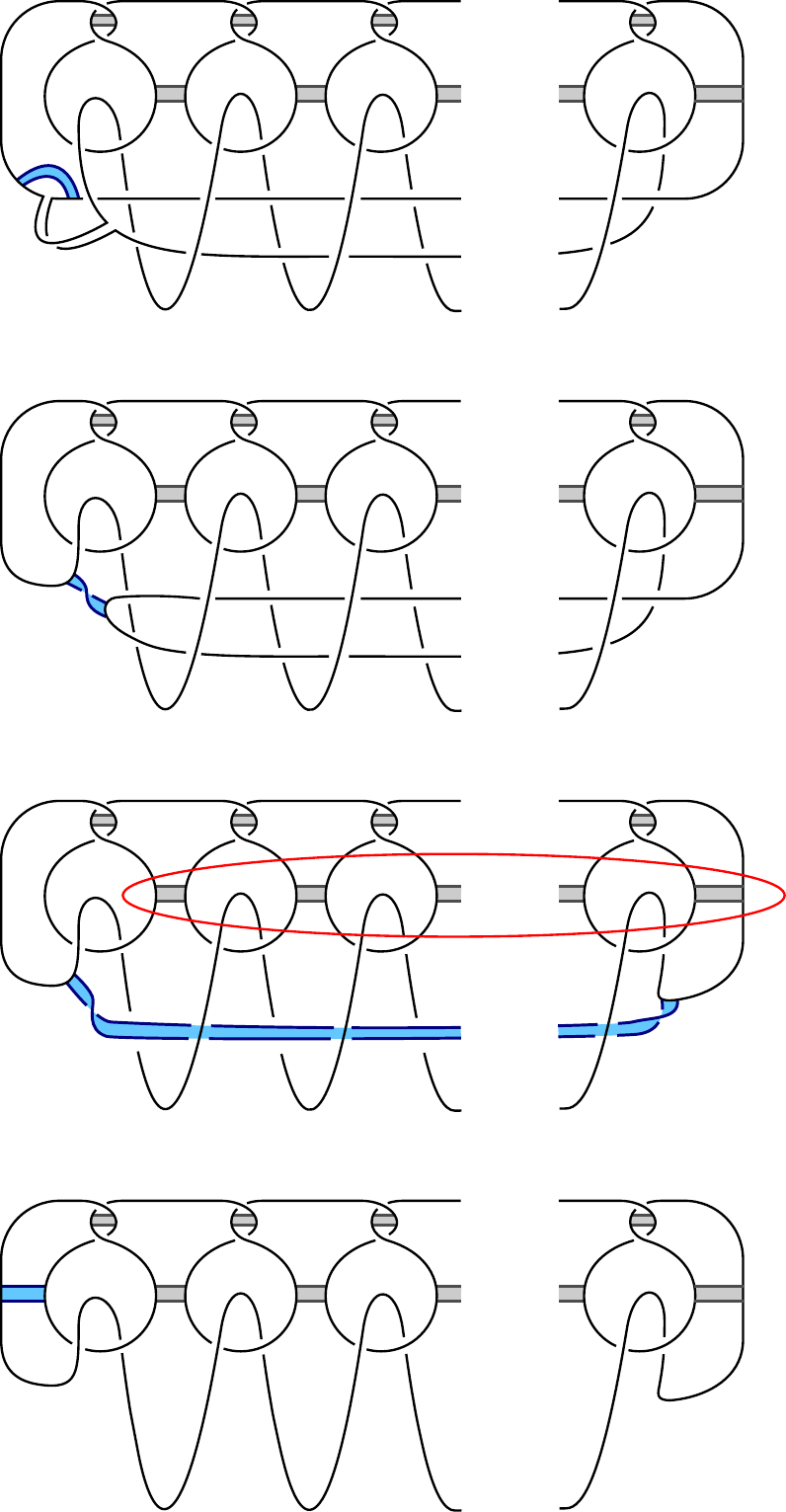}
\caption{We draw the boundary of $\widetilde{R}^g$ with bands indicating index-1 critical points into $B^4$. From top to bottom, we isotope $\widetilde{R}^g$. To get from the third to the fourth picture, we also slide the bottommost band over all of the circled bands. In the end, we obtain a positive genus $g$ Seifert surface (slightly pushed into $B^4$) for a connected sum of $g$ right-handed trefoil knots.} \label{fig:linkR_isotopy}
\end{figure}

We use $R^g$ to increase the genus of $\S_0$. For any integers $r,s$ with $r\ge0$ and $s\ge 1$, let $\S^{r,s}_0$ be obtained as follows (see Figure~\ref{fig:genusmn} (top)).

\begin{enumerate}[1.]
\item Start with a copy of $\S_0$.
\item If $r>0$, then band a copy of $R^r$ to $\S_0$. The bands are indicated in Figure~\ref{fig:genusmn} (top); note that the disk component of $R^r$ is banded to $S_2$ while the positive genus component of $R^r$ is banded to $S_1$.
\item If $s>1$, then band a copy of $R^{s-1}$ to $\S_0$. The bands are again indicated in Figure~\ref{fig:genusmn} (top); note that the disk component of $R^{s-1}$ is banded to $S_1$ while the positive genus component of $R^{s-1}$ is banded to $S_2$.
\item Call the result genus $(r,s)$ surface $\S^{r,s}_0$.
\end{enumerate}

\begin{figure}
\labellist
 \pinlabel $R^{s-1}$ at 500 535
 \pinlabel $R^{r}$ at 500 380
 \pinlabel \textcolor{red}{$\alpha$} at 15 70
 \pinlabel \textcolor{blue}{$b$} at 70 145
\endlabellist
\includegraphics[width=115mm]{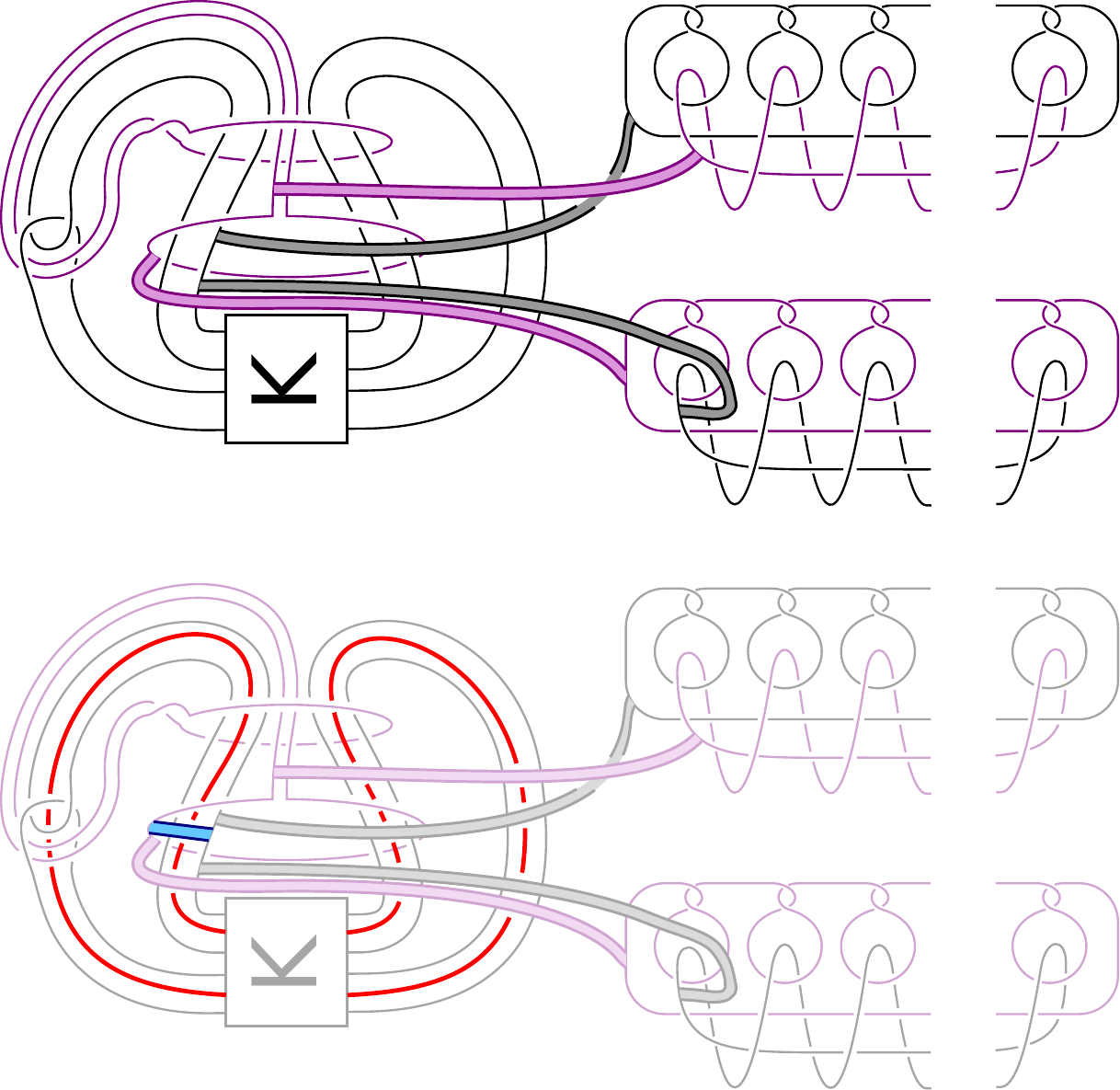}
\caption{Top: The 2--component surface link $\S^{r,s}_0$ (here we draw as if $r>0$ and  $s>1$). Bottom: a band $b$ between the components of $\S^{r,s}_0$ and a curve $\alpha$ lying on $\S^{r,s}_0$.} \label{fig:genusmn}
\end{figure}

In Figure~\ref{fig:genusmn} (bottom), we illustrate a certain band $b$ between the two components of $\S^{r,s}_0$ and a curve $\alpha$ (the same as in Theorem~\ref{thm:intro-surface-links}) that lies on $\S^{r,s}_0$. Just as in Proposition~\ref{prop:getdisks}, we see that $\alpha$ bounds a smooth framed disk into the complement of the component on which it lies, and a locally flat framed disk into the complement of $\S^{r,s}_0$. The following proposition illustrates why we chose the specific bands in Figure~\ref{fig:genusmn} (top).

\begin{proposition}\label{prop:genussqp}
The surface obtained by gluing the components of $\S^{r,s}_0$ together along $b$ is isotopic to a quasipositive Seifert surface $($pushed slightly into $B^4)$.
\end{proposition}

\begin{proof}
In Figure~\ref{fig:genusmn_isotopy}, we isotope the result of gluing $\S^{r,s}_0$ along $b$ to obtain a copy of $\S_0\cup b$ trivially boundary-summed with $R^r\cup b^r_R$ and $R^{s-1}\cup b^{s-1}_R$. (If $r=0$ or $s=1$ then ignore the corresponding $R$.) From Propositions~\ref{prop:bandgetstd} and~\ref{prop:gettrefoil}, we find that this is a quasipositive Seifert surface (pushed into $B^4$) for $\Wh_+(\Wh_+(K))\#(\#_{r+s-1}($right-handed trefoils$))$.
\end{proof}

\begin{figure}
\includegraphics[width=110mm]{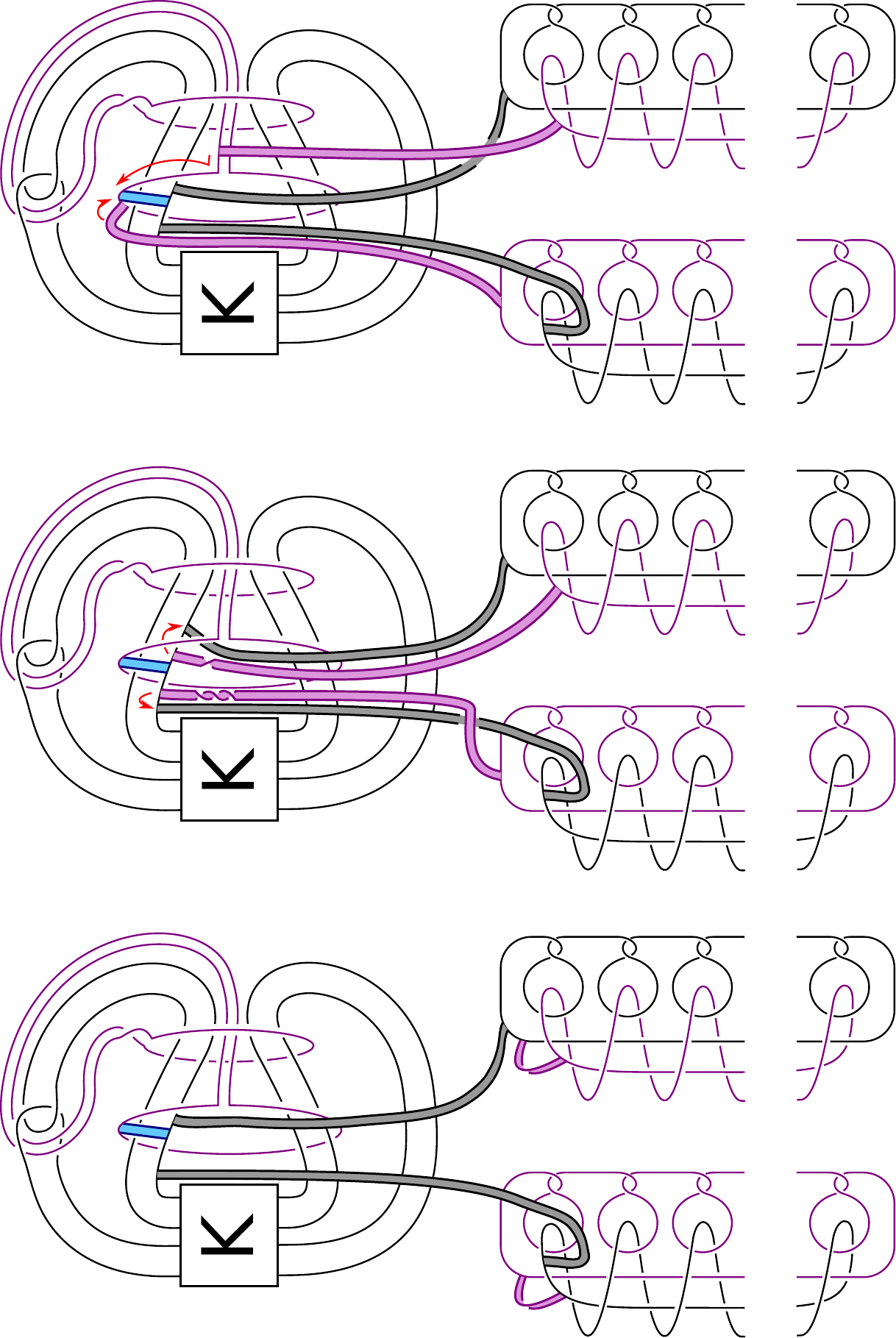}
\caption{We glue the components of $\S^{r,s}_0$ together along the band $b$. From top to bottom, we perform isotopy via band slides. We draw arrows to preemptively indicate band slides. Note at the end the resemblance to Figures~\ref{fig:infinitefamilyisotopy} and~\ref{fig:linkR_isotopy}.} \label{fig:genusmn_isotopy}
\end{figure}

Theorem~\ref{thm:increasing-the-genus} follows from the next statement.

\begin{theorem}\label{thm:increasing-the-genus-main-body}
Fix integers $r,s$ with $r\ge 0$ and $s\ge 1$. For each $n\in\N$, let $J_n$ be a knot in $S^3$ whose Alexander polynomial has $n$ irreducible factors counted with multiplicity. Let $\S^{r,s}_n=\S^{r,s}_0(\alpha;J_n,1)$. Then for all $m>n\ge 0$, we have that:
\begin{itemize}
\item $\S^{r,s}_n$ is a Brunnian genus $(r,s)$ link,
\item $\S^{r,s}_n$ and $\S^{r,s}_m$ are topologically isotopic rel.\ boundary, and
\item $(B^4,\S^{r,s}_n)$ is not diffeomorphic to $(B^4,\S^{r,s}_m)$.
\end{itemize}
\end{theorem}

\begin{remark}
As in Remark~\ref{remark:distinct-alex-polys}, if we only have that $\Delta_{J_n} \neq \Delta_{J_m}$ up to a unit in $\mathbb{F}_2[t^{\pm 1}]$, then we may still deduce that $(B^4,\S^{r,s}_n)$ and $(B^4,\S^{r,s}_m)$ are not smoothly isotopic rel.\ boundary. 
\end{remark}

\begin{proof}
The curve $\alpha$ bounds a smooth disk into the complement of the component of $\S^{r,s}_0$ on which it lies.
Additionally, $\S^{r,s}_0$ is Brunnian: the two components are trivial. We need to see this is true after the rim surgery. The rim surgery only changes one of the components, the one containing $\alpha$, so we only have to check that component is still trivial. Since $\alpha$ bounds a smooth disc in the complement of that component, Lemma 4.6 shows that component is still trivial.

Moreover, since $\alpha$ bounds a locally flat disk into the complement of $\S^{r,s}_0$, it follows also from Lemma~\ref{lemma:when-rim-surgery-doesnt-change-anything} that $\S^{r,s}_n$ is topologically isotopic rel.\ boundary to $\S^{r,s}_0$ for any $n$.

By Proposition~\ref{prop:genussqp}, $\S^{r,s}_0\cup b$ is isotopic to a strongly quasipositive Seifert surface pushed into $B^4$. It follows from Remark~\ref{rem:usefulprop}~\eqref{prop2} and Lemma~\ref{lem:effectofrim} that $\Omega(\S^{r,s}_n\cup b)=n$. As in Theorem~\ref{thm:intro-surface-links}, we note that if $\S^{r,s}_n$ and $\S^{r,s}_m$ are smoothly equivalent, then $\S^{r,s}_n\cup b$ is smoothly equivalent to $\S^{r,s}_m\cup b'$ for some band $b'$, so $\Omega(\S^{r,s}_m\cup b')=n$. But $(\S^{r,s}_m\cup b')$ is obtained from 1-twist rim surgery using knot $J_m$ on $(\S^{r,s}_0\cup b')$, so by Lemma~\ref{lem:effectofrim} we have
\[\Omega(\S^{r,s}_m\cup b')=-\infty\quad \text{or}\quad \Omega(\S^{r,s}_m\cup b')\ge m.\]
This yields a contradiction, so we conclude that $(B^4, \S^{r,s}_n)$ is not diffeomorphic to $(B^4, \S^{r,s}_m)$.
\end{proof}

\section{Surfaces with more than two components}\label{section:covering-maps}

In this section, we construct Brunnian exotic surface pairs with arbitrarily many components. To prove Theorem~\ref{thm:intro-any-no-components}, we take a pair of surface links from either Theorem~\ref{thm:intro-disk-links} or Theorem~\ref{thm:intro-surface-links}, and apply iterated Bing doubling to the first component, yielding exotic pairs.

This proof uses covering space methods, and was inspired by work of Cha-Kim \cite{Cha-Kim} on concordance of 1-links obtained from Bing doubling. We extend their covering link methods to the 4--dimensional setting.
The main obstacle is to show that there is no diffeomorphism of pairs relating Bing doubled surfaces, provided we start with a pair of surface links that admit no diffeomorphism of pairs.

Section~\ref{sec:covering} introduces the notion of covering surfaces, and investigates covering surfaces of Bing doubles.
Section~\ref{subsection-JSJ-trees} recalls some of the theory of JSJ decompositions of 3-manifolds, and applies it to Bing doubles and their covering links. Section~\ref{subsection:lifting-diffeomorphisms} uses this to show that, in a fairly general setting, Bing doubling a pair of exotic surface links yields another pair of exotic surface links.

\subsection{Coverings and Bing doubles}\label{sec:covering}

\begin{lemma}\label{lemma:covering-surface-link-1}
  Let $\Sigma = \Sigma_0 \sqcup \S_1 \sqcup \cdots \sqcup \S_n$ be a surface link in $B^4$, and suppose that $\S_0$ is an unknotted disk. Fix $k \geq 1$.  Suppose that $H_1(\Sigma\sm \Sigma_0;\Z) \to H_1(B^4 \sm \Sigma_0;\Z) \cong \Z$ has image in $k\Z$.  Let $p \colon B^4 \to B^4$ be a $k$-fold branched covering map with branching set $\S_0$.
  Then $p^{-1}(\S_1 \sqcup \cdots \sqcup \S_n)$ is a $kn$--component surface link.
\end{lemma}

\begin{proof}
  The homological condition guarantees that each component of $\S_1 \sqcup \cdots \sqcup \S_n$ lifts to a $k$ component surface link.
\end{proof}

\begin{definition}\label{defn:covering-surface-link}
  Let $\Sigma$ be a surface link. A \emph{covering surface link} of $\Sigma$ is a surface link obtained from $\S$ by a finite sequence of the following two operations.
  \begin{enumerate}
    \item Taking the pre-image of $\Sigma$ under a branched covering map $p$, as in Lemma~\ref{lemma:covering-surface-link-1}, with branching set an unknotted disk, and forgetting the branching set.
    \item Passing to a sublink.
  \end{enumerate}
\end{definition}

The following lemma is key to our use of covering surfaces in arguments by contradiction. It states that there is a covering surface of a Bing double isotopic to the original surface link.
Recall that $\BD(\S)$ is defined to be $\BD(\S_1)\sqcup\S_2\sqcup\cdots\sqcup\Sigma_n$. The ``first" component of $\BD(\S)$, denoted $BD(\Sigma)_1$, is one of the two components of $\BD(\S_1)$. By symmetry of the Bing doubling operation, these two components can be arbitrarily ordered.



\begin{lemma}\label{lemma:covering-link-of-BD-is-original-link}
Let $\Sigma = \Sigma_1 \sqcup \cdots \sqcup \Sigma_n$ be an $n$--component surface link, where $n \geq 2$. If $\Sigma_1$ is a trivial disk, then $\Sigma$ can be realized as a covering surface of $BD(\Sigma)$.

\end{lemma}

\begin{proof}
We begin by setting some notation. Fix a neighborhood $N \cong D^2 \! \times \! D^2$  of $\Sigma_1$ where $\Sigma_1$ is identified with $D^2 \! \times \! \{0\}$. Denote $\BD(\Sigma)$ by $\Sigma'$; its components are labeled so that $\Sigma_1'$ and $\Sigma_2'$ form $\BD(\Sigma_1)$ and so that  $\Sigma_i'=\Sigma_{i-1}$ for $i \geq 3$. Note that $\Sigma_1'$ and $\Sigma_2'$ lie inside $N$, whereas $\Sigma_i'$ lies in $B^4 \setminus N$ for $i \geq 3$.  Let $U$ denote the solid torus $(\partial D^2) \! \times \! D^2 = (\partial N) \cap S^3$; this solid torus is identified with  $ \nu(\partial \Sigma_1)$ in $S^3$ and contains $\partial \Sigma_1'$ and $\partial \Sigma_2'$.


The disk $\Sigma_2'$ consists of two push-offs of $\Sigma_1$ joined by a band in $U$. By Lemma~\ref{lem:view}, we may isotope $\Sigma_1'$ so that it is the push-in of a Seifert disk $D$ for the unknot $\partial \Sigma_1'$ as in Figure~\ref{fig:local}. Note that $D$ and $\Sigma_2'$ meet in a single ribbon intersection which occurs in $U$, as shown.


 \begin{figure}[h] \center
  \def\svgwidth{.35\linewidth}
\begingroup%
  \makeatletter%
  \providecommand\color[2][]{%
    \errmessage{(Inkscape) Color is used for the text in Inkscape, but the package 'color.sty' is not loaded}%
    \renewcommand\color[2][]{}%
  }%
  \providecommand\transparent[1]{%
    \errmessage{(Inkscape) Transparency is used (non-zero) for the text in Inkscape, but the package 'transparent.sty' is not loaded}%
    \renewcommand\transparent[1]{}%
  }%
  \providecommand\rotatebox[2]{#2}%
  \newcommand*\fsize{\dimexpr\f@size pt\relax}%
  \newcommand*\lineheight[1]{\fontsize{\fsize}{#1\fsize}\selectfont}%
  \ifx\svgwidth\undefined%
    \setlength{\unitlength}{573.07953121bp}%
    \ifx\svgscale\undefined%
      \relax%
    \else%
      \setlength{\unitlength}{\unitlength * \real{\svgscale}}%
    \fi%
  \else%
    \setlength{\unitlength}{\svgwidth}%
  \fi%
  \global\let\svgwidth\undefined%
  \global\let\svgscale\undefined%
  \makeatother%
  \begin{picture}(1,0.65580043)%
    \lineheight{1}%
    \setlength\tabcolsep{0pt}%
    \put(0,0){\includegraphics[width=\unitlength,page=1]{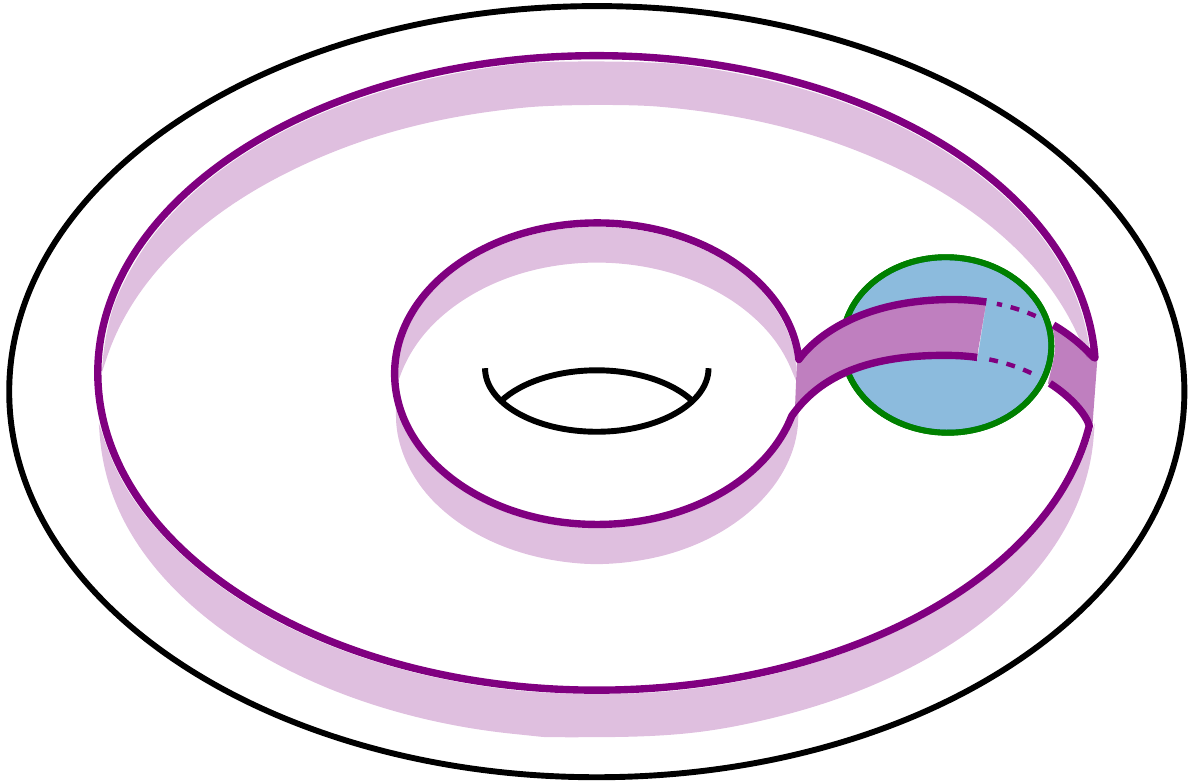}}%
    \put(0.14898436,0.38106428){\color[rgb]{0.50196078,0,0.50196078}\makebox(0,0)[lt]{\lineheight{1.25}\smash{\begin{tabular}[t]{l}$\Sigma_2'$\end{tabular}}}}%
    \put(0.75446046,0.23806304){\color[rgb]{0,0.27843137,0.70196078}\makebox(0,0)[lt]{\lineheight{1.25}\smash{\begin{tabular}[t]{l}$D$\end{tabular}}}}%
  \end{picture}%
\endgroup%

\caption{The solid torus $U = \partial N \cap S^3$.  The disk $\Sigma_1'$ is the push-in of the Seifert disk $D \subseteq U$ shown in the figure.  The disk $\Sigma_2'$ is formed from two push-offs of $\Sigma_1$ (represented by the light purple collars, which extend to disks out of view) joined by a band in $U$.}\label{fig:local}
\end{figure}

Since $\Sigma_1'$ is an unknotted disk in $B^4$, the double cover of $B^4$ branched along $\Sigma_1'$ is $B^4$. However, to lift the remaining surface link components, we describe the branched cover more precisely. As shown in the proof of Lemma~\ref{lem:view}, the disks $\Sigma_1'$ and $D$ cobound a 3-ball $\Delta$ lying in $N$. Let $Q$ denote the modified 4-ball (with corners) obtained by cutting $B^4$ open along $\Delta$. The double branched cover of $\Sigma_1'$ is obtained by gluing together two copies of this space, $Q$ and $\alt{Q}$, as depicted schematically in Figure~\ref{fig:schematic}, along the part of the boundary $\Delta_- \cup \Delta_+$ obtained from cutting along $\Delta$.
Fix a diffeomorphism 
\[Q \cup_{\Delta_- \cup \Delta_+} \alt{Q} \cong B^4.\]
Note that the 3-ball $\Delta$ is disjoint from the surface link components $\Sigma_i'$ for $i \geq 3$. 
The inclusion $Q \to Q \cup_{\Delta_- \cup \Delta_+} \alt{Q}$ therefore induces an embedding of  $\bigcup_{i \geq 3} \Sigma_i'$  into $B^4$. 

%
%
As we are only aiming to produce a covering surface for $\Sigma'$, and not reconstruct the entire lift, we discard the copies of the components $\Sigma_i'$ with $i \geq 3$ induced by the embedding  $\alt{Q} \to Q \cup_{\Delta_- \cup \Delta_+} \alt{Q}$, keeping only those in $Q$ (this is the only meaningful distinction between $Q$ and $\alt{Q}$).

 \begin{figure} \center
  \def\svgwidth{.9\linewidth}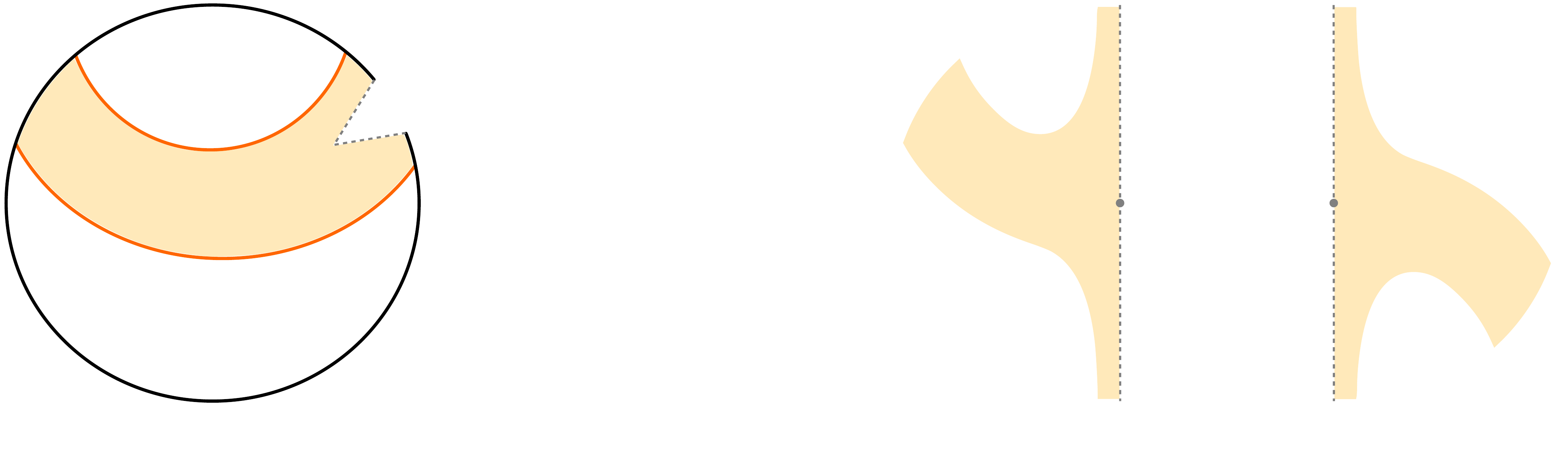
\caption{A schematic illustrating how to produce $Q$ and the double cover of $B^4$ branched over $\Sigma_1'$.
  {Left:} Obtaining $Q$ by slicing open $B^4$ along the 3-ball $\Delta$; the push-offs $\Delta_\pm$ of $\Delta$ lie in the boundary of $Q$.   {Right:} Gluing two copies $Q$ and $\alt{Q}$ to produce the double cover of $B^4$ branched over $\Sigma_1'$. The lifts of the surface components $\Sigma_3',\ldots,\Sigma_{n+1}'$ in $\alt{Q}$ are dashed because we discard them when constructing the covering surface link.}\label{fig:schematic}
\end{figure}

We still have to understand the lifts of the remaining component $\Sigma'_2$. To this end, observe that cutting along $\Delta$ separates $\Sigma'_2$  into two copies of the original disk $\Sigma_1$. The copies of these disks in $Q$ are glued to the copies in $\alt{Q}$ along the arcs where they meet the 3-ball $\Delta$; see the first row of Figure~\ref{fig:trivialize}. In $\alt{Q}$, where we have discarded the lifts of $\Sigma_i'$ for $i \geq 3$, we may isotope these copies of the trivial disk $\Sigma_1$ back into the boundary of $\alt{Q}$ while preserving the arcs where they are glued to the corresponding disks from $Q$; this is illustrated in the second row of Figure~\ref{fig:trivialize}. We may then pull the portions of the disks lying in $\alt{Q}$ back through the gluing region so that the lifts of $\Sigma_2'$ lie entirely in $Q$, as in the final row of Figure~\ref{fig:trivialize}. After this isotopy, these disks become parallel copies of $\Sigma_1$. Discarding one of these copies of $\Sigma_1$ and identifying $Q \cup \alt{Q}$ with $B^4$ as above, we see that the resulting covering link is smoothly equivalent to $\Sigma$, as desired.
\begin{figure} \center
\def\svgwidth{.8\linewidth}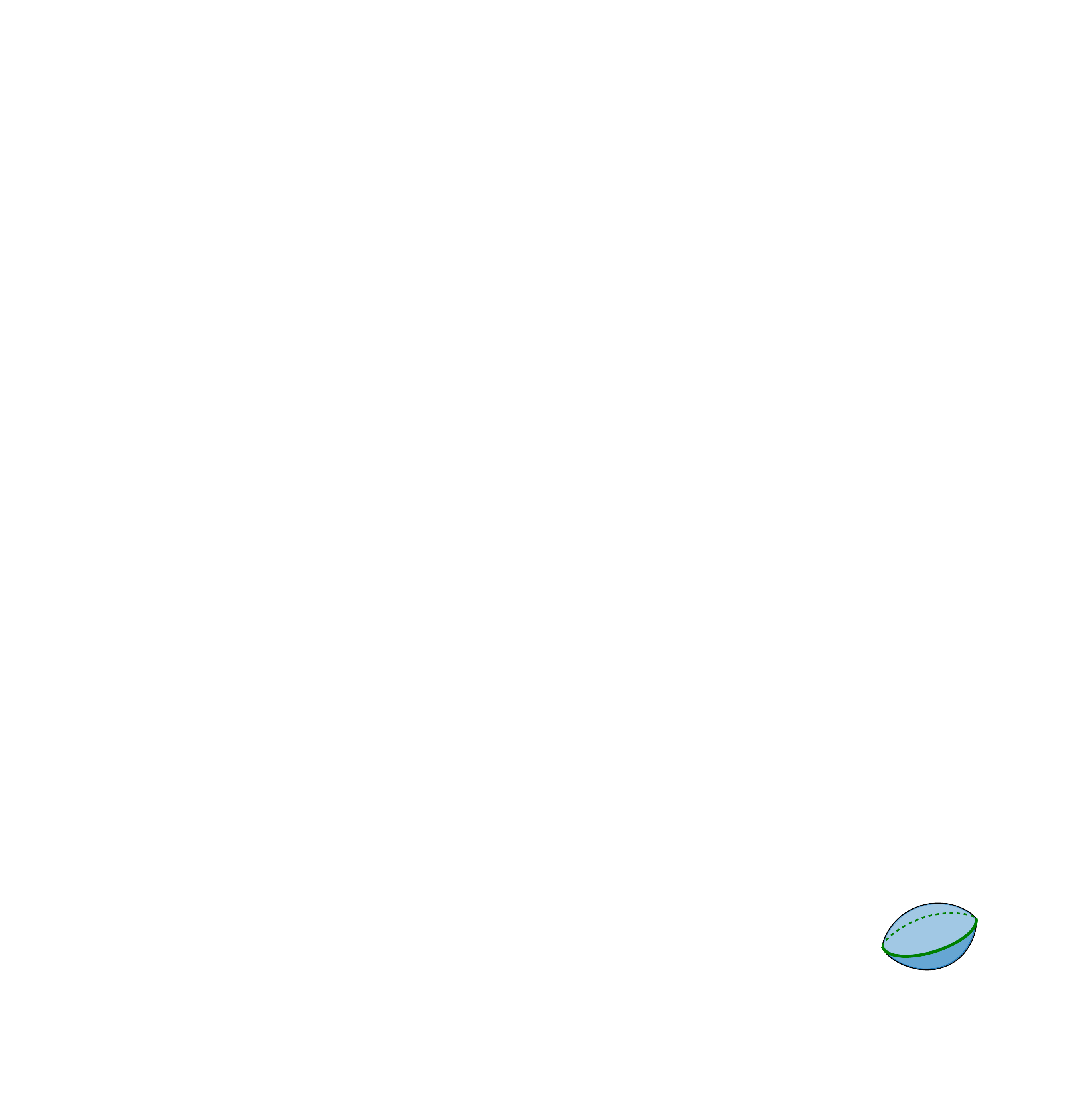 
\caption{The left and right sides depict $\partial Q \cap U$ and $\partial \alt{Q} \cap \alt{U}$, respectively. (Note that $\partial Q$ also contains the 3-ball $\Delta_+ \cup \Delta_-$, which is out of view because it lies on the other side of the 2-sphere $D_+ \cup D_-$; 
an analogous statement holds in $\partial \alt{Q}$.) 
The  torus is de-emphasized in $\partial \alt{Q} \cap \alt{U}$ because the lifts of $\Sigma'_3, \ldots, \Sigma'_{n+1}$ in $\alt{Q}$ have been deleted.  
  {Top row:} We see the lifts of $\Sigma_2'$ in $\partial Q$ and $\partial \alt{Q}$. 
  {Middle row:} After deleting the lifts of the components $\Sigma_3', \ldots, \Sigma_{n+1}'$ from $\alt{Q}$, the copies of $\Sigma_1$ in $\alt{Q}$ can be isotoped into the boundary. This isotopy requires that the copies of $\Sigma_1$ in $\alt{Q}$ temporarily exit the neighborhood $\alt{N}$ before they can be pulled into the boundary, so they are shown intersecting the dotted torus (which is the corner of $\partial \alt{N}$).   {Bottom row:} After pulling the portions of $\Sigma_1$ in $\alt{Q}$ through the gluing region, the lifts of $\Sigma_2'$ lie entirely in $Q$. }\label{fig:trivialize}
\end{figure}
\end{proof}

Our newfound understanding of coverings of Bing doubles allows us to study other knots and surfaces in the complement of $\S$.

\begin{corollary}\label{cor:diskinbing}
Let $\S$ be a surface link with first component a trivial disk. Let $J$ be a knot in $S^3\setminus \nu(\partial \S)$. Suppose $J$ bounds a smooth disk $D$ into the complement of $\BD(\S)$. Then $J$ bounds a smooth disk into the complement of $\S$.
\end{corollary}

\begin{proof}
Let $\S_D:=\BD(\S)\cup D$. As usual, call the components of $\BD(\S)$ by the names $\S_1,\ldots,\S_{n+1}$, where $\S_1,\S_2$ are the components obtained by taking the Bing double of the first component of $\Sigma$. Consider the preimage of $\S_D\setminus\Sigma_1$ in the 2-fold cyclic cover of $B^4$ branched along $\S_1$, as in Lemma~\ref{lemma:covering-link-of-BD-is-original-link}.

For $i=2,\ldots, n+1$ let $\Sigma_i^1$ and $\Sigma_i^2$ denote the two lifts of $\S_i$, and let $D^1$ and $D^2$ denote the two lifts of $D$. By Lemma~\ref{lemma:covering-link-of-BD-is-original-link}, we know that (after perhaps exchanging the names of $\Sigma_i^1$ and $\Sigma_i^2$ for some values of $i$) the surface link $\widetilde{\S}:=\sqcup_{i=2}^{n+1}\Sigma_i^1$ is equivalent to $\S$. Moreover, since $\partial\S_1$ is an unknot split from $J$, the link $\partial\widetilde{\S}\cup \partial D^1$ is equivalent to $\partial \S\cup J$ (perhaps after switching the roles of $D^1$ and $D^2$). Then $D^1$ is a slice disk for $J$ into the complement of $\S$, as desired.
\end{proof}

\subsection{JSJ trees}\label{subsection-JSJ-trees}

We recall some 3-manifold theory that will be used in the forthcoming arguments.

Every compact, irreducible 3-manifold $X$ with toroidal boundary admits a JSJ decomposition into finitely many codimension zero submanifolds, each of which are again compact, irreducible 3-manifolds with toroidal boundary. Each of these submanifolds is either Seifert fibered or has the property that every incompressible torus is boundary parallel \cite{jaco1979seifert, johannson1979homotopy}. A JSJ decomposition is obtained by cutting the 3-manifold $X$ along incompressible, non-boundary-parallel tori. A minimal collection of these tori are called \emph{JSJ tori}.  The closures of the connected components of the complement of the JSJ tori are called \emph{JSJ pieces}. The JSJ tori are unique up to isotopy, and are preserved (although possibly permuted) by any diffeomorphism from $X$ to itself. 

The JSJ decomposition of a 3-manifold naturally yields a graph with a vertex corresponding to each JSJ piece and an edge between two vertices if and only if the corresponding pieces share a JSJ torus. When the 3-manifold is the complement of a link in $S^3$, which is the only case considered here, this graph is a forest of trees \cite[Section~4]{JSJBudney}. The number of connected components of the graph is given by the number of split components of the link.

The next lemma is entirely 3-dimensional, but we state it in terms of surface links in $B^4$ in order to use our established Bing doubling terminology.
Note that the exterior of the Bing double pattern in a solid torus is diffeomorphic to the exterior of the Borromean rings (see Figure~\ref{fig:BorRings}), so this 3-manifold will appear frequently in our JSJ decomposition arguments.

\begin{figure}\center
\includegraphics[scale=1]{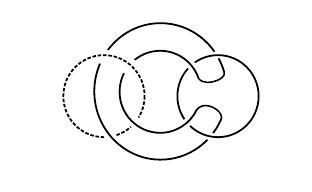}
\caption{The exterior of the Borromean rings in $S^3$ can be identified with the exterior of the Bing double of the unknot in the solid torus $S^3\setminus \nu(U)$, where $U$ denotes the dashed component.}
\label{fig:BorRings}
\end{figure}

\begin{lemma}\label{lemma:JSJ-trees}
  Let $\Sigma$ be a surface link in $B^4$. Let $T$ be the tree associated with the JSJ decomposition of $X:=S^3 \sm \nu (\partial \S)$. Let $Y \subseteq X$ be the JSJ piece that contains $\partial \S_1$ and let $v_Y$ be the vertex of $T$ corresponding to $Y$. Assume that $\partial \S_1$ is not a split unknotted component of $\partial \S$.  The JSJ tree of $X_{\BD} := S^3 \sm \nu (\partial \BD(\S))$ is obtained from $T$ by first adding a new vertex $v_E$ corresponding to a JSJ piece diffeomorphic to the exterior, $E$, of the Borromean rings, and then adding an edge connecting $v_E$ and $v_Y$.
\end{lemma}

\begin{figure}
\labellist
 \pinlabel $v_Y$ at 190 650
 \pinlabel $v_Y$ at 480 650
 \pinlabel $v_E$ at 530 650
\endlabellist
\includegraphics[width=80mm]{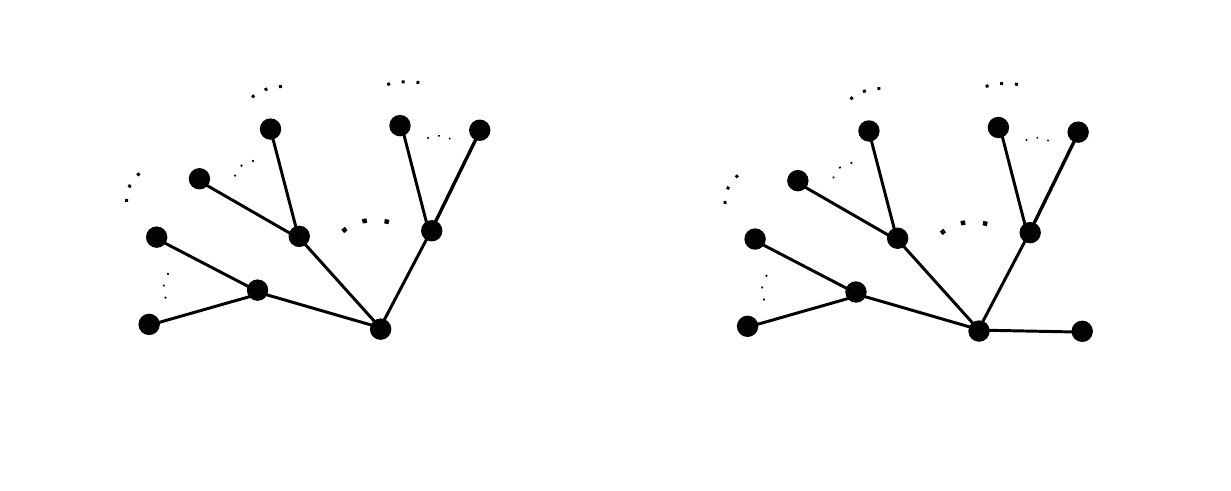}
\caption{JSJ trees for $S^3 \sm \nu (\partial \S)$ and $S^3 \sm \nu (\partial \BD(\S))$.} \label{fig:tree-plus-vE}
\end{figure}

\begin{proof}
Let $K:=\partial \S_1$ and let $F$ be the torus $\partial \ol{\nu} K$.  Note that $X_{\BD} \cong X \cup_F E$, where the zero-framed longitude of one Borromean component is glued to the meridian of $K$ on $F$ and the meridian of the Borromean component becomes the longitude of $F$.
 We claim that $F$ is a JSJ torus of $X_{\BD}$.
  Since $K$ is not a split unknotted component, the longitude of $F$ is nontrivial in~$\pi_1(X)$.  In $X_{\BD}$, the exterior of the satellite link, the longitude of $F$ is therefore again nontrivial. The meridian of $F$ is nontrivial in $X$.  The surface $F$ is also incompressible in $E$. Therefore $F$ is a JSJ torus of $X_{\BD}$ as desired. To obtain the JSJ tree for $X_{\BD}$, we start with $T$, the JSJ tree for $X$. Then, as shown in Figure~\ref{fig:tree-plus-vE}, we add a new vertex corresponding to $E$, and an edge connecting it to $v_Y$, since $Y$ is precisely the component that shares a JSJ torus with $E$. This gives the asserted JSJ tree.
\end{proof}

By repeatedly applying Lemma~\ref{lemma:JSJ-trees}, we see that iterated Bing doubling adds a line to the JSJ tree of the link exterior, with as many edges as applications of $\BD$. This feature of the tree allows us to restrict the possible automorphisms of the underlying 3-manifold, using the following result, which is a direct consequence of the uniqueness of JSJ decompositions.

\begin{theorem}\label{thm:key-JSJ-fact}
A diffeomorphism of a link exterior to itself induces an automorphism of the associated JSJ tree. Moreover each vertex, corresponding to a JSJ piece $Y$ say, must be sent to another vertex corresponding to a JSJ piece $Y'$ such that $Y \cong Y'$. \hfill $\Box$
\end{theorem}

The next lemma refers to the \emph{3-keychain link}: the three-component link consisting of an unknot and two copies of its meridian (see Figure~\ref{fig:3keychain}). 
The exterior of this link is diffeomorphic to the cartesian product of a circle $S^1$ with a connected planar surface with three boundary components (i.e.\ a pair of pants).

\begin{figure}
\includegraphics[width=50mm]{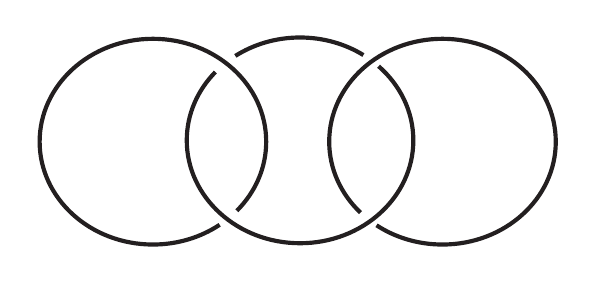}
\caption{The 3-keychain link.}\label{fig:3keychain}
\end{figure}

\begin{lemma}\label{lemma:JSJ-tree-of-covering-link}
Let $L$ be an $n$--component link in $S^3$ with first component $L_1$ unknotted. Let $\wt{L}$ be the covering link obtained by taking the 2-fold branched cover with branching set the first component of $\BD(L)$, forgetting the branching set, and forgetting one of the lifts of $(\BD(L))_2$ $($i.e.\ the other component of $\BD(L_1))$. The JSJ tree of the exterior of $\wt{L}$ is obtained as follows. Let $T_1$ and $T_2$ be two copies of the JSJ tree for $L$, and let $v_1$ and $v_2$ be the vertices corresponding to the JSJ piece containing the boundary of a tubular neighborhood of $L_1$, in each copy.  Take $T_1 \sqcup T_2$, add a new vertex $v_C$, and add two new edges, one joining $v_C$ to $v_1$ and one joining $v_C$ to $v_1$. The JSJ piece corresponding to $v_C$ is the 3-keychain link exterior.
\end{lemma}

\begin{figure}
\labellist
 \pinlabel $T_1$ at 160 760
 \pinlabel $T_2$ at 440 760
 \pinlabel \small{$v_1$} at 255 680
 \pinlabel \small{$v_C$} at 300 680
 \pinlabel \small{$v_2$} at 350 680
\endlabellist
\includegraphics[width=3.5in]{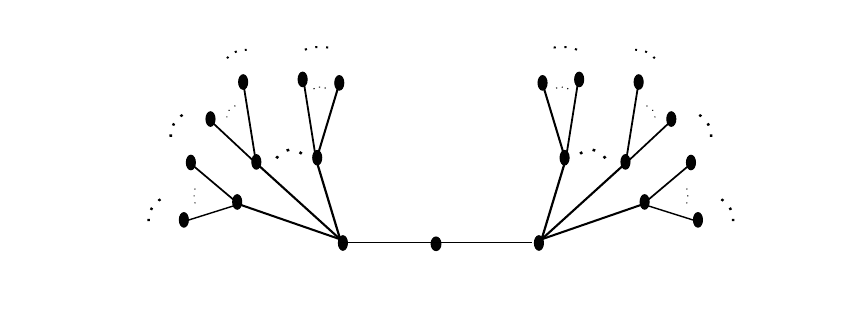}
\caption{JSJ tree for the exterior of the covering link $\wt{L}$ from a JSJ tree for the exterior of $L$.} \label{fig:covering-tree}
\end{figure}


  \begin{proof}
Observe that  $\wt{L}$ is isotopic to the result of taking two split copies of $L$ in disjoint copies of $B^3$, and performing a trivial band sum on the two copies of $L_1$. This follows from the argument in the proof of Lemma~\ref{lemma:covering-link-of-BD-is-original-link}, and restricting to $S^3$; see Figure~\ref{fig:trivialize}.
Therefore, the exterior of $\wt{L}$ has JSJ decomposition consisting of two copies of the JSJ decomposition for $L$ joined by a central $3$-keychain link exterior $C$.
It is a standard fact that, in the JSJ decomposition of exteriors of connected sums, keychain link exteriors appear in between the JSJ pieces of the summands (see e.g. \cite{schubert1953knoten}).  The conclusion for the resulting JSJ tree is immediate.
  \end{proof}

\begin{figure}
\labellist
 \pinlabel $T_1$ at 25 560
 \pinlabel $T_2$ at 565 560
 \pinlabel \small{$v_1$} at 255 480
 \pinlabel \small{$v_C$} at 295 480
 \pinlabel \small{$v_2$} at 335 480
\endlabellist
\includegraphics[width=5in]{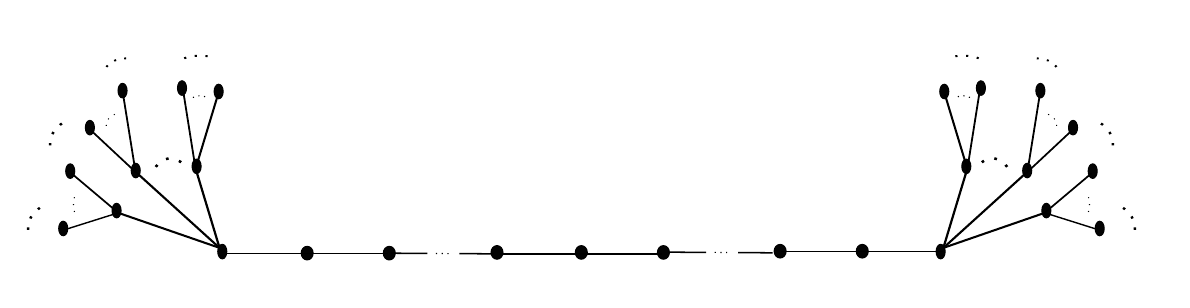}
\caption{JSJ tree for the exterior of a covering link of an iterated Bing double.} \label{fig:long-stick-tree}
\end{figure}

\subsection{Bing doubling exotic surface link pairs}\label{subsection:lifting-diffeomorphisms}

We now arrive at the main goal of this section. By combining our understanding of JSJ decompositions with our work on covering surfaces, as in Section~\ref{sec:covering}, 
we show that if two surfaces are smoothly inequivalent, then under some fairly weak assumptions, the results of Bing doubling one disk component of each are also smoothly inequivalent.

Combining the work in Sections~\ref{sec:covering} and~\ref{subsection-JSJ-trees} yields the following proposition, which shows that the disk links from Section~\ref{section:exotic-pair-disk-links} are not just non-equivalent, but in fact, any of their iterated Bing doubles have non-diffeomorphic complements.

\begin{proposition}\label{prop:bingcomp}
Let $\S_1 = \BD(D_1)$ and $\S_2  = \BD(D_2)$ be the disk links from Proposition~\ref{prop:base-case}. For any $n\ge 0$, $\BD^n(\S_1)$ and $\BD^n(\S_2)$ have non-diffeomorphic complements.
\end{proposition}

\begin{proof}Recall that $\S_1$ and $\S_2$ are each a link of two disks in $B^4$ whose components are individually unknotted disks, $\partial \S_1=\partial \S_2$, and $\S_1$ and $\S_2$ are topologically isotopic rel.\ boundary.

In Proposition~\ref{prop:base-case}, we showed that $\S_1$ and $\S_2$ are not smoothly equivalent by studying a knot~$\gamma$ in $S^3\setminus \nu(\partial \S_i)$ which is preserved up to isotopy by any automorphism of $(S^3,\partial \S_i)$. The knot~$\gamma$ bounds a smooth disk into the complement of $\S_2$ but not into $\S_1$, implying that there is no diffeomorphism between $(B^4,\S_1)$ and $(B^4,\S_2)$.

Since $\gamma$ bounds a smooth disk in the complement of $\S_2$, $\gamma$ bounds a smooth disk (e.g.\ the same disk) into the complement of $\BD^n(\S_2)$ for any $n\ge 0$. Since $\gamma$ does not bound a smooth disk into the complement of $\S_1$, it follows from Corollary~\ref{cor:diskinbing} that $\gamma$ does not bound a smooth disk into the complement of $\BD^n(\S_1)$ for any $n\ge 0$. Therefore, if there is a diffeomorphism from the complement of $\BD^n(\S_1)$ to the complement of $\BD^n(\S_2)$, it must not preserve the isotopy class of~$\gamma$. We claim this is not possible.

In Figure~\ref{fig:bingjsj}, we show a JSJ decomposition for $S^3\setminus \nu(\BD^n(\S_i))$. Recall that $\partial \S_i$ is a Bing double of a hyperbolic knot $K$ (cf.\ Lemma~\ref{lem:hyperbolic}), so $S^3\setminus \nu(\partial \S_i)$ has a JSJ decomposition of one essential torus cutting the manifold into two pieces: the complement of $K$ and a Borromean ring complement. Then $S^3\setminus \nu(\BD^n(\S_i))$ admits a JSJ decomposition with $n+2$ pieces in a line, as illustrated in Figure~\ref{fig:bingjsj}, with essential tori $T_1,\ldots, T_{n+1}$ in between. The first piece (bounded by $T_1$) is the complement of $K$, and every piece thereafter is a Borromean ring complement.

 The boundary $M^3$ of $B^4\setminus\nu(\BD^n(\S_i))$ is obtained from $S^3\setminus \nu(\partial \BD^n(\S_i))$ by performing 0-framed Dehn filling along every torus boundary component. Note that this Dehn filling on the Borromean rings exterior yields the 3-torus~$T^3$. In this closed 3-manifold $M^3$, $T_{n+1}$ bounds a copy of $(T^2\setminus\mathring{D}^2)\times S^1$, $T_1$ bounds a copy of $S^3\setminus \nu(K)$ and $T_i\sqcup T_{i+1}$ bounds $$T^3\setminus \nu(\text{two independent primitive curves})$$ for $i=1,\ldots, n$. This is again a JSJ decomposition for $M^3$. Since the complement of $K$ is not homeomorphic to any other piece in this JSJ decomposition, we conclude that any automorphism of $M^3$ preserves $T_1$ up to isotopy and induces an automorphism of $S^3\setminus \nu(K)$. Since $S^3\setminus \nu(K)$ has trivial automorphism group (Lemma~\ref{lem:hyperbolic}), we conclude that $\gamma$ is preserved up to isotopy by any automorphism of $M^3$. Therefore, there is no diffeomorphism from the complement of $\BD^n(\S_1)$ to the complement of $\BD^n(\S_2)$, as desired. 
\end{proof}

\begin{figure}
\labellist
  \pinlabel \footnotesize{$n+1$ vertices, each representing} at 102 -5
  \pinlabel \footnotesize{a Borromean ring complement} at 102 -12
  \pinlabel \footnotesize{$S^3\setminus \nu(K)$} at 208 8
  \pinlabel \footnotesize{\textcolor{red}{$T_1$}} at 191 23
  \pinlabel \footnotesize{\textcolor{red}{$T_2$}} at 156 23
  \pinlabel \footnotesize{\textcolor{red}{$T_{n+1}$}} at 53 23
\endlabellist
\includegraphics[width=90mm]{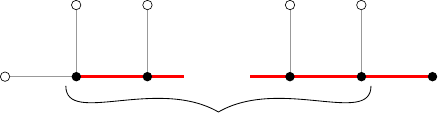}
\vspace{.2in}
\caption{A JSJ decomposition of $S^3\setminus \nu(\partial L_i)$ as in Proposition~\ref{prop:bingcomp}. The bold red edges represent essential tori. The black vertices are the components of the JSJ decomposition. The white vertices represent torus boundary components; grey edges between white and black vertices show to which JSJ component the corresponding boundary belongs.}\label{fig:bingjsj}
\end{figure}

Proposition~\ref{prop:bingcomp} allows us to immediately conclude that for every $n\geq 0$, $\BD^n(\S_1)$ and $\BD^n(\S_2)$ is an exotic pair of $(n+1)$--component Brunnian links. Next, we prove our main technical result on covering surfaces; it has the advantage  that it is applicable to a more general family of surface links, e.g.\ all of the families in Section~\ref{section:infinitely-many-surface-links} and those in Theorem~\ref{thm:increasing-the-genus-main-body}.

\begin{proposition}\label{prop:exotic-implies-exotic-BD}
   Let $\Sigma = \Sigma_1 \sqcup \Sigma_2$ and $\Sigma' = \Sigma_1' \sqcup \Sigma_2'$ be $2$--component surface links in $B^4$ with $\Sigma_1$ and $\Sigma_1'$ trivial disks. Suppose that $\partial \S = \partial \S' \subseteq S^3$, and let $X$ be the link exterior in $S^3$. Suppose that either:
    \begin{enumerate}
    \item[(i)]\label{item:JSJ-condition-i}  $X$ has no Borromean rings exterior and no 3-keychain link exterior in its JSJ decomposition; or
    \item[(ii)]\label{item:JSJ-condition-ii}  $X$ has no 3-keychain link exterior and exactly one Borromean rings exterior $E$ in its JSJ decomposition, but two of the three boundary components of $E$ are $\partial X$.
    \end{enumerate}
 Then, if the iterated Bing doubles $\BD^n(\Sigma)$ and $\BD^n(\Sigma')$ are smoothly equivalent $($Bing doubling the first component at each iteration$)$, then so are $\Sigma$ and $\Sigma'$.
\end{proposition}

\begin{proof}
Let $S:= \BD^n(\Sigma)$ and $S' := \BD^n(\Sigma')$.  The first components $S_1$ and $S'_1$ are both unknotted disks, since they arose from Bing doubling.
Consider the 2-fold branched coverings  $p \colon B^4 \to B^4$ and $p' \colon B^4 \to B^4$ with branching sets the first components $S_1$ and $S'_1$ of each surface link, respectively. Since first components are unknotted, the branched covering spaces are indeed again~$B^4$.

By Lemma~\ref{lemma:covering-link-of-BD-is-original-link}, $\BD^{n-1}(\Sigma)$ is a covering surface link of $S= \BD^n(\S)$. Moreover, $\BD^{n-1}(\Sigma)$ is a sublink of the inverse image $p^{-1}(S\sm S_1)$. Analogously, $\BD^{n-1}(\Sigma')$ is a sublink of the inverse image $p^{-1}(S'\sm S'_1)$.

Let $F \colon (B^4,S) \to (B^4,S')$ be a diffeomorphism of pairs. We claim that it must send $S_1$ to either $S'_1$ or $S'_2$.

To prove the claim, we consider the JSJ decomposition and tree for $S^3 \sm \nu(\partial S)$. There is a unique univalent vertex in the JSJ tree for $S^3 \sm \nu(\partial S)$ that corresponds to a copy of the exterior~$E$ of the Borromean rings. This vertex must be sent to the analogous vertex in the JSJ tree for $S^3 \sm \nu(\partial S')$ by Theorem~\ref{thm:key-JSJ-fact}.  The two boundary components of $S^3 \sm \nu(\partial S')$ that lie in this copy of $E$ are $S'_1$ and $S'_2$. It follows that $F(S_1)$ is $S'_1$ or $S'_2$.  After potentially relabelling, we can and shall assume, without loss of generality, that the diffeomorphism $F$ in fact sends $S_1$ to $S'_1$.



Since the diffeomorphism $F \colon (B^4,S) \to (B^4,S')$ preserves the branching set (i.e.\ $F(S_1) = S'_1$), we may lift $F$ to a diffeomorphism of pairs
\[\wt{F} \colon (B^4,p^{-1}(\BD^n(\S))) \to (B^4,p^{-1}(\BD^n(\S'))).\]
In the inverse image, $p^{-1}(\BD^n(\S))$, we see two copies of each component of $\BD^{n-1}(\S)$.
Restrict the diffeomorphism $\wt{F}$ to \[\BD^{n-1}(\S) \subseteq p^{-1}(S \sm S_1) \subseteq p^{-1}(S) = p^{-1}(\BD^n(\S)).\]
The image is a sublink of $p^{-1}(S' \sm S'_1)$.
We claim that the image is a copy of $\BD^{n-1}(\S')$.  Assuming this claim, we are done: by induction, we see that $\S$ and $\S'$ are smoothly equivalent, as desired.

To prove the claim, we again argue by considering the effect of the diffeomorphism restricted to the link exteriors in $S^3$
\[\wt{F}| \colon S^3 \sm \nu (\partial (p^{-1}(S \sm S_1))) \to  S^3 \sm \nu (\partial (p^{-1}(S'\sm S'_1))). \]
Since the links on the boundary coincide, so do their covering links.
We will restrict how the diffeomorphism can permute components using the relative rigidity of $3$-manifold diffeomorphisms, in particular by studying the JSJ decomposition and applying Theorem~\ref{thm:key-JSJ-fact} again.

Lemma~\ref{lemma:JSJ-tree-of-covering-link} provides a description of the JSJ tree of the covering link obtained by forgetting the branching set and one of the lifts of $\partial S_2$ and $\partial S_2'$.
As above, we choose a lift of $S_2$, and then the lift of $S_2'$ is determined by the diffeomorphism $F$. The automorphisms of the JSJ tree of this covering link are given by composing automorphisms of the two copies of the JSJ tree for the exterior of $\partial S$ with either:
\begin{enumerate}[(i)]
  \item the identity, or
  \item a flip map that switches the two sub-trees that are copies of the JSJ tree for the exterior of $\partial S$.
\end{enumerate}
To see this, we consider the structure of the JSJ trees.  The JSJ tree $T(n-1)$ 
of the exterior of $\partial \BD^{n-1}(\S)$ is obtained by iterating the procedure pictured in Figure~\ref{fig:tree-plus-vE} and consists of the JSJ tree for $X$ with a tail of $(n-1)$ Borromean exteriors attached to the vertex corresponding to the JSJ piece containing $\partial \S_1$. The JSJ tree for the covering link is pictured in Figure~\ref{fig:long-stick-tree} and is obtained from two copies of  $T(n-1)$ by adding one more vertex corresponding to the 3-keychain link exterior, with two edges connecting this new vertex to the ends of both tails, as described in Lemma~\ref{lemma:JSJ-tree-of-covering-link} and shown in Figure~\ref{fig:covering-tree}.
By our assumptions on the JSJ decomposition of $X$, the Borromean and 3-keychain link exteriors cannot be replicated anywhere in the tree but at its centre, and the tails must get either fixed or swapped with each other by $\wt{F}$.  The link components which have boundaries of tubular neighborhoods that lie in JSJ pieces corresponding to one half of the tree form a copy of $\BD^{n-1}(\S)$ in the domain and form a copy of $\BD^{n-1}(\S')$ in the codomain. Determining how $\wt{F}$ acts on connected components of our surface links is enough to control how $\wt{F}$ acts on their boundary components. Therefore, we have completed the proof of the claim that the image $\wt{F} (\BD^{n-1}(\S))$ is a copy of $\BD^{n-1}(\S')$. This proves the inductive step and therefore completes the proof of the proposition.
\end{proof}

\noindent Finally, we prove the following, which is a more specific version of the statement of Theorem~\ref{thm:intro-any-no-components} in Section~\ref{section:Introduction}. The version stated in Section~\ref{section:Introduction} immediately follows.

\begin{intro-any-no-components}
Let $\Sigma$ and $\Sigma'$ be an exotic pair of Brunnian 2--component surface links constructed as in Theorem~\ref{thm:intro-disk-links} or Theorem~\ref{thm:intro-surface-links}, with orderings such that the first components are disks. For all $n\ge 3$, the $n$--component surface links $\BD^{n-2}(\Sigma)$ and $\BD^{n-2}(\Sigma')$ are an 
exotic pair of Brunnian surface links.
\end{intro-any-no-components}



The disk links in Theorem~\ref{thm:intro-disk-links} arose by taking two slice disks $D_1$ and $D_2$ in $B^4$ for a fixed knot $K$ in $S^3$, and taking the Bing doubles $\S := \BD(D_1)$ and $\S' := \BD(D_2)$. This is an exotic pair of Brunnian disk links. Theorem~\ref{thm:intro-any-no-components} states that iterated Bing doubling both of these links yields an exotic pair with any specified number of components.

The surface disk links in  Theorem~\ref{thm:intro-surface-links} were constructed by taking a particular 2--component surface link $\S$ consisting of a disk and a genus one surface, and applying 1-twisted rim surgery (using a knot $J$) along a curve $\alpha$ that bounds a locally flat disk in the exterior of $\S$ but does not bound any smoothly embedded disk in the exterior.  By varying $J$, we obtain exotic surface links.  Theorem~\ref{thm:intro-any-no-components} shows that applying Bing doubling to the disk component of $\S$ produces $n$--component surface links such that the same rim surgeries give rise to an infinite family of pairwise exotic surface links. \\

\noindent \textit{Proof of Theorem~\ref{thm:intro-any-no-components}.} By Lemma~\ref{lem:doubling-preserves-brunnian}, the surface links $\BD^{n-2}(\S)$ and $\BD^{n-2}(\S')$ are still Brunnian. Moreover, Lemma~\ref{lem:doubling-preserves-isotopy} (with $\CAT=\TOP$) ensures they are still topologically isotopic (we may assume the doubling is done such that the boundaries are the same link in $S^3$).

If the assumptions of Proposition~\ref{prop:exotic-implies-exotic-BD} hold, then $\BD^{n-2}(\Sigma)$ and $\BD^{n-2}(\Sigma')$ being smoothly equivalent implies that $\Sigma$ and $\Sigma'$ are too, which is a contradiction.
Therefore we need to check that the hypotheses of the proposition hold.
For the disk links of Theorem~\ref{thm:intro-disk-links}, the JSJ decomposition of the exterior of $\partial \S = \partial \S'$ consists of a Borromean exterior and the exterior of a hyperbolic knot (Lemma~\ref{lem:hyperbolic}), and so Proposition~\eqref{item:JSJ-condition-ii}(ii) is satisfied, and therefore the proposition applies.

For the surface links of Theorem~\ref{thm:intro-surface-links}, we note that their boundary is a satellite of the knot $K$, as shown on the left of  Figure~\ref{fig:infinitefamilyJSJ}. In this case, the JSJ decomposition of $\partial S=\partial S'$ consists of the exterior of the hyperbolic (verified by SnapPy \cite{snappy}) 3--component link $L$ illustrated on the right of  Figure~\ref{fig:infinitefamilyJSJ} along with a JSJ decomposition of the exterior of $K$ in $S^3$. Since $L$ has pairwise zero linking numbers, the exterior of $L$ is not the exterior of a 3-keychain link. Moreover, as verified by SnapPy \cite{snappy}, the exterior of $L$ has symmetry group $\mathbb{Z}/2$, so is also not a Borromean rings exterior. Taking $K$ to be (for example) the positive untwisted Whitehead double of the trefoil, whose exterior admits a JSJ decomoposition consisting of a trefoil exterior and a Whitehead link exterior, we find that Proposition~\eqref{item:JSJ-condition-i}(i) is satisfied and the proposition applies. This completes the proof; see Appendix~\ref{appendixb} for further documentation of this use of SnapPy. \hfill $\Box$

\begin{figure}
\includegraphics[width=90mm]{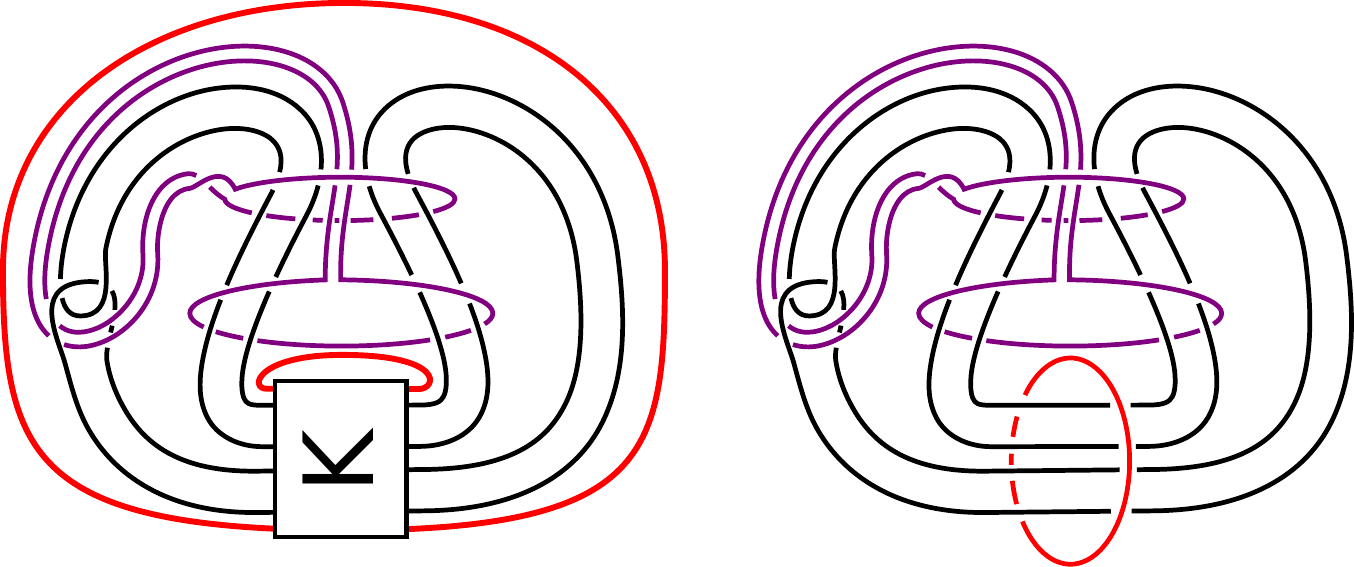}
\caption{  {Left:} an essential torus in the exterior of $\partial \S$, where $\S$ is the surface link from Theorem~\ref{thm:intro-surface-links}. This torus splits the exterior of $\partial\S$ into the exterior of $K$ glued to the exterior of a hyperbolic 3--component link $L$.   {Right:} the link $L$.} \label{fig:infinitefamilyJSJ}
\end{figure}

\section{Links of other Brunnian types}\label{section:nkBrunnian}

In this section, we apply our techniques to construct exotic surfaces of any Brunnian type, in the language of Debrunner \cite{Debrunner}.

\begin{definition}
Let $\Sigma$ be an $n$--component surface link in $B^4$.  Fix an integer $k$ with $1\le k\le n$.
The surface link $\Sigma$ is {\emph{$(n,k)$-Brunnian}} if every sublink of $\Sigma$ with fewer than $k$ components is an unlink of surfaces, while every sublink of $\Sigma$ with at least $k$ components is nontrivial.
\end{definition}

We adapt the Bing doubling construction of Brunnian links (in the usual sense) to produce $(n,k)$-Brunnian links.

\begin{procedure}\label{nkprocedure}{\emph{Constructing an $(n,k)$-Brunnian surface link.}} Let $D$ be a disk in $B^4$ whose boundary is a nontrivial knot in $S^3$. To simplify the exposition, we let $n_k$ denote $\binom{n}{k}$. Let $\Sigma_{n,k}(D)$ be the $n$--component surface link constructed as follows (see also Figures~\ref{fig:nkbrunnian1} and~\ref{fig:nkbrunnian2}).
\begin{enumerate}
\item For $i=1,2,\ldots,{n_k}$, let $\Sigma_i$ be a copy of $\BD^{k-1}(D)$. Arrange $\Sigma_1,\ldots,\Sigma_{n_k}$ from left to right, as ${n_k}$ split copies of $\BD^{k-1}(D)$.
\item Let $\mathcal{C}=\{C_1,\ldots, C_n\}$ indicate a set of $n$ distinct colors. For each $i$, color every component of $\Sigma_i$ a color in $\mathcal{C}$, so that no two components of $\Sigma_i$ are the same color and so that for $i\neq j$, the surface links $\Sigma_i$ and $\Sigma_j$ are not colored with the same size-$k$ subset of $\mathcal{C}$. See Figure~\ref{fig:nkbrunnian1} for an example.
\item Let $B_1,\ldots, B_{n_k}$ be disjoint 4-balls containing $\Sigma_1,\ldots, \Sigma_{n_k}$, respectively. We now band all components of $\Sigma_1\sqcup\cdots\sqcup \Sigma_{n_k}$ of each single color together to form an $n$--component link of disks, such that the bands are chosen to be trivial near their ends in the $B_i$'s and disjoint from other $B_j$'s.

Said differently: for each $i$ and $j$: when $\Sigma_i$ includes a color-$C_j$ component and there is a color-$C_j$ component in $\Sigma_s$ for some $s>i$, let $i_j=\min\{s>i\mid \Sigma_s$ has a color-$C_j$ component$\}$ and band the color-$C_j$ components in $\Sigma_i$ and $\Sigma_{i_j}$ together by a band that is trivial in $B_i$ and $B_{i_j}$ and does not intersect any other $B_l$. See Figure~\ref{fig:nkbrunnian2}.
\end{enumerate}

\begin{figure}
\labellist
\pinlabel {$\Sigma_1$} at 25 -8
\pinlabel {$\Sigma_2$} at 80 -8
\pinlabel {$\Sigma_{n_k}$} at 247 -8
\endlabellist
\includegraphics[width=120mm]{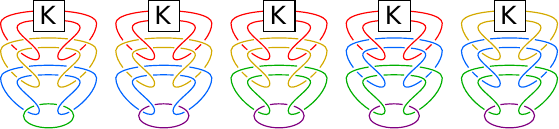}
\vspace{.15in}
\caption{To build $\Sigma_{n,k}(D)$, we start with a split union of $n_k$ copies $\Sigma_1,\ldots, \Sigma_{n_k}$ of $\BD^{k-1}(D)$. Here we draw the boundaries of $\Sigma_1\sqcup\cdots\sqcup \Sigma_{n_k}$, where $K=\partial D$. In addition, we color every component of each $\Sigma_i$ with a color in $\mathcal{C}$, such that (1) no two components of $\Sigma_i$ may be the same color, and (2) for $i\neq j$, the links $\Sigma_i$ and $\Sigma_j$ may not use the same subset of colors in $\mathcal{C}$. In this figure, $(n,k)=(5,4)$.}
\label{fig:nkbrunnian1}
\end{figure}

\begin{figure}
\labellist
\pinlabel {$B_1$} at 15 82
\pinlabel {$B_2$} at 70 82
\pinlabel {$B_{n_k}$} at 237 82
\endlabellist
\includegraphics[width=120mm]{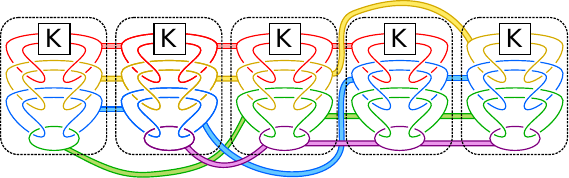}
\caption{Starting from Figure~\ref{fig:nkbrunnian1}, to complete our construction of $\Sigma_{n,k}(D)$ we band together all $C_i$--colored components of $\Sigma_1\sqcup\cdots\sqcup\Sigma_{n_k}$ for each $C_i\in\mathcal{C}$. The bands are taken to be trivial in the 4-balls $B_1,\ldots, B_{n_k}$ but otherwise may be chosen freely so long as they are disjoint. The result is an $n$--component link of disks that we call $\Sigma_{n,k}(D)$.}
\label{fig:nkbrunnian2}
\end{figure}
\end{procedure}

When restricted to $S^3$, this construction differs from Debrunner's \cite{Debrunner}; in particular, we elect to use Bing doubles to construct our generalized Brunnian links.
Notice that if $n=k$, $n_n =1$, hence $\Sigma_{n,n}(D) = \BD^{n-1}(D)$.  When $k=1$, $\Sigma_{n,1}(D)$ is a split union of $n$ copies of $D$. 

\newpage
\begin{proposition}\label{prop:nkbrunnian}
The surface link $\Sigma:=\Sigma_{n,k}(D)$ from Procedure~\ref{nkprocedure} is $(n,k)$-Brunnian.
\end{proposition}

\begin{proof}
The proof of Proposition~\ref{prop:nkbrunnian} naturally divides into two claims.


\begin{claim} \label{claim1}
\textit{Let $\Sigma'$ be a $(k-1)$--component sublink of $\Sigma$. Then $\Sigma'$ is an unlink of disks.}
\end{claim}

\noindent \textit{Proof of Claim~\ref{claim1}.}
By construction, each component of $\S$ is colored by a different element of a size-$n$ set $\mathcal{C}$. Suppose $\Sigma'$ is obtained from $\Sigma$ by deleting the components of colors $\mathcal{C}_r\subseteq\mathcal{C}$, where $|\mathcal{C}_r|=n-k+1$. Since $\mathcal{C}$ has $n$ colors and each $\Sigma_i$ has $k$ distinct colors, we conclude that $\Sigma'$ is obtained from $\Sigma=\cup_{\text{bands}}\Sigma_i$ by deleting a nonzero number of components from each $\Sigma_i$.

Note that \[\Sigma'=\bigcup_{\text{along bands}}^{i\ge 1}\Sigma_i \setminus (\text{components of colors } \mathcal{C}_r).\]

In Figure~\ref{fig:nkbrunnian3} (left), we depict the ball neighborhood $B_1$ of $\Sigma_1$ (the leftmost copy of $\BD^{k-1}(D)$ in the construction of $\Sigma$). The ball $B_1$ meets the bands in $\Sigma'$, which are attached to $\Sigma_1$, trivially: all bands extend from $\Sigma_1$ towards the right. The surface $B_1\cap\S'$ is obtained from $B_1\cap \S_1$ by deleting at least one component of $\S_1$ (and any attached bands). Thus, in $B_1$, we may isotope $B_1\cap \Sigma'$ rel.\ $(\partial \overline{B_1})\cap\mathring{B}^4$ to obtain trivial disks, as indicated in Figure~\ref{fig:nkbrunnian3}. (Note that while we only draw the boundary of $\Sigma'$ in Figure~\ref{fig:nkbrunnian3}, once any component of $\Sigma'\cap B_1$ is deleted the remaining components are determined up to smooth isotopy rel.\ boundary as trivial disks in $B^4$.) We conclude \[\Sigma'\approx \bigcup_{\text{along bands}}^{i\ge 2}\Sigma_i \setminus (\text{components of colors } \mathcal{C}_r),\] where we use $``\approx"$ to denote smooth isotopy throughout.

Now proceed by induction, assuming 
\[\Sigma'\approx\bigg(\bigcup_{\text{along bands}}^{i\ge s}\Sigma_i \setminus (\text{components of colors } \mathcal{C}_r)\bigg)\cup (\text{an unlink of disks})\] 
for some $s>1$ (up to smooth isotopy). In Figure~\ref{fig:nkbrunnian3} (right), we depict the ball neighborhood $B_s$ of $\Sigma_s$. The ball $B_s$ meets bands in $\Sigma'$ trivially attached to $\Sigma_s$, with all bands extending from $\Sigma_s$ toward the right. Note that some components of $\Sigma_s$ may not have any bands extending to the right -- in the construction of $\Sigma$, this means that such a component is the last of its color. The surface $\S'\cap B_s$ is obtained from $\S_s\cup$(bands extending to the right$)$ by deleting at least one component (and any attached bands). Then in $B_s$, we may isotope $B_s\cap \Sigma'$ rel.\ $(\partial \overline{B}_s)\cap\mathring{B}^4$ to obtain trivial disks as indicated in Figure~\ref{fig:nkbrunnian3}. Note that some disks might be split in the interior of $B_s$. We conclude that
\[\Sigma'\approx \bigg(\bigcup_{\text{along bands}}^{i\ge s+1}\Sigma_i \setminus (\text{components of colors } \mathcal{C}_r)\bigg)\cup(\text{an unlink of disks}).\]
Inductively, we obtain
\begin{align*}
\Sigma'&\approx \bigg(\Sigma_{n_k}\setminus(\text{components of colors } \mathcal{C}_r)\bigg)\cup (\text{an unlink of disks})\\
&\approx \text{an unlink of disks},
\end{align*}
since $\Sigma_{n_k}$ is Brunnian and includes at least one color in $\mathcal{C}_r$. This completes the proof of the claim.

\begin{claim} \label{claim2}
\textit{Let $\Sigma''$ be a $k$--component sublink of $\Sigma$. Then $\Sigma''$ is nontrivial and in fact, $\Sigma''$ is smoothly isotopic to $\BD^{k-1}(D)$.}
\end{claim}

\noindent \emph{Proof of Claim~\ref{claim2}.}
First, note that $\partial D$ is a nontrivial knot, so $\partial \BD^{k-1}(D)$ is not an unlink, and $\BD^{k-1}(D)$ is not an unlinked disk link.

The proof of Claim~\ref{claim2} is similar to that of Claim~\ref{claim1}. Suppose that $\Sigma''$ is obtained from $\Sigma$ by deleting the components of colors $\mathcal{C}_r\subseteq\mathcal{C}$, where $|\mathcal{C}_r|=n-k$. Since $\mathcal{C}$ has $n$ colors and each $\Sigma_i$ has $k$ distinct colors, we conclude that $\Sigma''$ is obtained from $\Sigma=\cup_{\text{bands}}\Sigma_i$ by deleting a nonzero number of components from each $\Sigma_i$ for all $i\neq m$, where $m$ is the unique integer such that $\Sigma_m$ is colored with all of the colors in $\mathcal{C}-\mathcal{C}_r$. Again, we write \[\Sigma''=\bigcup_{\text{along bands}}^{i\ge 1}\Sigma_i\setminus (\text{components of colors } \mathcal{C}_r).\]

Applying the proof of the previous claim applied to $\Sigma_1$ and working from left to right, \[\Sigma''\approx\bigg(\bigcup_{\text{along bands}}^{i\ge m}\Sigma_i \setminus(\text{components of colors } \mathcal{C}_r)\bigg)\cup(\text{an unlink of disks}).\] 
By horizontally mirroring the argument and applying it to $\Sigma_{n_k}$ and working from right to left, we find
\begin{align*}\Sigma''&\approx\Big(\Sigma_m \setminus (\text{components of colors } \mathcal{C}_r)\Big)\cup(\text{an unlink of disks})\\
&\approx \Sigma_m\cup\text{an unlink of disks}\\
&\approx \Sigma_m \approx \BD^{k-1}(D),
\end{align*}
since $\Sigma''$ and $\Sigma_m$ each have $k$ components. Since $\BD^{k-1}(D)$ is a nontrivial disk link, this completes the proof of the claim.\\

This completes the proof of the proposition.
\end{proof}

\begin{figure}
\labellist
\pinlabel {\LARGE$\sim$} at 60 35
\pinlabel {\LARGE$\sim$} at 213 35
\endlabellist
\includegraphics[width=122mm]{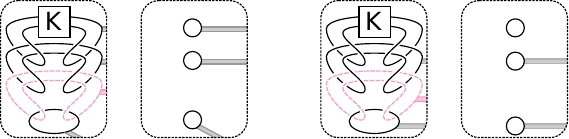}
\caption{Left: The ball $B_1$ intersecting $\Sigma'$. If we delete at least one component, e.g.\ the one highlighted, then we obtain a trivial link of disks with trivial bands extending out of $B_1$. Right: The ball $B_s$ intersecting $\Sigma'$ after applying an inductive hypothesis in the proof of Proposition~\ref{prop:nkbrunnian}. If we delete at least one component, e.g.\ the one highlighted, then we obtain a trivial link of disks with trivial bands extending from some components out of $B_s$.}
\label{fig:nkbrunnian3}
\end{figure}

\begin{nkbrunnian}
For any integers $n$ and $k$ with $n \geq 2$ and $1\le k\le n$, there exists a pair of $(n,k)$--Brunnian disk links in $B^4$ forming an exotic pair.
\end{nkbrunnian}

\begin{proof}
Let $D_1$ and $D_2$ be the exotic pair of disks of Proposition~\ref{prop:initial-disks}. Recall that $D_1$ and $D_2$ are topologically isotopic rel.\ boundary. We construct  $\Sigma:=\Sigma_{n,k}(D_1)$ and $\Sigma':=\Sigma_{n,k}(D_2)$ as in Procedure~\ref{nkprocedure}, such that all choices of colors and bands agree in $S^3$, and $\partial \Sigma=\partial \Sigma'$. By Proposition~\ref{prop:nkbrunnian}, $\Sigma$ and $\Sigma'$ are $(n,k)$-Brunnian.

The disks $D_1$ and $D_2$ are topologically isotopic rel.\ boundary; applying Lemma~\ref{lem:doubling-preserves-isotopy}, so are $\BD^{k-1}(D_1)$ and $\BD^{k-1}(D_2)$. As $\S$ and $\S'$ are obtained by banding together copies of $\BD^{k-1}(D_1)$ and $\BD^{k-1}(D_2)$ with the same choices of bands in $S^3$, we conclude that $\S$ and $\S'$ are topologically isotopic rel.\ boundary.

Suppose $f\colon (B^4,\S)\to (B^4,\S')$ is a diffeomorphism, and let $\S_k$ be a $k$--component sublink of $\S$. Thus, $f(\S_k)$ is a $k$--component sublink of $\S'$. By the proof of Claim~\ref{claim2} within Proposition~\ref{prop:nkbrunnian}, we know that $\S_k$ is smoothly isotopic to $\BD^{k-1}(D_1)$ and $f(\S_k)$ is smoothly isotopic to $\BD^{k-1}(D_2)$. This contradicts Theorem~\ref{thm:intro-any-no-components}; we conclude there is no diffeomorphism from $(B^4,\S)$ to $(B^4,\S')$.
\end{proof}

\section{Closed surfaces}\label{section:closed-surfaces}

We show that in some cases, our exotic Brunnian surface links may be promoted to exotic links of closed surfaces in $4$-manifolds with positive second Betti number.

\begin{theorem}\label{thm:trace}
Let $\S_1$ and $\S_2$ be the 2--component exotic disk links in $B^4$ from Proposition~\ref{prop:base-case}. Fix a non-negative integer $n$. Let $X^4$ be the 4-manifold obtained from $B^4$ by attaching a 0-framed 2-handle along every boundary component of $\BD^n(\S_i)$. Let $S_i\subseteq X$ be the $(n+1)$--component link of 2-spheres obtained from $\BD^n(\S_i)$ by gluing a core of each 2-handle to the corresponding component of $\BD^n(\S_i)$. Then $S_1$ and $S_2$ are topologically isotopic in $X$, but $(X,S_1)$ is not diffeomorphic to $(X,S_2)$.
\end{theorem}
\begin{proof}
Since $S_1$ and $S_2$ agree outside $\BD^n(\S_1)$ and $\BD^n(\S_2)$, which are topologically isotopic rel.\ boundary, it is immediate that $S_1$ and $S_2$ are topologically isotopic.

Let $Y^4_i$ be the 4-manifold obtained from $X$ by surgering each component of $S_i$ (i.e.\ for each component $F$ of $\S_i$, delete $\nu(F)\cong S^2\times \mathring{D}^2$ and reglue $B^3\times S^1$ in order to obtain $Y_i$). We claim that $Y_i\cong B^4\setminus \nu(\BD^n(\S_i))$. This can be seen via Kirby calculus~\cite[Section~5.4]{GompfStipsicz}: attaching a 0-framed 2-handle along an unknot and then surgering along a 2-sphere consisting of a trivial disk in $B^4$ glued to a core of that 2-handle has the effect in a Kirby diagram of first adding a 0-framed circle (the 2-handle) and then changing the 0 to a dot. The latter represents the result of carving out a trivial disk from $B^4$ with the specified boundary.
To see that the 0-dot exchange on an unknot realizes the desired surgery, note that surgery on $S^2 \times \{0\} \subseteq S^2 \times D^2$ yields $S^2 \times S^1 \times [0,1] \cup_{S^1 \times S^1 \times \{1\}} B^3 \times S^1 \cong B^3 \times S^1$. 

If $(X,S_1)$ were diffeomorphic to $(X,S_2)$, then $Y_1$ would be diffeomorphic to $Y_2$. But Proposition~\ref{prop:bingcomp} says precisely that $Y_1$ is not diffeomorphic to $Y_2$, so we conclude that $(X,S_1)$ and $(X,S_2)$ are not diffeomorphic.
\end{proof}

In the 2--component case we can obtain the analogous result to Theorem~\ref{thm:trace} in a closed 4-manifold $Z$.

\begin{theorem}
The 2--component disk links $\S_1 =\BD(D_1)$ and $\S_2 =\BD(D_2)$ from Proposition~\ref{prop:base-case}  smoothly embed into 2--component sphere links $S_1$ and $S_2$, respectively, in a closed 4-manifold $Z$ such that
\begin{enumerate}[label=(\roman*)]
\item  $S_1$ and $S_2$ are topologically isotopic in $Z$, and
\item  $(Z,S_1)$ is not diffeomorphic to $(Z,S_2)$.
\end{enumerate}
\end{theorem}

\begin{proof}
Much as in the proof of Theorem~\ref{thm:trace}, we begin by attaching 0-framed 2-handles to $B^4$ along the two unknotted components of the link $\BD(K)$, where $K$ is the underlying slice knot from the proof of Proposition~\ref{prop:base-case}. Denote this 4-manifold by $X$. As in the proof of Theorem~\ref{thm:trace}, the disk links $\BD(D_1)$ and $\BD(D_2)$ give rise to a pair of exotic sphere links  $S_1$ and $S_2$ in $X$ such that surgering $S_i$ turns $X$ into the disk link exterior $B^4 \sm \nu(\BD(D_i))$.

Our strategy will be to embed $X$ into a closed 4-manifold $Z$ by constructing a well-chosen 4--dimensional ``cap'' $C$ with $\partial C = \partial X$ (but oriented so that $\partial C = -\partial X$). The induced sphere links $S_1$ and $S_2$ will remain topologically isotopic in the larger 4-manifold $Z$. To show that $S_1$ and $S_2$ also remain smoothly distinct, we will distinguish the 4-manifolds  $B^4 \sm \nu(\BD(D_1)) \cup C$ and $B^4 \sm \nu(\BD(D_2)) \cup C$ obtained by surgering $Z$ along $S_1$ and $S_2$, respectively. The cap $C$ will be constructed with this goal in mind.

To this end, recall from the proof of Proposition~\ref{prop:base-case} that we can attach three 2-handles to $B^4 \sm \nu(\BD(D_1))$ to produce a 4-manifold $W$ that admits a Stein structure; see Figure~\ref{fig:add-handles}. For later use, we note that one of these is a $-1$-framed 2-handle attached along the curve $\gamma$ from Figure~\ref{fig:BD-exteriors}. By \cite{lisca-matic:embed}, we can then further embed the Stein domain $W$ into a closed, minimal K\"ahler surface, which we denote by $Q$; here \emph{minimality} implies that $Q$ contains no smoothly embedded 2-spheres of self-intersection $-1$. We define the desired cap $C$ to be the exterior of $B^4 \sm \nu(\BD(D_1))$ in $Q$, and let $Z=X \cup C$ as discussed above.

It remains to distinguish the 4-manifolds obtained from $Z$ by surgering $S_1$ and $S_2$.  The resulting 4-manifolds are $B^4 \sm \nu(\BD(D_1)) \cup C$ (which is the minimal K\"{a}hler surface $Q$) and $B^4 \sm \nu(\BD(D_2)) \cup C$.  To distinguish them, we first recall that the curve $\gamma$ bounds a smoothly embedded disk in the complement of $\BD(D_2)$. After gluing $C$ to $B^4 \sm \nu(\BD(D_2))$, this disk can be glued to the core of the 2-handle attached along $\gamma$ to yield a smoothly embedded 2-sphere of self-intersection $-1$ in $B^4 \sm \nu(\BD(D_2)) \cup C$. It follows that surgering $Z$ along $S_1$ and $S_2$ yields distinct 4-manifolds, so the pairs $(Z,S_1)$ and $(Z,S_2)$ are not diffeomorphic.
\end{proof}

\bigskip

\appendix

\section{The isometry group of $K$}\label{appendixa}

We began by drawing $K$ in SnapPy's link editor, as depicted in Figure~\ref{fig:snappy}, and extracted a Dowker-Thistlethwaite code for this knot projection:
\begin{align*}
\text{DT: }
[(&-70,20,-84,-82,-80,42,-32,56,-66,2,-44,18,36,-52,40,76,-16,-64,28,-60,-78,
\\
& \quad 68,-22,-12,-30,62,-26,-74,14,50,-38,54,72,46,-4,-24,34,-58,48,-10,-8,-6)]
\end{align*}

\begin{figure}[h]\center
\includegraphics[width=.4\linewidth]{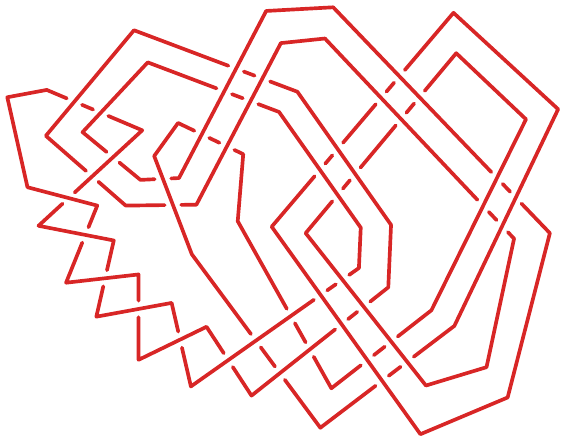}
\caption{}\label{fig:snappy}
\end{figure}

Running SnapPy inside Sage, we then verified that $K$ is a hyperbolic knot whose complement has trivial isometry group.



\bigskip

\noindent {\tt
sage: import snappy
}

\smallskip
\smallskip

\noindent{\tt
sage: K = snappy.Manifold(`DT:[(-70,20,-84,-82,-80,42,-32,56,-66,2,-44,18,36,}

{\tt \qquad \  -52,40,76,-16,-64,28,-60,-78,68,-22,-12,-30,62,-26,-74,14,50,-38,54,72,}

{\tt \qquad \ 46,-4,-24,34,-58,48,-10,-8,-6)]')
}

\smallskip
\smallskip

\noindent{\tt
sage: K.solution\_type()
\\
`all tetrahedra positively oriented'}

\smallskip
\smallskip

\noindent{\tt sage: K.verify\_hyperbolicity() 
\\
(True,
 [0.69085717467? + 0.50830991237?*I,
  -0.09695795674? + 0.91647294852?*I,\\
  0.97074783390? + 0.28150095915?*I,
  1.26374636309? + 0.55883319058?*I, \\
  0.13653313161? + 0.57669418470?*I,
  0.18676207068? + 0.42212622171?*I,\\
  1.71203281718? + 0.96878529641?*I,
  0.16077780284? + 2.5757092416?*I, \\
  -0.35397677376? + 0.58988805254?*I,
  -0.03195090969? + 1.66462568439?*I, \\
  -0.7299955531? + 1.9313492842?*I,
  0.16909543612? + 0.10116074904?*I, \\
  -1.1477637844? + 1.2694777149?*I,
  0.23479520751? + 0.83739365476?*I,\\
  -0.6690405348? + 0.68508313707?*I,
  1.6959300535? + 2.5175042223?*I,\\
  -0.9638799118? + 2.3142283467?*I,
  1.11640893998? + 0.55523026497?*I,\\
  0.7891261692? + 1.31673271484?*I,
  0.98323205817? + 1.45469490101?*I,\\
  -0.56189716658? + 0.44115682612?*I,
  -0.4883656525? + 2.1993281961?*I,\\
  0.19274256226? + 0.71697136107?*I])}

  \smallskip
  \smallskip

 \noindent {\tt
 sage: R = K.canonical\_retriangulation(verified = True)}

 \smallskip
 \smallskip

 \noindent {\tt
sage: len(R.isomorphisms\_to(R)) \qquad \#This gives the size of the isometry group.
\\
1}

\bigskip

\noindent The size of the isometry group is 1, so the identity is the unique isometry of $S^3 \setminus K$.

\section{The isometry group of $L$}
\label{appendixb}

As in Appendix~\ref{appendixa}, we began by drawing $L$ in SnapPy's link editor, as depicted in Figure~\ref{fig:linkL}, and extracted a Dowker-Thistlethwaite code for this knot projection:
\begin{align*}
\text{DT: }
[(&16,34,-64,54,40,68,-44,-60,32,50,42,-66,-38,-56,70,-48),
\\
&\quad (30,-58,8,-24,-12,20,52,-2,-62,-14,-36,-26,6,46,18),(-4,28,10,-22)]
\end{align*}

\begin{figure}[h]\center
\includegraphics[width=.3\linewidth]{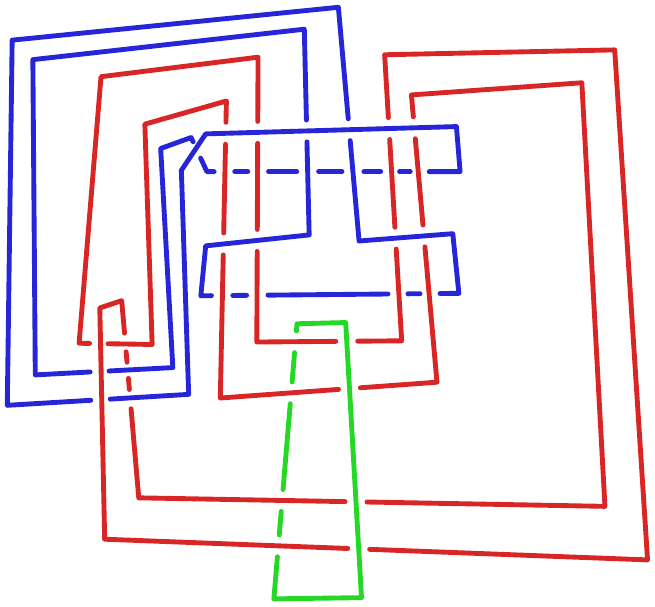}
\caption{}\label{fig:linkL}
\end{figure}

Running SnapPy inside Sage, we then verified that $L$ is a hyperbolic link whose complement has isometry group $\mathbb{Z}/2\mathbb{Z}$.

\bigskip

\noindent {\tt
sage: import snappy
}

\smallskip
\smallskip

\noindent{\tt
sage: L = snappy.Manifold(`DT:[(16,34,-64,54,40,68,-44,-60,32,50,42,-66,-38,-56,}

{\tt \qquad \ 70,-48), (30,-58,8,-24,-12,20,52,-2,-62,-14,-36,-26,6,46,18),(-4,28,10,}

{\tt \qquad \ -22)]')
}

\smallskip
\smallskip

\noindent{\tt
sage: L.solution\_type()
\\
`all tetrahedra positively oriented'}

\smallskip
\smallskip

\noindent{\tt sage: L.verify\_hyperbolicity() 
\\
(True,
  [0.63324006006? + 1.10322026006?*I,
  0.10993343042? + 0.47600086029?*I,\\
  0.16413577877? + 0.79246676123?*I,
  0.66490644413? + 1.23999412266?*I,\\
  0.38972085196? + 0.84475722622?*I,
  -0.02052870871? + 0.71226998786?*I,\\
  -1.09458686743? + 1.17956342451?*I,
  1.08258208868? + 0.80923399264?*I,\\
  0.07133786888? + 0.68448305451?*I,
  0.63642037823? + 0.39539813066?*I,\\
  0.31403980641? + 1.10759837099?*I,
  0.40489568556? + 0.46123907137?*I,\\
  0.16243243753? + 0.76871625189?*I,
  0.38972085196? + 0.84475722622?*I,\\
  0.66413540183? + 0.62635898843?*I,
  0.50355806795? + 0.72289062001?*I,\\
  0.50203604063? + 0.63579926026?*I,
  1.13580695569? + 1.15030349850?*I,\\
  0.07522244611? + 0.83274531566?*I,
  -0.09975044831? + 0.76471739458?*I,\\
  -0.38903974680? + 0.93113075790?*I,
  0.20490074235? + 1.28331759380?*I,\\
  1.63503890307? + 1.75198642689?*I,
  0.25525121024? + 0.99618331677?*I,\\
  0.38848526399? + 0.87511804068?*I,
  0.81713558269? + 0.50449818986?*I,\\
  0.33037527178? + 0.55709566053?*I,
  0.6350389031? + 1.75198642689?*I,\\
  0.33037527178? + 0.55709566053?*I])}

  \smallskip
  \smallskip

 \noindent {\tt
 sage: R = L.canonical\_retriangulation(verified = True)}

 \smallskip
 \smallskip

 \noindent {\tt
sage: len(R.isomorphisms\_to(R)) \qquad \#This gives the size of the isometry group.
\\
2}

\bigskip
 \rm

\noindent We conclude that $L$ is hyperbolic and not the Borromean rings.

\bibliographystyle{alpha}
\bibliography{biblio-exotic-surfaces}

\begin{thebibliography}{{The}19}

\bibitem[Bru92]{Brunn}
Hermann Brunn.
\newblock {\"U}ber verkettung.
\newblock {\em Bayer. Akad. Wiss. Math. Naturwiss. Abt.}, 22:77--99, 1892.

\bibitem[BS16]{baldwinsivekcontact}
John~A. Baldwin and Steven Sivek.
\newblock A contact invariant in sutured monopole homology.
\newblock {\em Forum Math. Sigma}, 4:Paper No. e12, 82, 2016.

\bibitem[Bud06]{JSJBudney}
Ryan Budney.
\newblock J{SJ}-decompositions of knot and link complements in {$S^3$}.
\newblock {\em Enseign. Math. (2)}, 52(3-4):319--359, 2006.

\bibitem[CDGW]{snappy}
Marc Culler, Nathan~M. Dunfield, Matthias Goerner, and Jeffrey~R. Weeks.
\newblock Snap{P}y, a computer program for studying the geometry and topology
  of $3$-manifolds.
\newblock \url{http://snappy.computop.org}.

\bibitem[CK08]{Cha-Kim}
Jae~Choon Cha and Taehee Kim.
\newblock Covering link calculus and iterated {B}ing doubles.
\newblock {\em Geom. Topol.}, 12(4):2173--2201, 2008.

\bibitem[CP21]{ConwayPowell}
Anthony Conway and Mark Powell.
\newblock Characterisation of homotopy ribbon discs.
\newblock {\em Adv. Math.}, 391:Paper No. 107960, 29, 2021.

\bibitem[Deb61]{Debrunner}
Hans Debrunner.
\newblock Links of {B}runnian type.
\newblock {\em Duke Math. J.}, 28:17--23, 1961.

\bibitem[EK71]{EdwardsKirby}
Robert~D. {Edwards} and Robion~C. {Kirby}.
\newblock {Deformations of spaces of imbeddings.}
\newblock {\em {Ann. Math. (2)}}, 93:63--88, 1971.

\bibitem[Eli90]{yasha:stein}
Ya. Eliashberg.
\newblock Topological characterization of {S}tein manifolds of dimension
  {$>2$}.
\newblock {\em Internat. J. Math.}, 1(1):29--46, 1990.

\bibitem[FKV88]{KreckNonorient}
S.~M. Finashin, M.~Kreck, and O.~Ya. Viro.
\newblock Nondiffeomorphic but homeomorphic knottings of surfaces in the
  {$4$}-sphere.
\newblock In {\em Topology and geometry---{R}ohlin {S}eminar}, volume 1346 of
  {\em Lecture Notes in Math.}, pages 157--198. Springer, Berlin, 1988.

\bibitem[FQ90]{FreedmanQuinn}
Michael~H. Freedman and Frank Quinn.
\newblock {\em Topology of 4-manifolds}, volume~39 of {\em Princeton
  Mathematical Series}.
\newblock Princeton University Press, Princeton, NJ, 1990.

\bibitem[FS97]{fintushelstern}
Ronald Fintushel and Ronald~J. Stern.
\newblock Surfaces in {$4$}-manifolds.
\newblock {\em Math. Res. Lett.}, 4(6):907--914, 1997.

\bibitem[GS99]{GompfStipsicz}
Robert~E. Gompf and Andr\'{a}s~I. Stipsicz.
\newblock {\em {$4$}-manifolds and {K}irby calculus}, volume~20 of {\em
  Graduate Studies in Mathematics}.
\newblock American Mathematical Society, Providence, RI, 1999.

\bibitem[Hat83]{Hatcher-smale-conj}
Allen~E. Hatcher.
\newblock A proof of the {S}male conjecture, {${\rm Diff}(S^{3})\simeq {\rm
  O}(4)$}.
\newblock {\em Ann. of Math. (2)}, 117(3):553--607, 1983.

\bibitem[Hay20]{Hayden}
Kyle Hayden.
\newblock Exotically knotted disks and complex curves.
\newblock {\em ArXiv 2003.13681}, 2020.

\bibitem[JMZ21]{JuhaszMillerZemke}
Andr\'{a}s Juh\'{a}sz, Maggie Miller, and Ian Zemke.
\newblock Transverse invariants and exotic surfaces in the 4-ball.
\newblock {\em Geom. Topol.}, 25(6):2963--3012, 2021.

\bibitem[Joh79]{johannson1979homotopy}
K.~Johannson.
\newblock Homotopy equivalences of 3-manifolds with boundaries.
\newblock {\em Lecture Notes in Math.}, 761, 1979.

\bibitem[JS79]{jaco1979seifert}
William Jaco and Peter~B. Shalen.
\newblock Seifert fibered spaces in 3-manifolds.
\newblock In {\em Geometric topology}, pages 91--99. Elsevier, 1979.

\bibitem[Juh06]{juhasz06}
Andr\'{a}s Juh\'{a}sz.
\newblock Holomorphic discs and sutured manifolds.
\newblock {\em Algebr. Geom. Topol.}, 6:1429--1457, 2006.

\bibitem[Juh16]{juhasz}
Andr\'{a}s Juh\'{a}sz.
\newblock Cobordisms of sutured manifolds and the functoriality of link {F}loer
  homology.
\newblock {\em Adv. Math.}, 299:940--1038, 2016.

\bibitem[JZ23]{juhaszzemkeconcordancesurgery}
Andr\'{a}s Juh\'{a}sz and Ian Zemke.
\newblock Concordance surgery and the {O}zsv\'{a}th-{S}zab\'{o} 4-manifold
  invariant.
\newblock {\em J. Eur. Math. Soc. (JEMS)}, 25(3):995--1044, 2023.

\bibitem[LM97]{lisca-matic:embed}
P.~Lisca and G.~Mati\'{c}.
\newblock Tight contact structures and {S}eiberg-{W}itten invariants.
\newblock {\em Invent. Math.}, 129(3):509--525, 1997.

\bibitem[LM98]{lisca-matic}
P.~Lisca and G.~Mati{\'c}.
\newblock Stein $4$-manifolds with boundary and contact structures.
\newblock {\em Topology Appl.}, 88:55--66, 1998.

\bibitem[Mos68]{Mostow}
G.~D. Mostow.
\newblock Quasi-conformal mappings in {$n$}-space and the rigidity of
  hyperbolic space forms.
\newblock {\em Inst. Hautes \'{E}tudes Sci. Publ. Math.}, 34:53--104, 1968.

\bibitem[OS04]{osknot}
Peter Ozsv\'{a}th and Zolt\'{a}n Szab\'{o}.
\newblock Holomorphic disks and knot invariants.
\newblock {\em Adv. Math.}, 186(1):58--116, 2004.

\bibitem[OS08]{oslink}
Peter Ozsv\'{a}th and Zolt\'{a}n Szab\'{o}.
\newblock Holomorphic disks, link invariants and the multi-variable {A}lexander
  polynomial.
\newblock {\em Algebr. Geom. Topol.}, 8(2):615--692, 2008.

\bibitem[Rud01]{rudolphsqp}
Lee Rudolph.
\newblock Quasipositive pretzels.
\newblock {\em Topology Appl.}, 115(1):115--123, 2001.

\bibitem[Sch53]{schubert1953knoten}
Horst Schubert.
\newblock Knoten und vollringe.
\newblock {\em Acta Mathematica}, 90(1):131--286, 1953.

\bibitem[Sma59]{Smale-diffeo-D2}
Stephen Smale.
\newblock Diffeomorphisms of the {$2$}-sphere.
\newblock {\em Proc. Amer. Math. Soc.}, 10:621--626, 1959.

\bibitem[Swe21]{klo}
F.~Swenton.
\newblock K{L}{O} ({K}not-{L}ike {O}bjects).
\newblock {\em Middlebury College}, pages http://klo--software.net, 2021.

\bibitem[{The}19]{sagemath}
{The Sage Developers}.
\newblock {S}agemath, the {S}age mathematics software system.
\newblock Available at \url{https://www.sagemath.org}, 2019.

\bibitem[Wal68]{Waldhausen}
Friedhelm Waldhausen.
\newblock On irreducible {$3$}-manifolds which are sufficiently large.
\newblock {\em Ann. of Math. (2)}, 87:56--88, 1968.

\bibitem[Zee65]{zeeman}
E.~C. Zeeman.
\newblock Twisting spun knots.
\newblock {\em Trans. Amer. Math. Soc.}, 115:471--495, 1965.

\end{thebibliography}











\renewcommand*\footnoterule{}
\newcommand\blfootnote[1]{%
  \begingroup
  \renewcommand\thefootnote{}\footnote{#1}%
  \addtocounter{footnote}{-1}%
  \endgroup
}
\blfootnote{\hspace{-0.4cm}\large\textsc{Epilogue}  \vspace{1em}

\normalsize\noindent\textit{Peculiar surfaces live in $B^4$,\\
delicately tangled sorts:\\
drab if untethered,\\
but exotic together.\\
Bing doubling yields them in scores!}}

\end{document}